\documentclass[12pt]{amsart}
\usepackage{amsmath}

\usepackage{amsfonts}
\usepackage{amssymb}
\usepackage{setspace}
\usepackage{datetime}
\usepackage{color,enumitem,graphicx}
\usepackage[colorlinks=true,urlcolor=blue,
citecolor=red,linkcolor=blue,linktocpage,pdfpagelabels,
bookmarksnumbered,bookmarksopen]{hyperref}
\usepackage{geometry}
\usepackage{titlesec}

\newcommand{\mylargefont}{\fontsize{12}{16}\selectfont}

\titleformat{\section}
    [block]
    {\normalfont\Large\bfseries\centering}
    {\bfseries\thesection}
    {1em}
    {\boldmath}

\titleformat{\subsection}{\normalfont\normalsize\bfseries}{\bfseries\thesubsection}{1em}{\boldmath}
\titleformat{\subsubsection}{\normalfont\mylargefont\bfseries}{\bfseries\thesubsubsection}{1em}{\boldmath}
\allowdisplaybreaks[4]

\newtheorem{thm}{Theorem}[section]

\newtheorem{lem}[thm]{Lemma}
\newtheorem{prop}[thm]{Proposition}
\theoremstyle{definition}
\newtheorem{defn}[thm]{Definition}
\newtheorem{rem}[thm]{Remark}
\theoremstyle{conclusion}

\theoremstyle{question}
\numberwithin{equation}{section}

\newcommand{\lr}{\left(}
\newcommand{\rr}{\right)}
\newcommand{\R}{\mathbb{R}}
\newcommand{\mms}{\mathbb{S}}
\newcommand{\Y}{\mathbb{Y}}

\newcommand{\N}{\mathbb{N}}
\newcommand{\md}{\mathrm{d}}
\newcommand{\Sm}{\mathbb{S}}
\newcommand{\lb}{\left\{}
\newcommand{\rb}{\right\}}

\newcommand{\be}{\begin{equation}}
	\newcommand{\ee}{\end{equation}}

\newcommand{\al}{\alpha}

\newcommand{\la}{\lambda}

\newcommand{\var}{\varepsilon}

\newcommand{\abs}[1]{\left\vert#1\right\vert}

\newcommand{\dd}{\,\mathrm d}

\begin{document}

\title[H\'enon-Hardy type Liouville equation in $\R^N$]{Non-radial solutions for the quasi-linear H\'enon type $N$-Laplacian Liouville equation}

\author{Wei Dai, Lixiu Duan, Changfeng Gui, Yuan Li}

\address{School of Mathematical Sciences, Beihang University (BUAA), Beijing 100191, P. R. China, and Key Laboratory of Mathematics, Informatics and Behavioral Semantics, Ministry of Education, Beijing 100191, P. R. China}
\email{weidai@buaa.edu.cn}

\address{School of Mathematics and Statistics, Shaanxi Normal University, Xi'an, Shaanxi, 710119, P. R. China, and School of Mathematical Sciences, Beihang University (BUAA), Beijing 100191, P. R. China}
\email{lixiuduan@snnu.edu.cn}

\address{Department of Mathematics, University of Macau, Macau SAR, P. R. China}
\email{changfenggui@um.edu.mo}

\address{School of Mathematics and Statistics, Yunnan University, Kunming, Yunnan 650504, P. R. China}
\email{liyuan5397@163.com}

\thanks{Wei Dai is supported by the NNSF of China (No. 12222102 \& No. 12571113), the National Science and Technology Major Project (2022ZD0116401) and the Fundamental Research Funds for the Central Universities. Lixiu Duan is supported by the National Science and Technology Major Project (2022ZD0116401), the Fundamental Research Funds for the Central Universities and the Academic Excellence Foundation of BUAA for PhD Students. Changfeng Gui is supported by  NSFC Key Program (Grant No.12531010), University of Macau research grants CPG2024-00016-FST, CPG2025-00032-FST, CPG2026-00027-FST, SRG2023-00011-FST, MYRGGRG2023-00139-FST-UMDF, UMDF Professorial Fellowship of Mathematics, Macao SAR FDCT 0003/2023/RIA1 and Macao SAR FDCT 0024/2023/RIB1.   Yuan Li is supported by the NNSF of China (No. 12401132).}

\maketitle

\begin{abstract}
In this paper, we investigate the following quasi-linear weighted $N$-Laplacian Liouville equation
\begin{equation*}
-\Delta_N u=|x|^{N\alpha}e^{u}, \qquad  x\in \R^N,
\end{equation*}
where $N \geq 2$. For $\al>0$, by carefully studying the linearized problem and applying the approximation method and bifurcation theory, we prove that, when the parameter $\alpha$ equals to the critical values $\alpha(k):=\frac{\sqrt{k(N-1)(k+N-2)}}{N-1}-1$ for $k \geq 2$, there exist non-radial solutions $u$ (bifurcating from $U_{\alpha(k)}$) to the above quasi-linear H\'enon type Liouville equation such that $u\sim \ln|x|$, $|\nabla u|= O(|x|^{-1})$ at $\infty$ and $\int_{\R^N}|x|^{N\alpha}e^{u}\md x=N\left(\frac{N^2}{N-1}\right)^{N-1}(\alpha+1)^{N-1}\omega_N$. One should note that, $\alpha(k)=k-1$ for $k\geq2$ when $N=2$. Our results successfully extend the existence result of J. Prajapat and G. Tarantello in \cite{PT} concerning the $2$-dimension and Laplacian case (i.e., $N=2$) to the more general $N$-dimension and $N$-Laplacian cases ($N\geq 2$), and extend the results of F. Gladiali, M. Grossi, and S. L. N. Neves in \cite{GGN} and the authors in \cite{DDGL} from $1<p<N$ to the much more complicated limiting case $p=N$. If $k$ is even, our results indicate that nonradial solutions possess more richer bi-radial symmetries $\mathcal{O}(N-l)\times \mathcal{O}(l)$ ($l=1,\cdots,[\frac{N}{2}]$) than the axial symmetry $\mathcal{O}(N-1)$ as $N$ increases. We introduced some new ideas and overcame a series of crucial difficulties, including the nonlinearity nature of the $N$-Laplacian $\Delta_N$, the lack of Green integral representation formula and critical weighted Sobolev embedding inequality, the absence of Kelvin type transforms for linearized/difference equations, the invariance of the total mass under scalings of $u$, and the signs-changing and divergence (to $-\infty$) at $\infty$ of the solutions, which makes the suitable choices of the approximate problems, the (normalized) approximate function sequences and the working space to be quite difficult.
\end{abstract}


\smallskip


\section{Introduction}
In this paper, we study weak solution $u$ to the following Liouville H\'enon-Hardy equation involving $N$-Laplacian
\begin{equation}\label{11}
-\Delta_N u=|x|^{N\alpha}e^{u}, \qquad x\in \R^N,
\end{equation}
where $\Delta_N(\cdot):=\text{div}(|\nabla(\cdot)|^{N-2}\nabla(\cdot))$ is $N$-Laplace operator (i.e., $p$-Laplace operator with $p=N$), $N\geq2$ and $\alpha>-1$. Equation \eqref{11} is invariant under the scaling $u\mapsto u_{\lambda}(\cdot):=u(\lambda\cdot)+N(\alpha+1)\ln\lambda$ for any $\lambda>0$.

A solution $u$ of \eqref{11} stands for a function $u\in W_{loc}^{1,N}(\R^N)$ which satisfies
$$\int_{\R^N}|\nabla u|^{N-2}\nabla u\cdot\nabla\varphi=\int_{\R^N}|x|^{N\al}e^u\varphi,\quad\forall\ \varphi\in C_0^1(\R^N).$$

\subsection{Existence of non-radial solutions for $\mathbf{\alpha} > 0$}
\subsubsection{Background and main results}
Let us review some known radial symmetry and classification results for $-1<\al\leq0$. If $\alpha=0$, \eqref{11} becomes the famous $N$-Laplacian Liouville equation
\begin{align}\label{criticeqution}
	-\Delta_N u=e^{u}, \qquad x\in \R^N.
\end{align}
The Liouville equation was initially investigated by Liouville \cite{Liou}, which originates from a variety of situations, such as from the prescribed Gaussian curvature problem in differential geometry, and from combustion theory, astrophysics, nonlinear elasticity and plastic mechanics in physics.

When $N=2$, W. Chen and C. Li in \cite{CL} classified the solutions with finite mass $\int_{\R^2}e^{u}\md x<+\infty$ to \eqref{criticeqution} by using the method of moving planes and isoperimetric inequality. For general $N\geq2$, by using the Pohozaev identity and the isoperimetric inequality, P. Esposito \cite{E1} proved the classification result for \eqref{criticeqution} under finite mass condition $\int_{\R^N}e^{u}\md x<+\infty$. For $\al\in(-1,0)$, V. H. Nguyen \cite{N} proved the classification of solutions for \eqref{11} under the extra assumptions $\int_{\R^N}|x|^{N\alpha}e^{u}\md x\leq N\left(\frac{N^2(  \al+1)}{N-1}\right)^{N-1}\omega_N$ and $\sup\limits_{x\in\R^N} u<+\infty$. P. Esposito \cite{E2} derived the necessary condition for the existence of finite mass solutions to \eqref{11} is $\al>-1$. Subsequently, G. Ciraolo and X. Li \cite{CiL} derived classification result for the anisotropic $N$-Laplacian Liouville equation in the whole space $\mathbb{R}^{N}$, W. Dai, C. Gui and Y. Luo \cite{DGL} proved the classification results for anisotropic $N$-Laplacian Liouville equation in general convex cones $\mathcal{C}$. For more related literature on Liouville type theorems and classification results for Liouville type equation, refer to e.g. \cite{CDQ0,C,CY,CDQ,DGHP,DLS,DQ,DQ0,GHM,GJM,GM,Li,Lin,WX} and the references therein.

Unlike the cases $\al \leq 0$, the radial symmetry and classification of solutions to equation \eqref{11} with $\al > 0$ cannot be proved by using the  method of moving planes, the method of moving spheres in conjunction with the integral representation formula, and isoperimetric inequality, indicating that nonradial solutions to \eqref{11} may exist. When $N=2,\ \alpha=b-1$ and $b>0$, the equation \eqref{11} reduces to the classical second order equation with regular Laplace operator, i.e.,
\begin{equation}\label{17e}
-\Delta u=8\pi b|x|^{2(b-1)}e^{u}, \qquad x\in \R^2.
\end{equation}
Under the finite mass condition $\int_{\R^2}|x|^{2(b-1)}e^{u}\md x=1$, J. Prajapat and G. Tarantello \cite{PT} completely classified the solutions to the two-dimensional Liouville equation \eqref{17e} using the method of moving planes and complex analysis techniques. They proved that, if $b\notin\mathbb{N}$, then $u(x)=u_*(\var x)+2b\ln\var$ for some $\var>0$ and $u_*(x)=\ln\frac{1}{\left(1+|x|^{2b}\right)^2}+\ln\frac b\pi$. If $b\in\mathbb{N}$, let $\theta$ denote the angular coordinate of $x$ in polar coordinates. Define the function
$$
U_b^*(x) = -2\ln\left(1-2|x|^b \cos b(\theta - \theta_0)\tanh \xi + |x|^{2b}\right) - 2\ln(\cosh \xi) + \ln\frac b\pi,
$$
where $\xi, \theta_0 \in \mathbb{R}$. Then $U_b^*$ satisfies equation \eqref{17e}, and any solution to \eqref{17e} can be expressed as
$$
u(x) = U_b^*(\var x) + 2b \ln \var
$$
for some $\var > 0$. In \cite{GGN1}, F. Gladiali, M. Grossi and S. L. N. Neves also investigated the $2$-D weighted Liouville equation with exponential nonlinearity. They established the existence of branches of non-radial solutions at suitable values of the positive parameter $\lambda$ for the corresponding Dirichlet problem, and characterized the kernel of linearized operator and Morse index of the radial solution.

For $1<p<N$, the counterpart to \eqref{11} is the following $D^{1,p}$-critical quasi-linear H\'enon equation involving $p$-Laplacian:
\begin{equation}\label{1}
\left\{
\begin{aligned}
&-\Delta_p u=|x|^{\alpha}u^{p_\al^*-1}, & x\in \R^N, \\
&u>0, & x\in \R^N,
\end{aligned}
\right.
\end{equation}
where $N\geq2, \ 1<p<N, \ p_\al^*=\frac{p(N+\al)}{N-p}$ and $\alpha>0$. For the special case $p=2$, when $\alpha=\alpha(k)=2(k-1)$ with $k\geq2$, F. Gladiali, M. Grossi and S. L. N. Neves \cite{GGN} proved the existence of nonradial solutions (bifurcating from the radial solution $U_{\alpha(k)}$) to equation \eqref{1} using bifurcation theory and approximation method. Specially, when $\al=2$ and $N\geq4$ is even, they obtained that, for any $a\in\R$, the functions
$$u(x)=u(|x'|,|x''|)=\frac{1}{\lr1+|x|^4-2a(|x'|^2-|x''|^2)+a^2\rr^{\frac{N-2}{4}}}$$
form a branch of nonradial solutions to \eqref{1} bifurcating from the radial solution $U_2:=\frac{1}{(1+|x|^4)^{\frac{N-2}{4}}}$, where $x\in\R^{\frac{N}{2}}\times\R^{\frac{N}{2}}$, $x=(x',x'')$ with $x'\in\R^{\frac{N}{2}}$ and $x''\in\R^{\frac N2}$. They also raised a conjecture that the nonradial solutions to \eqref{1} exist only when $\al=\al(k)=2(k-1)$. Very recently, for $p=2$, Q. Jiang and J. Xiong \cite{JX} surprisingly proved that, if $k>\frac{N-2}{2}$ is even, then there exists nonradial solutions to \eqref{1} for $\alpha\neq\al(k)$ close to $\al(k)$, which gives a negative answer to this conjecture in \cite{GGN}. They recast the problem as a semilinear equation with supercritical exponent on the cylinder via the Emden-Fowler change of variables. This approach can hardly be applied to general $p$-Laplacian cases with $p\neq2$, because it is difficult for us to separate the angular gradient and the radial gradient.

For general $1<p<N$, by carefully studying the linearized problem and applying the approximation method and bifurcation theory, the authors proved in \cite{DDGL} that, when the parameter $\al$ takes the critical values $\al(k):=\frac{p\sqrt{(N+p-2)^2+4(k-1)(p-1)(k+N-1)}-p(N+p-2)}{2(p-1)}$ for $k\geq2$, the above quasi-linear H\'enon equation \eqref{1} admits non-radial solutions $u$ (bifurcating from the radial solution $U_{\alpha(k)}$) such that $u\sim |x|^{-\frac{N-p}{p-1}}$ and $|\nabla u|\sim |x|^{-\frac{N-1}{p-1}}$ at $\infty$. One should note that, $\alpha(k)=2(k-1)$ for $k\geq2$ when $p=2$. The results in \cite{DDGL} successfully extend the classical work of F. Gladiali, M. Grossi, and S. L. N. Neves in \cite{GGN} concerning the Laplace operator (i.e., the case $p=2$) to the more general setting of the nonlinear $p$-Laplace operator ($1<p<N$). Comparing with the Laplacian case $p=2$, the authors have to deal with the following crucial difficulties specific to $p$-Laplacian cases $p\neq2$ in \cite{DDGL}:
\begin{itemize}
  \item nonlinear feature of the $p$-Laplacian $\Delta_p$,
  \item absence of Kelvin type transforms,
  \item lack of the Green integral representation formula.
\end{itemize}
In order to overcome these crucial difficulties, the authors introduced in \cite{DDGL} some new ideas and methods:
\begin{itemize}
  \item used auxiliary functions and comparison principles instead of Green integral representation formula to prove uniform upper bound estimate for nonradial approximating solution and its gradient, and the rigidity of the limiting scale;
  \item proved and applied a weighted Sobolev embedding inequality to overcome the nonlinearity nature of the $p$-Laplacian and establish a uniform boundedness estimate on a weighted integral of the $L^\infty$-normalized difference function between radial and nonradial approximating solutions;
  \item lifted the uniform boundedness estimate on a weighted integral up to a pointwise uniform fast decay estimate by De Giorgi-Morser-Nash iteration method, without using Kelvin type transforms and the Green integral representation formula.
\end{itemize}

\medskip

In this paper, we mainly aim to prove the existence of nonradial solutions to the quasi-linear H\'enon type $N$-Laplacian Liouville equation \eqref{11}, and to extend the existence result of J. Prajapat and G. Tarantello in \cite{PT} concerning the $2$-dimension and Laplacian case (i.e., $N=2$) to the more general $N$-dimension and $N$-Laplacian cases ($N\geq 2$), and extend the results of F. Gladiali, M. Grossi, and S. L. N. Neves in \cite{GGN} and the authors in \cite{DDGL} from $1<p<N$ to the much more complicated limiting case $p=N$. One should note that, for $N>2$, both the method of moving planes and complex analysis techniques are not applicable, so we can not derive a complete classification of solutions for \eqref{11}. We need to introduce some new ideas and overcome a series of crucial difficulties, including the nonlinearity nature of the $N$-Laplacian $\Delta_N$, the absence of  Kelvin type transforms for linearized/difference equations, the lack of Green integral representation formula and critical weighted Sobolev embedding inequality, the invariance of the total mass under scalings of $u$, and the signs-changing and divergence (to $-\infty$) at $\infty$ of the solutions, which makes the suitable choices of the approximate problems, the (normalized) approximate function sequences and the working space to be quite difficult.

Problem \eqref{11} admits explicit radial solutions, which can be written as
\begin{align}\label{15}
	U_{\lambda,\al}(x) = \ln\frac{C_{N,\al} \lambda^{N( \al+1)}}{\left( 1 + (\lambda|x|)^{\frac{N( \al+1)}{N-1}} \right)^N},\quad x\in\R^N,
\end{align}
where $C_{N,\al} = N( \al+1)\left( \frac{N^2( \al+1)}{N-1} \right)^{N-1}.$
Specifically, if we set $\la = 1$, then
\begin{align}\label{14}
U_\al(x) := U_{1,\al}(x) = \ln\frac{C_{N,\al}}{\left( 1 + |x|^{\frac{N( \al+1)}{N-1}} \right)^N},\quad x\in\R^N.
\end{align}

For $\al > -1$, we first derive the following uniqueness result for radially symmetric solutions to \eqref{11}.
\begin{thm}\label{thm12}
Suppose $N \geq 2$ and $\al > -1$. Then problem \eqref{11} in $W^{1,N}_{loc}(\R^N)$ has a unique (up to scalings) radial nontrivial solution of the form $U_{\lambda,\al}(x)$ with some $\lambda > 0$, where $U_{\lambda,\al}(x)$ defined in \eqref{14}.
\end{thm}

We need to study the linearized problem  of equation \eqref{11}  at the radial solution $U_\al$, i.e.,
\begin{equation}\label{16}
		-\textrm{div}(|\nabla U_\alpha|^{N-2}\nabla v)-(N-2)
		\textrm{div}
		\left(|\nabla U_\alpha|^{N-4}(\nabla U_\alpha\cdot\nabla v)\nabla U_\alpha\right)=|x|^{N\alpha}e^{U_\alpha}v.
\end{equation}
In this paper, we say $v$ is a solution to \eqref{16}, if $v\in L_{loc}^\infty(\mathbb{R}^N) \cap W^{1,N}_{loc}(\R^N)$ such that $\liminf\limits_{R\rightarrow+\infty}\sup\limits_{|x|=R}\frac{|v(x)|}{\ln|x|}=0$ satisfies
$$\int_{\Omega}|\nabla U_\alpha|^{N-2}\nabla v\cdot\nabla\phi+(N-2)
		|\nabla U_\alpha|^{N-4}(\nabla U_\alpha\cdot\nabla v)(\nabla U_\alpha\cdot\nabla\phi)\mathrm{d}x=\int_{\Omega}|x|^{N\alpha}e^{U_\alpha}v\phi\mathrm{d}x$$
for any bounded open subset $\Omega\subset\mathbb{R}^{N}$ and any $\phi\in W^{1,N}_{0}(\Omega)$.

For the special case $\alpha = 0$, in Theorem 1 in \cite{T}, F. Takahashi proved that all solutions $v \in L^\infty(\mathbb{R}^N) \cap C^2(\mathbb{R}^N)$ of the equation
\begin{align}\label{Ta}
-\operatorname{div}\bigl(\abs{\nabla U_0}^{N-2}\nabla v\bigr) - (N-2)\operatorname{div}\left(\abs{\nabla U_0}^{N-4}(\nabla U_0\cdot\nabla v)\nabla U_0\right) = e^{U_0}v
\end{align}
in $\mathbb{R}^N$ can be expressed as a linear combination of the functions
\begin{equation*}
Z_0(x) = N + x\cdot\nabla U_0,\qquad Z_i(x) = \frac{\partial U_0(x)}{\partial x_i},\quad i = 1,\cdots,N.
\end{equation*}
Very recently, by transferring the weighted cases $\al>-1$ to the case $\al=0$ studied by F. Takahashi \cite{T}, G. Ciraolo, P. Esposito and X. Li classified all solutions $v\in L^\infty(\mathbb{R}^N) \cap W^{1,N}_{loc}(\R^N)$ to the linearized problem \eqref{16} for $\al>-1$ in Theorem 5.1 in \cite{CEL}.

We can remove the global boundedness assumption $v\in L^\infty(\mathbb{R}^N)$ in F. Takahashi \cite{T} and G. Ciraolo, P. Esposito and X. Li \cite{CEL}, and derive the following classification result for all solutions $v\in L_{loc}^\infty(\mathbb{R}^N) \cap W^{1,N}_{loc}(\R^N)$ such that $\liminf\limits_{R\rightarrow+\infty}\sup\limits_{|x|=R}\frac{|v(x)|}{\ln|x|}=0$ of linearized problem \eqref{16}.
\begin{thm}\label{th11}
Let $\alpha>-1$. If $\alpha\ne\alpha(k)$, then the space of solutions of \eqref{16} has dimension $1$ and is spanned by
\begin{align}\label{17'}
\widehat{Z}(x)=\frac{(N-1)-|x|^{\frac{N( \al+1)}{N-1}}}
{1+|x|^{\frac{N( \al+1)}{N-1}}}.
\end{align}
If $\alpha=\alpha(k)$ for some $k\in \N$, then the space of solutions of \eqref{16} has dimension $1+\frac{(N+2k-2)(N+k-3)!}{(N-2)!k!}$ and is spanned by
\begin{align}\label{18'}
	\widehat{Z}(x)=\frac{(N-1)-|x|^{\frac{N( \al+1)}{N-1}}}
	{1+|x|^{\frac{N( \al+1)}{N-1}}},\quad\quad
	\widehat{Z_{k,i}}(x)=\frac{|x|^{\frac{ \al+1}{N-1}-k}\Phi_{k,i}(x)}
	{1+|x|^{\frac{N( \al+1)}{N-1}}},
\end{align}
where $$\alpha(k):=\frac{\sqrt{k(N-1)(k+N-2)}}{N-1}-1$$
for $k\in\N$ and $\{\Phi_{k,i}\}\ (i = 1,\cdots, \frac{(N+2k-2)(N+k-3)!}{(N-2)!k!})$ form a basis of the space $\Y_k(\R^N)$ consisting of all homogeneous harmonic polynomials of degree $k$ in $\R^N$.
\end{thm}
\begin{rem}
It is clear that the global boundedness assumption $v\in L^\infty(\mathbb{R}^N)$ in Theorem 1 in \cite{T} and in Theorem 5.1 in \cite{CEL} can be replaced by the more weaker assumptions $v\in L_{loc}^\infty(\mathbb{R}^N)$ such that $\liminf\limits_{R\rightarrow+\infty}\sup\limits_{|x|=R}\frac{|v(x)|}{\ln|x|}=0$ in our Theorem \ref{th11}. In fact, by changing variables, similar to the proof of Theorem 5.1 in \cite{CEL}, we can also transfer the weighted cases $\al>-1$ to the case $\al=0$ in Theorem 1 in \cite{T}. Then, it can be seen clearly from ``the case $k=0$" in the proof of Theorem 1 in page 6 of \cite{T} that, the assumption $\liminf\limits_{R\rightarrow+\infty}\sup\limits_{|x|=R}\frac{|v(x)|}{\ln|x|}=0$ can replace the assumption $v\in L^\infty(\mathbb{R}^N)$ and can also lead to a contradiction. The rest of the proof of Theorem \ref{th11} is similar to Theorem 1 in \cite{T} and Theorem 5.1 in \cite{CEL}, so we omit the details here.
\end{rem}

We define the operator associated with the equation \eqref{11} by
\begin{equation*}
    \mathcal{F}(u) := -\Delta_N u - |x|^{N\alpha} e^u.
\end{equation*}
Then, the linearized equation $\mathcal{L}_{\mathcal{F},\, U_\alpha}(v) = 0$ for operator $\mathcal{F}$ at $U_\al$ is exactly \eqref{16}.

Under the change of variable $g := e^u$, the corresponding operator acting on $g$ is defined as
\begin{equation*}
    \mathcal{T}(g) := -\Delta_N g - |x|^{N\alpha} g^N + (N-1)\frac{|\nabla g|^N}{g}.
\end{equation*}
By direct expansion, we obtain $
    g^{N-1}\bigl[-\Delta_N(\ln g)\bigr] = -\Delta_N g + (N-1)\frac{|\nabla g|^N}{g}.$ This leads to the operator identity
\begin{equation}\label{fg}
    \mathcal{T}(g) = g^{N-1}\mathcal{F}(\ln g).
\end{equation}
In particular, we have
\begin{equation*}
    \mathcal{F}(U_\alpha) = 0 \quad \Longleftrightarrow \quad \mathcal{T}(e^{U_\alpha}) = 0.
\end{equation*}
The linearized equation $\mathcal{L}_{\mathcal{T},\,e^{U_\alpha}}(v) = 0$ for operator $\mathcal{T}$ at $e^{U_{\al}}$ can be written explicitly as
\begin{align}\label{v224}
    &\quad -\operatorname{div}\left( |\nabla e^{U_\alpha}|^{N-2}\nabla v + (N-2)|\nabla e^{U_\alpha}|^{N-4}(\nabla e^{U_\alpha}\cdot\nabla v)\nabla e^{U_\alpha} \right) \\
    &= N|x|^{N\alpha}e^{(N-1)U_\al}v - (N-1)\left[ \frac{N|\nabla e^{U_\alpha}|^{N-2}(\nabla e^{U_\alpha}\cdot\nabla v)}{e^{U_\alpha}} - \frac{|\nabla e^{U_\alpha}|^N v}{e^{2U_\alpha}} \right].\nonumber
\end{align}
We say $v$ is a solution to \eqref{v224}, if $v\in L_{loc}^\infty(\mathbb{R}^N) \cap W^{1,N}_{loc}(\R^N)$ such that $\liminf\limits_{R\rightarrow+\infty}\sup\limits_{|x|=R}\frac{e^{-U_\al}|v(x)|}{\ln|x|}=0$ satisfies
\begin{align*}
&\quad \int_{\Omega}|\nabla e^{U_\alpha}|^{N-2}\nabla v\cdot\nabla\phi + (N-2)|\nabla e^{U_\alpha}|^{N-4}(\nabla e^{U_\alpha}\cdot\nabla v)(\nabla e^{U_\alpha}\cdot\nabla\phi)\mathrm{d}x\\
&=\int_{\Omega}\left\{N|x|^{N\alpha}e^{(N-1)U_\al}v - (N-1)\left[ \frac{N|\nabla e^{U_\alpha}|^{N-2}(\nabla e^{U_\alpha}\cdot\nabla v)}{e^{U_\alpha}} - \frac{|\nabla e^{U_\alpha}|^N v}{e^{2U_\alpha}} \right]\right\}\phi\mathrm{d}x
\end{align*}
for any bounded open subset $\Omega\subset\mathbb{R}^{N}$ and any $\phi\in W^{1,N}_{0}(\Omega)$.

By showing the relationship between the kernels $\ker \mathcal{L}_{\mathcal{T},\,e^{U_\alpha}}$ and $\ker \mathcal{L}_{\mathcal{F},\, U_\alpha}$, we establish the following classification result for all solutions $v\in L_{loc}^\infty(\mathbb{R}^N) \cap W^{1,N}_{loc}(\R^N)$ such that $\liminf\limits_{R\rightarrow+\infty}\sup\limits_{|x|=R}\frac{e^{-U_\al}|v(x)|}{\ln|x|}=0$ of \eqref{v224}.
\begin{thm}\label{c13}
Let $\alpha>-1$. If $\alpha\ne\al(k)$, then the space of solutions of \eqref{v224} has dimension $1$ and is spanned by
\begin{align}\label{17}
Z(x)=\frac{(N-1)-|x|^{\frac{N(1+\al)}{N-1}}}
{\lr1+|x|^{\frac{N(1+\al)}{N-1}}\rr^{N+1}}
\end{align}
If $\alpha=\al(k)$ for some $k\in \N$, then the spaces of solutions of \eqref{v224} has dimension $1+\frac{(N+2k-2)(N+k-3)!}{(N-2)!k!}$ and is spanned by
\begin{align}\label{18}
	Z(x)=\frac{(N-1)-|x|^{\frac{N(1+\al)}{N-1}}}
	{\lr1+|x|^{\frac{N(1+\al)}{N-1}}\rr^{N+1}},\quad\quad
	Z_{k,i}(x)=\frac{|x|^{\frac{1+\al}{N-1}-k}\Phi_{k,i}(x)}
	{\lr1+|x|^{\frac{N(1+\al)}{N-1}}\rr^{N+1}}
\end{align}
where $\al(k):=\frac{\sqrt{k(N-1)(k+N-2)}}{N-1}-1$, and $\{\Phi_{k,i}\}$ ($i = 1,\cdots, \frac{(N+2k-2)(N+k-3)!}{(N-2)!k!}$) form a basis of $\Y_k(\R^N)$, the space of all homogeneous harmonic polynomials of degree $k$ in $\R^N$.
\end{thm}

For more literature on non-degeneracy and classification results for linearized operators related to $p$-Laplacian equations, refer to e.g. \cite{DDGL,FN,FZ,PV} and the references therein.

\medskip

Next, we apply Theorem \ref{c13} to investigate the existence of nonradial solutions to equation \eqref{11}. To this end, we introduce necessary notations and definitions.

Define the weighted space
\begin{align*}
L^{\infty}_\gamma(\mathbb{R}^N) := \left\{ g \in L^{\infty}(\mathbb{R}^N) \mid\,\,  \|g\|_{\gamma}<+\infty \right\}
\end{align*}
with the weighted norm
\begin{align}\label{d18}
	\|g\|_\gamma := \sup_{x\in\R^N}\left(1+|x|\right)^{\gamma}|g(x)|, \quad \text{ where } \gamma\in \left(N(\al+1),\frac{N^2(\al+1)}{N-1}\right).
\end{align}
For the convenience of subsequent calculations, we may fix $\gamma=N(\al+1)+\frac12$ throughout our paper.
Furthermore, set
\begin{align}\label{19}
X:=X_\al= W^{1,N}(\R^N)\bigcap L^\infty_\gamma(\R^N),
\end{align}
which is a Banach space equipped with the norm
\begin{align}\label{110}
	\|g\|_X:=\max\{\|g\|_{1,N},\|g\|_\gamma\},
\end{align}
where $\|\cdot\|_{1,N}$ denotes the usual norm in $W^{1,N}(\R^N)$, i.e., $\|g\|_{1,N} = \left(\int_{\R^N}|g|^N\md x\right)^{\frac{1}{N}}+\left(\int_{\R^N}|\nabla g|^N\md x\right)^{\frac{1}{N}}$ for $g\in W^{1,N}(\R^N)$. If there is no need to emphasize the dependency on $\al$, we will abbreviate $X_\al$ as $X$.

For any fixed $\bar{\alpha}$ and any $\alpha \in (\bar{\alpha}-\delta, \bar{\alpha}+\delta)$, we use the norm for $X_{\bar{\alpha}+\delta}$ to determine the common topology on $X=X_\al$. The $\delta$-neighborhood of $(\bar{\alpha}, g_{\bar{\alpha}})$ in $\bigcup\limits_{\al\in(0,+\infty)}\bigcup\limits_{g\in X_\al}\{(\al,g)\}$ is defined by
$$
\bigg\{ (\alpha, g)\in \bigcup\limits_{\al\in(0,+\infty)}\bigcup\limits_{g\in X_\al}\{(\al,g)\} \bigg| \, |\alpha - \bar{\alpha}| + \| g - g_{\bar{\alpha}} \|_{X_{\bar{\alpha}+\delta}} < \delta \bigg\}.
$$

\begin{defn}
Let $U_\alpha$ be the radial solution of \eqref{11} defined in \eqref{14}. We say that a non-radial bifurcation occurs at $(\bar{\alpha},U_{\bar{\alpha}})$, if for every neighborhood of $(\bar{\alpha},e^{U_{\bar{\alpha}}})$ in $\bigcup\limits_{\al\in(0,+\infty)}\bigcup\limits_{g\in X_\al}\{(\al,g)\}$, there exists a point $(\alpha,e^{v_\alpha})$  in this neighborhood such that $v_\alpha$ is a non-radial solution of \eqref{11}.
\end{defn}

Let $\mathcal{O}(k)$ be the orthogonal group in $\R^k$. Our main result is the following theorem.
\begin{thm}\label{th14}
	Assume $N\geq2$. Let $\al=\alpha(k):=\frac{\sqrt{k(N-1)(k+N-2)}}{N-1}-1$ with $k\in\N,\ k\geq2$. Then \\
	(i) there exists at least a continuum of nonradial solutions to \eqref{11}, invariant with
	respect to $\mathcal{O}(N-1)$, bifurcating from the pair $(\al,U_\al)$.\\
	(ii) If $k$ is even, there exist at least $[\frac{N}{2}]$ continua of nonradial solutions to \eqref{11} bifurcating from $(\al,U_\al)$. The $l$-th branch is invariant with respect to $\mathcal{O}(N-l)\times \mathcal{O}(l)$ for $l=1,\cdots,[\frac{N}{2}]$.\\
	Furthermore, all these nonradial solutions $h$ derived in (i) and (ii) satisfy $e^{h}\in X$, $h\in C^{1,\eta}_{loc}(\mathbb{R}^{N})$, $h(x)\sim-\ln|x|$ and $|\nabla h|= O(|x|^{-1})$, as $|x|\to+\infty$, and $\int_{\R^N}|x|^{N\alpha}e^{h}\md x=N\left(\frac{N^2}{N-1}\right)^{N-1}(\alpha+1)^{N-1}\omega_N$.
\end{thm}

\begin{rem}
Theorem \ref{th14} demonstrates that for $\al>0$, the solution structure to equation \eqref{11} becomes significantly more complex than in the case $\al=0$. In particular, for $N=2$, one has $\al(k)=k-1$ ($k\in\N$). So our results in Theorem \ref{th14} coincide with the existence result in Theorem 1.1 of \cite{PT}. Consequently, our results extend the classical existence result by J. Prajapat and G. Tarantello in \cite{PT} on the Laplace operator (the special case $N=2$) to the more general framework of the nonlinear $N$-Laplace operator ($N\geq 2$), and extend the results of F. Gladiali, M. Grossi, and S. L. N. Neves in \cite{GGN} and the authors in \cite{DDGL} from $1<p<N$ to the limiting case $p=N$. When $k$ is even, (ii) in Theorem \ref{th14} indicates that nonradial solutions possess more richer bi-radial symmetries $\mathcal{O}(N-l)\times \mathcal{O}(l)$ ($l=1,\cdots,[\frac{N}{2}]$) than the axial symmetry $\mathcal{O}(N-1)$ as $N$ increases.
\end{rem}

\begin{rem}
Besides the bifurcation arguments, the existence of nonradial solutions can also be deduced from other ways or mechanism, such as symmetry breaking for the corresponding variational minimization problems. For instance, see Theorem 1.4 in \cite{CPP} for the Neumann problem of \eqref{criticeqution} (i.e., $\alpha=0$ in \eqref{11}) with $p=2$ in nonconvex unbounded cone under some assumptions on the spherical cross-section $D$ of the cone; and Theorem 3 in \cite{NR} for the Neumann problem of \eqref{11} with $\alpha\leq 0$ in unbounded cone under some assumptions on the cross-section $D$ of the cone.
\end{rem}

\subsubsection{The crucial difficulties specific to the limiting case $p=N\geq2$ and comparison with known results \cite{DDGL,GGN,PT} for $1<p<N$ and $p=N=2$.}
We would like to mention some of the key difficulties in our proof, comparing with the special $2$-D semi-linear case $p=N=2$ in \cite{PT}, semi-linear case $p=2<N$ in \cite{GGN} and quasi-linear cases $1<p<N$ in \cite{DDGL}. The main difficulties lie on the nonlinearity nature of the $N$-Laplacian $\Delta_N$, the lack of Green integral representation formula and critical weighted Sobolev embedding inequality, the absence of Kelvin type transforms for linearized/difference equations, the invariance of the total mass under scalings of $u$, and the signs-changing and divergence (to $-\infty$) at $\infty$ of the solutions, which makes the suitable choices of the approximate problems, the (normalized) approximate function sequences and the working space to be quite difficult, see e.g. Sections 2-5.

Since the complex analysis method in \cite{PT} for the special case $p=N=2$ can not be applied to general $p=N>2$, we will prove our main results in Theorem \ref{th14} by bifurcation theory.

Since the radial solutions $U_\al$ in the whole space $\mathbb{R}^{N}$ are not isolated and are degenerate in the space of radial functions, we cannot directly apply the classical bifurcation theory in $\mathbb{R}^{N}$ to obtain the existence of nonradial bifurcation points. Inspired by Gladiali, Grossi and Neves \cite{GGN}, we will show the existence of nonradial bifurcation solutions to an approximate problem \eqref{31} in balls $B_{\frac{1}{\var}}(0)$ with radius $\frac{1}{\var}$ instead, and establish some key uniform estimates on approximate solutions, then prove these nonradial solutions in $B_{\frac{1}{\var}}(0)$ will converge to the nonradial solutions of problem \eqref{11} in $\mathbb{R}^{N}$, as $\var=\var_n\rightarrow0$, see also \cite{DDGL}.

Due to the signs-changing and divergence (to $-\infty$) at $\infty$ of the radial solutions $U_{\la,\alpha}$, we can not simply take Dirichlet problem as the approximate problem on $B_{\frac{1}{\var}}$ like \cite{DDGL,GGN}. We carefully set the approximate problem \eqref{31} with limiting blowing-up boundary value and show that, the radial approximate solution $u_{\var,\alpha}$ to \eqref{31} is isolated and nondegenerate in the space of radial functions (see Lemma \ref{lem31}), the linearized equation at $u_{n,\al}:=u_{\var_n,\alpha}$ is non-degenerate and the operator $I-\mathcal{L}_{h}^n(\al,u_{n,\al}):\mathcal{Q}_n\to \mathcal{Q}_n$ is invertible for $\al \neq\al_k^n$, $k=1,2,\cdots$ (see Remark \ref{morse-index} and Theorem \ref{thm38}). Therefore, following the methods in Gladiali, Grossi and Neves \cite{GGN} (see also \cite{DDGL}) and overcoming some technique difficulties caused by the nonlinearity of $N$-Laplacian, by applying the classical bifurcation theory in the balls $B_{\frac{1}{\var}}(0)$, we can obtain the existence of nonradial solutions that bifurcate from some radial approximate solutions close to $U_{\al}$, see Section 3.

Next, the main key point is to prove that, these nonradial solutions in $B_{\frac{1}{\var}}(0)$ will converge, in a certain sense, to the nonradial solutions of problem \eqref{11} in $\mathbb{R}^{N}$. To this end, let $\al_n$ and $\var_n$ be sequences such that $\al_n\to\al>0$ and $\var_n\to0$ as $n\to+\infty$. Let $u_n:=u_{\var_{n},\alpha_{n}}$ and $v_n$ be a sequence of nonradial solutions of \eqref{31} in $B_{\frac{1}{\var_n}}$ related to the exponent $\al_n$. For $1<p<N$, in \cite{DDGL,GGN}, the authors only need to consider $u_n$ and $v_n$ themselves, define $w_n:=\frac{u_n-v_n}{\|u_n-v_n\|_{L^{\infty}(\mathbb{R}^{N})}}$, and show that
\begin{equation}\label{a6}
  \|v_{n}-U_{\lambda,\al}\|_{X}\rightarrow0 \quad \Rightarrow \quad \lambda=1,
\end{equation}
and establish the following key uniform fast decay estimate for $w_n$:
\begin{equation}\label{a3}
  |w_n(y)|\leq\frac{C}{(1+|y|)^{\frac{N-p}{2(p-1)}}}, \qquad \forall \,\, y\in \mathbb{R}^{N},
\end{equation}
and then prove that
\begin{equation}\label{a4}
  \|u_n-v_n\|_{X}\geq C, \qquad \text{if} \,\, \alpha\neq\alpha(k), \,\, \forall \, k\geq1,
\end{equation}
where $C>0$ is a uniform constant independent of $n$, which indicates that the limit of non-radial approximate solutions is non-radial solution in $\R^{N}$.

As to our problem \eqref{11} for $p=N$, since solutions tend to $-\infty$ at logarithmic rate as $|x|\rightarrow+\infty$, we need to investigate convergence of $e^{u_n}$ and $e^{v_n}$ in space $X$, which makes the analysis much more complicated than those for $p<N$ in \cite{DDGL,GGN}. To this end, besides $u_n$, $v_n$ and $w_n$, we also need to define $f_n:= e^{u_{n}} - e^{\beta_{\alpha_{n}}}$ and $g_{n}:= e^{v_{n}}-e^{\beta_{\alpha_{n}}}$ and $h_{n}(x):=\frac{e^{u_n}- e^{v_{n}}}{\|e^{u_n}- e^{v_{n}}\|_{L^{\infty}(\mathbb{R}^N)}}$. Moreover, we define $\psi_n(x) := \frac{u_n - v_n}{\|e^{u_n}-e^{v_n}\|_{L^{\infty}(\mathbb{R}^N)}}$ to relate $w_n$ and $h_n$. In Section 4, we will show that (see Proposition \ref{pp41})
\begin{equation}\label{e21}
  |v_n(x)+\gamma_n\ln|x||\leq C
\end{equation}
with $\gamma_n\to\frac{N^2(\al+1)}{N-1}$ as $n\to+\infty$, and deal with three different sequences $w_n$, $\psi_n$ and $h_n$ simultaneously, and establish the following key uniform $(\ln R)^{-1}$-type integral decay estimates for $\nabla w_n$ outside the ball $B_R(0)$ (see \eqref{ln 420} in Proposition \ref{ppp43}):
\begin{align}
\int_{B_{\frac{1}{\var_n}}\setminus B_R}|x|^{-(N-2)}|\nabla w_n|^2\md x\leq \frac{C}{\ln R}, \qquad \forall \,\, n\geq n_0,
\label{aa438'}
\end{align}
and then show the following uniform decay estimates on $w_n$ and $\nabla w_n$ in annulus near the boundary of ball $\partial B_{\frac{1}{\var_n}}$ (see Propositions \ref{p1} and \ref{pro26}):
\begin{align}\label{412'-}
|w_n(x)|\leq \frac{C}{\sqrt{\ln|x|}}, \qquad \ \forall \,\, x\in \mathbb{R}^{N} \,\,\, \text{such that} \,\, \frac{1}{2\var_n}<|x|<\frac{1}{\var_n},
\end{align}
\begin{align}\label{g4'}
|\nabla w_n(x)|\leq\frac{C\var_n}{\sqrt{|\ln\var_n|}},
\qquad \, \forall \,\, x\in \overline{B_{\frac{1}{\var_n}}}\setminus B_{\frac{3}{4\var_n}}, \quad \forall \,\, n\geq n_0,
\end{align}
and hence establish the rigidity of the limiting scale (see Proposition \ref{pp42}):
\begin{equation}\label{a6'}
  \|e^{v_n(x)}-e^{\beta_{\al_n}}-e^{U_{\lambda,\al}}\|_{X}\rightarrow0 \quad \Rightarrow \quad \lambda=1;
\end{equation}
and prove the following key uniform fast decay estimate for $h_n\sim\psi_{n}e^{U_{\al_n}}$ (see Proposition \ref{ppq43}):
\begin{align}\label{412'+}
|h_n(x)|\leq C\left(\ln R_{0} + (\ln R_{0}) R_0^{-\frac{N(\al+1)}{2(N-1)}}\ln |x|\right)e^{U_{\al_n}(x)}, \qquad \,\,\, \forall \,\, |x| > 2R_0,
\end{align}
which implies $\liminf\limits_{R\rightarrow+\infty}\sup\limits_{|x|=R}\frac{e^{-U_{\al_n}}|h_n|}{\ln|x|}=0$, and hence establish that (see Proposition \ref{pp44})
\begin{equation}\label{a4'}
  \|e^{u_n}-e^{v_n}\|_{X}\geq C, \qquad \text{if} \,\, \alpha\neq\alpha(k), \,\, \forall \, k\geq1,
\end{equation}
where all the uniform constants $C>0$ in the above estimates are independent of $n$, $n_0$, $R$ and $R_0$, which indicates that the limit of non-radial approximate solutions is non-radial solution in $\R^{N}$.

Comparing with $p=2$ in \cite{GGN,PT}, for $p=N$, we have the same difficulties for $p\neq2$ as \cite{DDGL}: nonlinear nature of the $N$-Laplacian $\Delta_N$, absence of Kelvin type transforms for equations satisfied by $w_n$, $\psi_n$ and $h_n$, and lack of the Green integral representation formula. Besides these common difficulties, we still have more additional essential difficulties for the more complicated case $p=N$, comparing with $p<N$ in \cite{DDGL}:
\begin{itemize}
  \item the signs-changing and divergence (to $-\infty$) at $\infty$ of the solutions: makes the suitable choices of the approximate problems, the (normalized) approximate function sequences and the working space to be quite difficult;
  \item the lack of critical weighted Sobolev embedding inequality: $\int_{\R^{N}}\frac{|\phi|^2}{|x|^N}\mathrm{d}x\leq C\int_{\R^{N}}\frac{|\nabla \phi|^2}{|x|^{N-2}}\mathrm{d}x$ does not hold, which is the limiting form of the weighted Sobolev embedding inequality proved in \cite{DDGL} as $p\rightarrow N$. Thus we can not deduce the uniform decay estimate on $w_n$ from the uniform boundedness estimate on $\int_{\R^{N}}\frac{|w_n|^2}{|x|^N}\mathrm{d}x$, namely, the approach in \cite{DDGL} for $p<N$ does not work;
  \item the invariance of the total mass under scalings of $u$: the total mass $\int_{\mathbb{R}^{N}}|x|^{N\al}e^{U_{\la,\al}}\mathrm{d}x$ is invariant with respect to $\lambda>0$, thus the rigidity of the limiting scale \eqref{a6'} can not be deduced from  $\int_{\mathbb{R}^{N}}|x|^{N\al}e^{U_{\la,\al}}\mathrm{d}x=N\left(\frac{N^2(  \al+1)}{N-1}\right)^{N-1}\omega_N$, namely, the approach in \cite{DDGL,GGN} for $p<N$ does not work;
  \item the deficiency of the uniform upper bound estimates for $v_n$: merely proving the uniform upper bound estimates \eqref{43q+} for $v_n$ is far from enough for our proof, namely, the approach in \cite{DDGL,GGN} for $p<N$ does not work. We need to prove the precise asymptotic estimate for nonradial approximate solution $v_{n}$ with explicit constants, so as to establish all the crucial estimates \eqref{e21}--\eqref{a4'} and \eqref{wL2+'} throughout our paper;
  \item the deficiency of the classification results on $\ker\mathcal{L}_{\mathcal{F},\, U_\alpha}\cap L^{\infty}(\mathbb{R}^N)$ in \cite{CEL,PT}: \eqref{412'+} implies that the limiting function $h\in \ker \mathcal{L}_{\mathcal{T},\,e^{U_\alpha}}$ of $h_n$ satisfies $\liminf\limits_{R\rightarrow+\infty}\sup\limits_{|x|=R} \frac{e^{-U_\al}|h(x)|}{\ln |x|}=0$. By the relationship $\ker \mathcal{L}_{\mathcal{T},\,e^{U_\alpha}} = e^{U_\alpha} \ker \mathcal{L}_{\mathcal{F},\,U_\alpha}$ in \eqref{13}, one has $e^{-U_\al}h(x)\in \ker \mathcal{L}_{\mathcal{F},\,U_\alpha}$. Thus the classification results under global boundedness assumption in \cite{CEL,PT} are not enough to show $h=AZ(x)$. The key point is that, the optimal estimate for $h$ exactly match with the best condition for classification of $\ker \mathcal{L}_{\mathcal{T},\,e^{U_\alpha}}$ in Theorem \ref{c13}, so we can not loss any estimates.
\end{itemize}

\subsubsection{Novelty and key ingredients in our proof to overcome the crucial difficulties}
We will give a detailed description on the novelty and key ingredients in our proof to overcome the main difficulties, see Sections 2-5.

\smallskip

In order to overcome the difficulties caused by the signs-changing and divergence at logarithmic rate (to $-\infty$) at $\infty$ of the solutions, we set the work space $X$ matching with $e^{u}$ not $u$ itself, and carefully set the approximate problem \eqref{31} with limiting blowing-up boundary value in Section 3. In order to investigate convergence of $e^{u_n}$ and $e^{v_n}$ in space $X$, we define $f_n:= e^{u_{n}} - e^{\beta_{\alpha_{n}}}$ and $g_{n}:= e^{v_{n}}-e^{\beta_{\alpha_{n}}}$ and transfer the approximate problem \eqref{31} to \eqref{47eq} accordingly, see also \eqref{54eq} in Section 5. Since $w_n\rightarrow A\widehat{Z}(x)$ with $\widehat{Z}(x)\rightarrow-1$ as $|x|\rightarrow+\infty$, we could not expect to derive uniform decay estimate for $w_n$ in the global ball $B_{\frac{1}{\var_n}}$. Thus we define new normalized approximate solutions sequence $h_{n}(x):=\frac{e^{u_n}- e^{v_{n}}}{\|e^{u_n}- e^{v_{n}}\|_{L^{\infty}(\mathbb{R}^N)}}$ which satisfies a very complicated equation \eqref{e22}, and prove the uniform fast decay estimate \eqref{412'+} for $h_n$, so as to lead to a contradiction when $A=0$ in the proof of \eqref{a4'}. To this end, we define $\psi_n(x) := \frac{u_n - v_n}{\|e^{u_n}-e^{v_n}\|_{L^{\infty}(\mathbb{R}^N)}}$, and proved the following uniform upper bound for the $L^2$-integral average of $\psi_n$ on annulus (see Proposition \ref{qpp44}):
\begin{align}\label{wL2+'}
		R^{-N}\int_{B_{8R}\setminus B_{R}}|\psi_n|^2\md x\leq C[\ln R_{0}]^{2} + C [\ln R_{0}]^{2} R_0^{-\frac{N(\al+1)}{N-1}}[\ln R]^{2}, \qquad \forall \,\, R\geq R_0,\ \forall \,\, n\geq n_0,
	\end{align}
which leads to \eqref{412'+} via De Giorgi-Moser-Nash iteration argument. Moreover, in the case $A\neq 0$ in the proof of \eqref{a4'}, we can also use $\psi_n$ to relate the magnitude relationship between boundary derivatives $z_n$ and $\rho_n(\theta)$ to $A>0$ or $A<0$, which leads to a contradiction with the Pohozaev identity (see \eqref{e18})
\begin{equation}\label{e18+'}
\int_{\partial B_1}
\bigl(F_n(\rho_n(\theta))-F_n(z_n)\bigr)\dd\theta=0.
\end{equation}

\smallskip

In order to overcome the difficulties caused by the lack of critical weighted Sobolev embedding inequality, we proved and applied the truncated weighted Hardy-Sobolev inequalities \eqref{aa436+} and \eqref{aa436} on exterior domain $\R^N\setminus B_{R}(0)$ instead of the critical weighted Sobolev embedding inequality. The inequality \eqref{aa436}, in conjunction with \eqref{aa438'} and De Giorgi-Moser-Nash iteration, yields the uniform decay estimates \eqref{412'-} and \eqref{g4'} on $w_n$ and $\nabla w_n$ in annulus near $\partial B_{\frac{1}{\var_n}}$, which implies the convergence of the boundary derivatives $\|z_n-\rho_n(\theta)\|_{L^{\infty}(\partial B_1)}\leq\frac{C}{\sqrt{|\ln\var_n|}}\rightarrow0$ and hence leads to the rigidity of the limiting scale \eqref{a6'}. The inequality \eqref{aa436+}, in conjunction with the Harmonic replacement $H_n$ of $\psi_n$, yields \eqref{wL2+'} and hence \eqref{412'+} via De Giorgi-Moser-Nash iteration. We also use \eqref{aa436+} to relate the magnitude relationship between boundary derivatives $z_n$ and $\rho_n(\theta)$ to $\lambda>1$ or $\lambda<1$ ($A>0$ or $A<0$, resp.) in the proof of \eqref{a6'} (\eqref{a4'}, resp.).

\smallskip

In order to overcome the difficulties caused by the invariance of the total mass under scalings of $u$, we use the uniform decay estimates \eqref{g4'} on $\nabla w_n$ in annulus near $\partial B_{\frac{1}{\var_n}}$ to derive $\|z_n-\rho_n(\theta)\|_{L^{\infty}(\partial B_1)}\leq\frac{C}{\sqrt{|\ln\var_n|}}\rightarrow0$, then apply the truncated weighted Hardy-Sobolev inequality \eqref{aa436+} to relate the magnitude relationship between boundary derivatives $z_n$ and $\rho_n(\theta)$ to $\lambda>1$ or $\lambda<1$, which leads to a contradiction with the Pohozaev identity \eqref{e18+'} and the strict monotonicity of the function $F_n(t)$ on $[N(1+\al_n),+\infty)$ and finally yields the rigidity of the limiting scale \eqref{a6'}.

\smallskip

In order to overcome the difficulties caused by the deficiency of the uniform upper bound estimates for $v_n$, we proved the precise uniform asymptotic estimate \eqref{e21} for nonradial approximate solution $v_n$ with explicit constants $\gamma_n$ instead of the rough upper bound \eqref{43q+}, by using super/sub solution method and comparison principles. From \eqref{e21}, we can deduce two key uniform estimates (see \eqref{w421} and \eqref{xi425}) $$\|u_n-v_n\|_{L^\infty(B_{\frac{1}{\var_n}})}\leq C \qquad \text{and} \qquad \frac{1}{C}e^{U_{\al_n}}\leq \zeta_n(x)=\frac{h_n(x)}{\psi_n(x)}\leq Ce^{U_{\al_n}},$$
and hence prove all the crucial estimates \eqref{e21}--\eqref{a4'} and \eqref{wL2+'} throughout our paper.

\smallskip

In order to overcome the difficulties caused by the deficiency of the classification results on $\ker\mathcal{L}_{\mathcal{F},\, U_\alpha}\cap L^\infty(\R^{N})$ in \cite{CEL,PT}, we improve the classification results in \cite{CEL,PT} from $v\in L^\infty(\R^{N})$ to more general $v\in L_{loc}^\infty(\R^{N})$ satisfying $\liminf\limits_{R\rightarrow+\infty}\sup\limits_{|x|=R} \frac{|v(x)|}{\ln |x|}=0$ in our Theorem \ref{th11} for $\al>-1$, and hence derive the classification results on $\ker \mathcal{L}_{\mathcal{T},\,e^{U_\alpha}}$ for general $v$ satisfying $\liminf\limits_{R\rightarrow+\infty}\sup\limits_{|x|=R} \frac{e^{-U_\al}|v(x)|}{\ln |x|}=0$ in Theorem \ref{c13}. Fortunately, in the proof of \eqref{a4'}, the limiting function $h\in \ker \mathcal{L}_{\mathcal{T},\,e^{U_\alpha}}$ of $h_n$ satisfies $\liminf\limits_{R\rightarrow+\infty}\sup\limits_{|x|=R}\frac{e^{-U_\al}|h(x)|}{\ln |x|}=0$, thus Theorem \ref{c13} implies $h=AZ(x)$, which leads to a contradiction and proves \eqref{a4'}.

\smallskip

In order to overcome the difficulties caused by the absence of the Green integral representation formula, by applying super/sub solution method and comparison principle, we first proved the precise asymptotic estimate \eqref{e21} for nonradial approximate solution $v_n$ with explicit constants $\gamma_n$ without using the Green integral representation formula, and derived the uniform fast decay estimate for $|\nabla v_n|$ via rescaling arguments and regularity estimates in Proposition \ref{pro43}. Then, by proving the uniform upper bound for the $L^2$-integral average on annulus \eqref{wL2+'} and \eqref{wL2'}, we can establish the uniform fast decay estimate \eqref{412'-} for $w_n$ and \eqref{412'+} for $h_n$ via a De Giorgi-Moser-Nash iteration argument, without using the Green integral representation formula. Moreover, by using the uniform decay estimate \eqref{g4'} for $|\nabla w_n|$, \eqref{aa436+} and the Pohozaev identity \eqref{e18+'}, we were also able to prove the rigidity of the limiting scale \eqref{a6'} and the crucial uniform lower bound \eqref{a4'} without resorting the Green integral representation formula.

\smallskip

In order to overcome the difficulties caused by the unavailability of Kelvin type transforms for equations satisfied by $w_n$, $\psi_n$ and $h_n$, by the key uniform $(\ln R)^{-1}$-type integral decay estimates \eqref{aa438'} for $\nabla w_n$ outside the ball $B_R(0)$ and \eqref{aa436}, we can derive the uniform decay estimate on $L^2$-integral average for $w_n$ in annulus (see Proposition \ref{qpp44'}):
 \begin{align}\label{wL2'}
		\frac{1}{R^N}\int_{B_{8R}(0)\setminus B_{R}(0)}|w_n|^2\md x\leq \frac{C}{\ln R}, \qquad \forall \,\, \frac{1}{14\var_n}<R<\frac{1}{2\var_n},\ \forall \,\, n\geq n_0;
	\end{align}
by the harmonic replacement $H_n$ and \eqref{aa436+}, we can derive the uniform upper bound \eqref{wL2+'} for the $L^2$-integral average of $\psi_n$ on annulus; then through a De Giorgi-Moser-Nash iteration argument, we can finally establish the uniform fast decay estimates \eqref{412'-} for $w_n$ and \eqref{412'+} for $h_n$ without using Kelvin type transforms.

\smallskip

In order to overcome the difficulties caused by the nonlinear nature of the $N$-Laplacian $\Delta_N$, by proving and applying the truncated weighted Hardy-Sobolev inequalities \eqref{aa436+} and\eqref{aa436}, we overcame the nonlinear nature of the $N$-Laplacian $\Delta_N$ and first proved the uniform upper bound for the $L^2$-integral average on annulus \eqref{wL2+'} and \eqref{wL2'}, and then established the uniform fast decay estimate \eqref{412'-} for $w_n$ and \eqref{412'+} for $h_n$ via the De Giorgi-Moser-Nash iteration. Finally, by using the uniform decay estimate \eqref{g4'} for $|\nabla w_n|$ to derive $\|z_n-\rho_n(\theta)\|_{L^{\infty}(\partial B_1)}\leq\frac{C}{\sqrt{|\ln\var_n|}}\rightarrow0$, considering the problem satisfied by $\psi_n$ directly and using \eqref{aa436+} to relate the magnitude relationship between boundary derivatives $z_n$ and $\rho_n(\theta)$ to $A>0$ or $A<0$, and deriving a contradiction with the Pohozaev identity \eqref{e18+'} via the strict monotonicity of the function $F_n(t)$, we successfully overcame the nonlinear nature of the $N$-Laplacian $\Delta_N$ and proved the crucial uniform lower bound \eqref{a4'}, which indicates that the limit of non-radial approximate solutions is non-radial solution in $\R^N$ and plays a quite important role in our proof of the bifurcation result in $\R^N$ in Section 5.

\subsection{Classification of solution for the singular case $\al\in(-1,0)$}
When $-1<\al<0$, we also investigate the classification of solutions for the Hardy type singular Liouville equation \eqref{11}. In \cite{N}, V. H. Nguyen proved the classification of solutions for \eqref{11} under the extra assumptions $\int_{\R^N}|x|^{N\alpha}e^{u}\md x\leq N\left(\frac{N^2(  \al+1)}{N-1}\right)^{N-1}\omega_N$ and $\sup\limits_{x\in\R^N} u<+\infty$.

In this paper, we aim to remove the extra assumptions in V. H. Nguyen \cite{N}. For $\al \in (-1, 0)$, we prove the following classification result for \eqref{11}.
\begin{thm}\label{thm11}
Suppose $\al\in(-1,0)$ and $N\geq2$. Let $u \in W_{loc}^{1, N}(\R^N)$ be a weak solution of \eqref{11} with $\int_{\R^N}|x|^{N\alpha}e^{u}\md x<+\infty$. Then
\begin{align}\label{15pp}
	u(x)=\ln\frac{C_{N,\alpha}
		\lambda^{N(  \al+1)}}{\left(1+(\lambda|x|)^{\frac{N(  \al+1)}{N-1}
		}\right)^N},\quad x\in\R^N
\end{align}
for some $\la>0$.
\end{thm}

\begin{rem}
Our Theorem \ref{thm11} completely improved the classification result in \cite{N} by removing all the extra assumptions therein.
\end{rem}

\begin{rem}
Very recently, G. Ciraolo, P. Esposito and X. Li \cite{CEL} proved classification result on finite mass solutions for \eqref{11} by using the ``P-function" method. In Theorem \ref{thm11}, we reproved their result by simply deducing the extra assumptions $\int_{\R^N}|x|^{N\alpha}e^{u}\md x\leq N\left(\frac{N^2(  \al+1)}{N-1}\right)^{N-1}\omega_N$ and $\sup\limits_{x\in\R^N} u<+\infty$ in \cite{N} from the finite mass condition $\int_{\R^N}|x|^{N\al}e^{u}\md x<+\infty$.
\end{rem}

For more related literature on H\'{e}non-Hardy type equations, refer to e.g. \cite{BCG,CDQ0,CPY,CDQ,DDGL,DQ,DQ0,GGN,JX,NR,PS,PT,SSW} and the references therein.

\medskip

The rest of this paper is organized as follows. In Section 2, we present the proofs of Theorems~\ref{thm12} and \ref{c13} for $\alpha > -1$. Section 3 investigates the approximate problem on the ball and establishes bifurcation results for the approximate solutions. Section 4 establishes a series of key uniform estimates (e.g., Propositions \ref{pp42} and \ref{pp44} etc.), indicating that the limit of non-radial approximate solutions is non-radial solution to \eqref{11} in $\mathbb{R}^{N}$. Section 5 is devoted to proving our main result, Theorem \ref{th14}, for the case $\alpha > 0$. Finally, Section 6 establishes the classification results for equation \eqref{11} when $\alpha \in (-1, 0)$, i.e., Theorem \ref{thm11}.

\medskip

\noindent{\bfseries Notations.}
Throughout this paper, we use the notation $B_R := B_R(0)$ to denote the ball of radius $R$ centered at the origin. Furthermore, $c$, $C$, $C'$ and $C_i$ represent various strictly positive constants, whose values may vary from line to line. The notation $a \sim b$ indicates that $C'b \leq a \leq Cb$.

\section{Classification of entire radial solutions and the kernel of the linearized operator at $e^{U_\al}$}

\noindent {\bf Proof of Theorem \ref{thm12}.}
For all radial solution $u$ to PDE \eqref{11}, the corresponding ODE is following
\begin{equation}\label{21q}
		\left\{
		\begin{aligned}
			&-\left( r^{N-1}|u'(r)|^{N-2}u'\right)'=r^{N\alpha+N-1}e^{u}   &\text{in}\ \left(0,+\infty\right), \\
			&\lim\limits_{r\to0}r^{N-1}|u'|^{N-2}u'=0,\quad u(0)=1.
		\end{aligned}
		\right.
	\end{equation}

We use contradiction argument. Suppose $v\ne u$ is another radial solution of \eqref{11} satisfying ODE \eqref{21q}. Integrating both sides of \eqref{21q} from $0$ to $r$ gives
\begin{align}\label{22q}
-u'(r)=\left(\frac{1}{r^{N-1}}\int_0^r t^{N( \al+1)-1}e^{u(t)}\md t\right)^{\frac{1}{N-1}},
\end{align}
and for $v$,
\begin{align}\label{23q}
-v'(r)=\left(\frac{1}{r^{N-1}}\int_0^r t^{N( \al+1)-1}e^{v(t)}\md t\right)^{\frac{1}{N-1}}.
\end{align}

Define
\begin{align}\label{24q}
r_0:=\sup\left\{r\geq0\,\big|\,u(s)=v(s),\ \forall s\leq r\right\}.
\end{align}
By $u(0)=v(0)$ and $v\ne u$, we have $0\leq r_0<+\infty$. Without loss of generality, we may assume that, there exists $\varepsilon>0$ small enough such that $u(r)>v(r)$ and $u'(r)>v'(r)$ for any $r\in(r_0,r_0+\varepsilon)$. Subtracting \eqref{22q} from \eqref{23q} leads to
\begin{align}\label{25q}
0&>-\bigl(u'(r)-v'(r)\bigr)  \\
&=\left(\frac{1}{r^{N-1}}\int_0^r t^{N( \al+1)-1}e^{u(t)}\md t\right)^{\frac{1}{N-1}}-\left(\frac{1}{r^{N-1}}\int_0^r t^{N( \al+1)-1}e^{v(t)}\md t\right)^{\frac{1}{N-1}}>0 \nonumber
\end{align}
for any $r\in(r_0,r_0+\varepsilon)$, which is absurd.

Since $U_{\al}(x)$ is a radial solution to \eqref{11}, all radial solutions of \eqref{11} are unique up to scaling and take the form $U_{\lambda,\al}$. This completes the proof of Theorem \ref{thm12}. \qed

\bigskip

\noindent {\bf Proof of Theorem \ref{c13}.}
First, we establish the following relation between the linearized operators:
\begin{equation}\label{227v}
 \mathcal{L}_{\mathcal{T},\,e^{U_\alpha}}(v) = e^{(N-1)U_\alpha} \mathcal{L}_{\mathcal{F},\,U_\alpha}\left(\frac{v}{e^{U_\alpha}}\right).
\end{equation}
Indeed, since $U_\alpha$ is a solution to the  equation $\mathcal{F}(U_\alpha) = 0$, it follows that $e^{U_\alpha}$ solves the transformed equation, that is, $\mathcal{T}(e^{U_\alpha}) = 0$.

We introduce a perturbation in the direction $v$ by setting
\begin{equation*}
    g_t := e^{U_\alpha} + tv.
\end{equation*}
To ensure that $\ln g_t$ is well-defined, we introduce the normalized perturbation $\phi := \frac{v}{e^{U_\alpha}}$,
which allows us to express $g_t$ as $
    g_t = e^{U_\alpha}(1 + t\phi).$
Assuming $\phi \in L^\infty(\mathbb{R}^N)$, we have $1 + t\phi > 0$ for sufficiently small $|t|$, so that $g_t > 0$. Taking the logarithm then yields $
    \ln g_t = \ln e^{U_\alpha} + \ln(1 + t\phi) = U_\alpha + \ln(1 + t\phi).$

Recalling the operator identity \eqref{fg} and evaluating it at $g_t$, we obtain
\begin{align}\label{eq228}
    \mathcal{T}(g_t)
    &= \mathcal{T}\bigl(e^{U_\alpha}(1 + t\phi)\bigr)
    = \bigl(e^{U_\alpha}(1 + t\phi)\bigr)^{N-1} \mathcal{F}\bigl(U_\alpha + \ln(1 + t\phi)\bigr) \\
    &= e^{(N-1)U_\alpha}(1 + t\phi)^{N-1} \mathcal{F}\bigl(U_\alpha + \ln(1 + t\phi)\bigr)
    =: A(t)B(t), \nonumber
\end{align}
where $A(t)= e^{(N-1)U_\alpha}(1 + t\phi)^{N-1}$, and
$ B(t) = \mathcal{F}\bigl(U_\alpha + \ln(1 + t\phi)\bigr)$.
We compute the derivative of $\mathcal{T}(g_t)$ at $t = 0$. First, note that $A(0) = e^{(N-1)U_\alpha}$ and
\begin{equation*}
    A'(t) = (N-1)e^{(N-1)U_\alpha}(1 + t\phi)^{N-2}\phi,
\end{equation*}
which gives $A'(0) = (N-1)e^{(N-1)U_\alpha}\phi$. For $B(0)$, we have $
    B(0) = \mathcal{F}(U_\alpha + \ln 1) = \mathcal{F}(U_\alpha).$
Since $U_\alpha$ satisfies $\mathcal{F}(U_\alpha) = 0$, it follows that $B(0) = 0$, and hence the term $A'(0)B(0)$ vanishes.

For the term $A(0)B'(0)$, applying the chain rule yields
\begin{equation*}
    B'(0) = \left. \frac{\mathrm{d}}{\mathrm{d}t} \right|_{t=0} \mathcal{F}\bigl(U_\alpha + \ln(1 + t\phi)\bigr) = \mathcal{L}_{\mathcal{F},\,U_\alpha}(\phi).
\end{equation*}
Consequently, we obtain
\begin{equation}\label{eq230}
    A(0)B'(0) = e^{(N-1)U_\alpha} \mathcal{L}_{\mathcal{F},\,U_\alpha}(\phi).
\end{equation}
Combining \eqref{eq230}, we conclude that
\begin{equation*}
   \mathcal{L}_{\mathcal{T},\,e^{U_\alpha}}(v)=\left.\frac{\mathrm{d}}{\mathrm{d}t}\right|_{t=0} \mathcal{T}(e^{U_\alpha} + tv) = e^{(N-1)U_\alpha}\mathcal{L}_{\mathcal{F},\,U_\alpha}(\phi).
\end{equation*}
Recalling $\phi = \frac{v}{e^{U_\alpha}}$, we arrive at \eqref{227v}.
Consequently, we obtain the equivalence
\begin{equation}\label{e227}
    \mathcal{L}_{\mathcal{T},\,e^{U_\alpha}}(v) = 0 \quad\Longleftrightarrow\quad \mathcal{L}_{\mathcal{F},\,U_\alpha}\left(\frac{v}{e^{U_\alpha}}\right) = 0.
\end{equation}
Setting $\phi := \frac{v}{e^{U_\alpha}}$, we see that $v$ solves the linearized transformed equation if and only if $\phi$ satisfies the linearized problem associated with \eqref{16}.
The equivalence \eqref{e227} directly yields the identity for their kernels
\begin{equation}\label{13}
    \ker \mathcal{L}_{\mathcal{T},\,e^{U_\alpha}} = e^{U_\alpha} \ker \mathcal{L}_{\mathcal{F},\,U_\alpha}.
\end{equation}

Combining Theorem \ref{th11} with \eqref{13}, we deduce that
$$ \ker \mathcal{L}_{\mathcal{T},\,e^{U_\alpha}} = \frac{1}{\lr1+|x|^{\frac{N(1+\al)}{N-1}}\rr^{N}}\ker \mathcal{L}_{\mathcal{F},\,U_\alpha},$$
and hence
\begin{align}\label{v229}
	& \ker \mathcal{L}_{\mathcal{T},\,e^{U_\alpha}}=span\{Z(x)\},\quad \text{if}\  \alpha\ne\alpha(k),\ k\in\N,\nonumber\\
	& \ker \mathcal{L}_{\mathcal{T},\,e^{U_\alpha}}=span\left\{Z(x),\, Z_{k,i}(x), \, 1\leq i\leq \frac{(N+2k-2)(N+k-3)!}{(N-2)!k!}\right\}, \quad \text{if}\  \alpha=\alpha(k),\ k\in\N,\nonumber
\end{align}
where $Z(x)=\frac{(N-1)-|x|^{\frac{N(1+\al)}{N-1}}}{\lr1+|x|^{\frac{N(1+\al)}{N-1}}\rr^{N+1}}$, $Z_{k,i}(x)=\frac{|x|^{\frac{1+\al}{N-1}-k}\Phi_{k,i}(x)}{\lr1+|x|^{\frac{N(1+\al)}{N-1}}\rr^{N+1}}$.
\qed

\section{The approximate problem in $B_{\frac1\var}(0)$}
In this section, we study the approximate problem
\begin{equation}\label{31}
\begin{cases}
		-\Delta_N u=|x|^{N\alpha}e^{u} &\text{in}\  B_{\frac1\var}(0), \\
		u=\ln\frac{C_{N,\alpha}\var^{\frac{N^2(  \al+1)}{N-1}}}{\left(1+\var^{\frac{N(  \al+1)}{N-1}}\right)^N} &\text{on}\  \partial B_{\frac1\var}(0),
\end{cases}
\end{equation}
where $\alpha>0$, $B_{\frac{1}{\var}}(0)$ denotes the ball of radius $\frac{1}{\var}$ centered at the origin and we extend $u$ by the constant $U_\al\left(\frac1\var\right)$ outside $B_{\frac1\var}(0)$. We define the solution
\begin{equation}\label{32}
	u_{\var,\alpha}=\left\{
	\begin{aligned}
		&U_\al(x)=\ln\frac{C_{N,\alpha}}{\left(1+|x|^{\frac{N(  \al+1)}{N-1}
				}\right)^N},&  \text{if}\ |x|\leq\frac1\var, \\
		&U_\al\left(\frac1\var\right)=
		\ln\frac{C_{N,\alpha}\var^{\frac{N^2(  \al+1)}{N-1}}}{\left(1+\var^{\frac{N(  \al+1)}{N-1}}\right)^N}, & \text{if}\ |x|>\frac1\var,
	\end{aligned}
	\right.
\end{equation}
where $U_\alpha(x)$ is defined in \eqref{14}. For the sake of convenience, we define $\beta_\al:=U_\al\lr\frac1\var\rr=\ln\frac{C_{N,\alpha}\var^{\frac{N^2(  \al+1)}{N-1}}}{\left(1+\var^{\frac{N(  \al+1)}{N-1}}\right)^N}$.
\begin{lem}\label{lem31}
	For any $\al\geq0$ and any sufficiently small $0<\var<(N-1)^{-\frac{N-1}{N( \al+1)}}$, the function $u_{\var,\alpha}$ is a isolated radial solution of \eqref{31}, which is nondegenerate in the space of radial functions.
\end{lem}
\begin{proof}
One can easily verify that $u_{\var,\alpha}$ is a radial solution of \eqref{31}. 
For any sufficiently small $0<\var<(N-1)^{-\frac{N-1}{N( \al+1)}}$ and $x\in B_{\frac1\var}(0)$, it is straightforward to verify that $u_{\var,\alpha}$ is a radial solution of \eqref{31}. If the linearized problem
\begin{equation}\label{33}
\left\{
\begin{aligned}
&\quad-\textrm{div}(|\nabla U_\alpha|^{N-2}\nabla \phi)-(N-2)
\textrm{div}\left(|\nabla U_\alpha|^{N-4}(\nabla U_\alpha\cdot\nabla \phi)\nabla U_\alpha\right)\\
&=|x|^{N\alpha}e^{U_\al}\phi, &\text{in}\  B_{\frac1\var}(0), \\
&\phi=0, &\text{on}\ \partial B_{\frac1\var}(0),
\end{aligned}
\right.
\end{equation}
has no nontrivial radial solutions, then $u_{\var,\alpha}$ is nondegenerate in the space of radial functions.
Rewriting \eqref{33} in the form of radial coordinates, i.e.,
\begin{equation}\label{34}
	\left\{
	\begin{aligned}
		&-\left( r^{N-1}|U_\al '|^{N-2}\phi'\right)'=\frac{r^{N\al+N-1}}{N-1}e^{U_\alpha}\phi, &r\in\left(0,\frac{1}{\var}\right), \\
		&\phi'(0)=0,\quad\phi\left(\frac1\var\right)=0.
	\end{aligned}
	\right.
\end{equation}
Observe that the function
$$z(r)=\frac{(N-1)-r^{\frac{N( \al+1)}{N-1}}}
{1+r^{\frac{N( \al+1)}{N-1}}}$$
satisfies the linearized equation
\begin{align}\label{35}
	-\left( r^{N-1}|U_\al '|^{N-2}z'\right)'=\frac{r^{N\al+N-1}}{N-1}e^{U_\alpha}z,\quad r\in\R^+,\quad z'(0)=0.
\end{align}
However, we note that $z(\frac1\var)\ne0$ for $0<\var<(N-1)^{-\frac{N-1}{N( \al+1)}}$.

Multiplying \eqref{34} by $z$, \eqref{35} by $\phi$, integrating in $(0,\frac1\var)$, we have
$\phi'\left(\frac1\var\right)=0.$
Furthermore, since $\phi\left(\frac1\var\right)=0$, the classical existence and uniqueness theorem for initial value problems implies $\phi\equiv0$. This demonstrates that \eqref{34} admits no nontrivial radial solutions, proving the radial nondegeneracy of $u_{\var,\al}$.
\end{proof}

\begin{prop}\label{pp32}
Let $\al_n$ and $\var_n$ be sequences such that $\al_n\to\al>0$ and $\var_n\to0$ as $n\to+\infty$.
Then we have
\begin{align}\label{36}
	\|e^{u_{\var_n,\al_n}}-e^{\beta_{\al_n}}-e^{U_\al}\|_X\to0,\quad\quad as \ n\to+\infty.
\end{align}	
\end{prop}
\begin{proof}
Define $f_{\var_n,\alpha_n}=e^{u_{\var_n,\al_n}}-e^{\beta_{\al_n}}$, then $f_{\var_n,\alpha_n}$ satisfies
\begin{equation}\label{47eq}
		\begin{cases}
			-\Delta_N f_{\var_n,\alpha_n}=|x|^{N\alpha_n}(f_{\var_n,\alpha_n}+e^{\beta_{\al_n}})^N-\frac{(N-1)|\nabla f_{\var_n,\alpha_n}|^N}{f_{\var_n,\alpha_n}+e^{\beta_{\al_n}}}, & x\in B_{\frac1{\var_n}}, \\
			f_{\var_n,\alpha_n}=0, & x\in\R^N \backslash B_{\frac1{\var_n}}.
		\end{cases}
	\end{equation}
According to the definition of norm \eqref{110}, we need to prove $\|f_{\var_n,\alpha_n}-e^{U_\al}\|_{1,N}\to0$ and $\|f_{\var_n,\alpha_n}-e^{U_\al}\|_{\gamma}\to0$ as $n\to+\infty$.

We obtain
\begin{align*}
	\int_{\R^N}\left|f_{\var_n,\alpha_n}-e^{U_\al}\right|^N\md x =\int_{B_{\frac{1}{\var_n}}}\left|e^{U_{\al_n}}-e^{\beta_{\al_n}}-e^{U_\al}\right|^N\md x +\int_{\R^N\backslash B_{\frac{1}{\var_n}}}\left| e^{U_\al}\right|^N\md x
\end{align*}	
and
\begin{align*}
	\int_{\R^N}\left|\nabla f_{\var_n,\alpha_n}-\nabla e^{U_\al}\right|^N\md x =\int_{B_{\frac{1}{\var_n}}}\left|\nabla e^{U_{\al_n}}-\nabla e^{U_\al}\right|^N\md x +\int_{\R^N\backslash B_{\frac{1}{\var_n}}}\left|\nabla e^{U_\al}\right|^N\md x .
\end{align*}	
Since $e^{U_\al}\in   W^{1,N}(\R^N)$, we get
\begin{align*}
	\int_{\R^N\backslash B_{\frac{1}{\var_n}}}\left| e^{U_\al}\right|^N\md x\to0 \ \text{and}\ \int_{\R^N\backslash B_{\frac{1}{\var_n}}}\left|\nabla e^{U_\al}\right|^N\md x\to0, \quad\text{as}\ n\to+\infty .
\end{align*}
From Lebesgue's dominated convergence theorem, we get
\begin{align*}
\int_{B_{\frac{1}{\var_n}}}\left| e^{U_{\al_n}}-e^{\beta_{\al_n}}- e^{U_\al}\right|^N\md x\to0, \quad\text{as}\ n\to+\infty
\end{align*}		
and
\begin{align*}
\int_{B_{\frac{1}{\var_n}}}\left|\nabla e^{U_{\al_n}}-\nabla e^{U_\al}\right|^N\md x\to0, \quad\text{as}\ n\to+\infty .
\end{align*}		
By the definition of $L_\gamma^\infty$ with $\gamma\in\left( N(\al+1),\frac{N^2(\al+1)}{N-1}\right)$ and using the mean value theorem, we have
\begin{align}\label{eq49}
\|f_{\var_n,\alpha_n}-e^{U_\al}\|_{\gamma}&=\sup_{\R^N}(1+|x|)^{\gamma}|e^{u_{\var_n,\al_n}}-e^{\beta_{\al_n}}-e^{U_\al}|\\
&\leq\sup_{x\in B_{\frac{1}{\var_n}}}(1+|x|)^\gamma|e^{U_{\al_n}(x)}-e^{U_\al(x)}|
+\sup_{x\in B_{\frac{1}{\var_n}}}(1+|x|)^\gamma e^{U_{\al_n}\lr\frac1{\var_n}\rr}\nonumber\\
&\quad		+\sup_{\R^N\setminus B_{\frac{1}{\var_n}}}\frac{C_{N,\alpha}(1+|x|)^{\gamma}}{\left(1+|x|^{\frac{N(  \al+1)}{N-1}
		}\right)^N}\leq O(|\al-\al_n|)+O\lr\var_n^{\frac{N^2( \al+1)}{N-1}-\gamma}\rr, \nonumber
\end{align}
where we used the fact $\gamma<\frac{N^2( \al+1)}{N-1}$. This completes the proof of Proposition \ref{pp32}.
\end{proof}

\begin{lem}[Pohozaev identity]\label{lem-p}
Let $u$ be a solution to the approximate problem \eqref{31}, then $u$ satisfies the following Pohozaev identity:
\begin{align}\label{poh}
&\quad\frac{N-1}{N}
\int_{\partial B_{1}}\lr-\frac1{\var}\partial_\nu u\lr\frac1{\var}\theta\rr\rr^N\md\theta \\ &=N(1+\al)\int_{\partial B_{1}}\lr-\frac1{\var}\partial_\nu u\lr\frac1{\var}\theta\rr\rr^{N-1}\md\theta
-|\partial B_{1}|\lr\frac1{\var}\rr^{N(1+\al)}e^{\beta_{\alpha}}, \nonumber
\end{align}
where $\nu:=\frac{x}{|x|}$ is the unit outer normal vector on sphere.
\end{lem}
\begin{proof}
Multiplying both sides of equation \eqref{31} by $x \cdot \nabla u$ and integrating over $B_{\frac1{\var}}$, we obtain
\begin{equation}\label{po310}
-\int_{B_{\frac1{\var}}}  \text{div}(|\nabla u|^{N-2}\nabla u) (x \cdot \nabla u) \mathrm{d}x
= \int_{B_{\frac1{\var}}} |x|^{N\al}e^{u} (x \cdot \nabla u) \mathrm{d}x.
\end{equation}

We first calculate the left-hand side of \eqref{po310}. Applying integration by parts yields
\begin{equation}\label{po311}
\begin{aligned}
\text{L.H.S.} &= -\int_{\partial B_{\frac1{\var}}} |\nabla u|^{N-2} (\partial_\nu u) (x \cdot \nabla u) \md\sigma + \int_{B_{\frac1{\var}}} |\nabla u|^{N-2} \nabla u \cdot \nabla(x \cdot \nabla u) \mathrm{d}x.
\end{aligned}
\end{equation}
By direct computations, the vector identity $\nabla(x \cdot \nabla u) = \nabla u + D^2 u\cdot x$ gives
\begin{align*}
|\nabla u|^{N-2} \nabla u \cdot \nabla(x \cdot \nabla u)
= |\nabla u|^N + |\nabla u|^{N-2} \nabla u \cdot D^2 u\cdot x = |\nabla u|^N + \frac{1}{N} x \cdot \nabla\left( |\nabla u|^N \right).
\end{align*}
By the divergence theorem, we derive that
\begin{equation}\label{po312}
\begin{aligned}
&\quad\int_{B_{\frac1{\var}}} |\nabla u|^{N-2} \nabla u \cdot \nabla(x \cdot \nabla u) \mathrm{d}x \\
&= \int_{B_{\frac1{\var}}} |\nabla u|^N \,\mathrm{d}x
+ \frac{1}{N} \int_{\partial B_{\frac1{\var}}} |\nabla u|^N (x \cdot \nu) \,\mathrm{d}\sigma - \int_{B_{\frac1{\var}}} |\nabla u|^N \mathrm{d}x
= \frac{1}{N} \int_{\partial B_{\frac1{\var}}} |\nabla u|^N (x \cdot \nu) \mathrm{d}\sigma.
\end{aligned}
\end{equation}
Substituting \eqref{po312} into \eqref{po311}, the left-hand side of \eqref{po310} simplifies to
\begin{equation}\label{po313}
\text{L.H.S.} = -\int_{\partial B_{\frac1{\var}}} |\nabla u|^{N-2} (\partial_\nu u) (x \cdot \nabla u) \,\mathrm{d}\sigma
+ \frac{1}{N} \int_{\partial B_{\frac1{\var}}} |\nabla u|^N (x \cdot \nu) \,\mathrm{d}\sigma.
\end{equation}

Next, we calculate the right-hand side of \eqref{po310}. Using $e^{u} (x \cdot \nabla u) = x \cdot \nabla\left(e^{u}\right)$ together with integration by parts, we have
\begin{align}\label{po314}
\text{R.H.S.}& = \int_{B_{\frac1{\var}}} |x|^{N\al} x \cdot \nabla\left(e^{u}\right) \,\mathrm{d}x
= \int_{\partial B_{\frac1{\var}}} |x|^{N\al} e^{u} (x \cdot \nu) \,\mathrm{d}S
- \int_{B_{\frac1{\var}}} \operatorname{div}( |x|^{N\alpha}x) e^{u} \,\mathrm{d}x\\
&=\int_{\partial B_{\frac1{\var}}} |x|^{N\alpha} e^{\beta_{\alpha}} (x \cdot \nu) \,\mathrm{d}\sigma
- N(1+\al) \int_{B_{\frac1{\var}}} |x|^{N\alpha} e^{u} \,\mathrm{d}x.\nonumber\\
&=\int_{\partial B_{\frac1{\var}}} |x|^{N\alpha} e^{\beta_{\alpha}} (x \cdot \nu) \,\mathrm{d}\sigma
+N(1+\al) \int_{\partial B_{\frac1{\var}}}|\nabla u|^{N-2}\partial_\nu u \mathrm{d}\sigma .\nonumber
\end{align}
Since the right-hand side of \eqref{31} is strictly positive and $u\equiv\beta_\al$ on the boundary $\partial B_{\frac{1}{\var}}$, the strong maximum principle and Hopf's lemma yield
\begin{equation}\label{eq:15}
u > \beta_{\alpha} \quad \text{in } B_{\frac1{\var}},
\qquad
\partial_\nu u=-|\nabla u|< 0 \quad \text{on } \partial B_{\frac1{\var}}.
\end{equation}
Finally, combining \eqref{po313}--\eqref{eq:15}, and using $x=\frac1{\var}\theta$ and $\md \sigma=\lr\frac1{\var}\rr^{N-1}\md\theta$, we immediately arrive at \eqref{poh}.

The computation above is also valid for general weak $N$-Laplace solutions $u$. Indeed, one first uses the standard strictly elliptic regularization $\mathcal A_\delta(\xi)=(\delta+|\xi|^2)^{\frac{N-2}{2}}\xi$, performs the above calculation for the regularized solutions $u_\delta$, then passes to the limit $\delta\rightarrow0$ and derives \eqref{poh}. This finishes our proof of Lemma \ref{lem-p}.
\end{proof}

\subsection{Convergence of the spectrum}

Consider the linearized problem associated with \eqref{31}, i.e.,
\begin{equation}\label{37}
\begin{cases}
		-\textrm{div}(|\nabla U_\alpha|^{N-2}\nabla v)-(N-2)
		\textrm{div}
		\left(|\nabla U_\alpha|^{N-4}(\nabla U_\alpha\cdot\nabla v)\nabla U_\alpha\right)
		=|x|^{N\alpha}e^{u_{\var,\alpha}}v \quad\,\, \text{in}\  B_{\frac1\var}(0), \\
		v=0 \qquad\quad \text{on}\ \partial B_{\frac1\var}(0).
\end{cases}
\end{equation}
Since $u_{\var,\al}=U_\al(x)$ for $x\in B_{\frac1\var}(0)$, the equation \eqref{37} simplifies to
\begin{equation}\label{38}
\begin{cases}
		-\textrm{div}(|\nabla U_\alpha|^{N-2}\nabla v)-(N-2)
		\textrm{div}
		\left(|\nabla U_\alpha|^{N-4}(\nabla U_\alpha\cdot\nabla v)\nabla U_\alpha\right)
		=|x|^{N\alpha}e^{U_\alpha}v \quad\,\, \text{in}\  B_{\frac1\var}(0), \\
		v=0 \qquad\quad \text{on}\ \partial B_{\frac1\var}(0).
\end{cases}
\end{equation}
The problem \eqref{38} can be decomposed into the radial part and angular part by using the spherical harmonic functions in similar way as in Section 2. Specifically, $v$ solves \eqref{38} if and only if  $v_k(r)=\int_{\mms^{N-1}}v(r,\theta)\Phi_k(\theta)\md\theta$ is a solution of
\begin{equation}\label{39}
\left\{
	\begin{aligned}
&\quad-v_k''(r)-\frac{v_k'(r)}{r}\left(1+\frac{N(N-2)}{N-1}
\frac{( \al+1)}{1+r^{\frac{N( \al+1)}{N-1}}}\right)+
\frac{\la_k}{N-1}\frac{v_k(r)}{r^2}\\
&=\frac{N^3( \al+1)^2}{(N-1)^2}\frac{r^{\frac{N(\al+1)}{N-1}-2}}{(1+r^{\frac{N( \al+1)}{N-1}})^2}v_k(r), \  r\in\left(0,\frac1\var\right),\ v_k\in \mathcal{B}, \\
&v_k'(0)=0\ \ \text{if}\ \ k=0,\ \ \ \ v_k(0)=0\ \ \text{if} \ \ k\geq1,\ \ \ \ \text{and}\ \ v_k\lr\frac1\var\rr=0
	\end{aligned}
	\right.
\end{equation}
for $\la_k=k(N+k-2)$, where $\Phi_k(\theta)$ denotes $k$-th spherical harmonic function, and the function space
\begin{align}\label{23b} \mathcal{B}:=\lb\varphi\in C\left(\Big[0,\frac{1}{\var}\Big)\right)\cap C^{2}\left(\left(0,\frac{1}{\var}\right)\right)\Big|\int_0^{\frac{1}{\var}}
r^{N\al+N-1}|\varphi(r)|^2e^{U_\alpha}\md r<+\infty\rb.\end{align}
Thus, if $v_k$ solves \eqref{39}, then all  eigen-function of \eqref{37} take the form of $v_k(r)\Phi_k(\theta)$. By Lemma~\ref{lem31}, we obtain that, only when $k \neq 0$, \eqref{39} has a solution. Therefore, we consider the following eigenvalue problem
\begin{equation}\label{310}
\left\{
\begin{aligned}
&\quad-z''(r)-\frac{z'(r)}{r}\left(1+\frac{N(N-2)}{N-1}
\frac{( \al+1)}{1+r^{\frac{N( \al+1)}{N-1}}}\right)-\frac{N^3( \al+1)^2}{(N-1)^2}\frac{r^{\frac{N( \al+1)}{N-1}-2}}{(1+r^{\frac{N( \al+1)}{N-1}})^2}z(r)\\
&=\frac{\mu}{N-1}\frac{z(r)}{r^2}, \qquad  r\in\left(0,\frac1\var\right),\\
&z(0)=0=z\left(\frac1\var\right).
\end{aligned}
\right.
\end{equation}
The problem \eqref{310} admits a sequence of eigenvalues $\{\mu_i^\varepsilon(\alpha)\}_{i=1}^\infty$. Consequently, solving equation \eqref{39} is equivalent to finding integer pairs $(i,k)$ with $i, k \geq 1$ and $\alpha > 0$ satisfying
\begin{align}\label{311}
	-\la_k=\mu_i^\var(\al).
\end{align}
Based on the above analysis, we consider the eigenvalues of equation \eqref{310} and establish the following lemma.

\begin{lem}\label{lem33}
	For any $\var>0$ sufficiently small, we have that
	\begin{align}\label{312}
		\mu_1^\var(\al)<0
	\end{align}	
	and
	\begin{align}\label{313}
		\mu_2^\var(\al)>0.
	\end{align}	
\end{lem}
\begin{proof}
A straightforward calculation demonstrates that $\mu_1^{\varepsilon}(\alpha) < 0$ for $\varepsilon > 0$ sufficiently small. Indeed, if $z$ is an  eigen-function associated with the eigenvalue $\mu_1^{\varepsilon}(\alpha)$ of the corresponding equation, i.e.,
\begin{equation}\label{314e}
		\left\{
		\begin{aligned}
			&\quad-z''(r)-\frac{z'(r)}{r}\left(1+\frac{N(N-2)}{N-1}
			\frac{( \al+1)}{1+r^{\frac{N( \al+1)}{N-1}}}\right)-\frac{N^3( \al+1)^2}{(N-1)^2}\frac{r^{\frac{N( \al+1)}{N-1}-2}}{(1+r^{\frac{N( \al+1)}{N-1}})^2}z(r)
\\&=\frac{\mu_1^\var(\al)}{N-1}\frac{z(r)}{r^2}, \qquad  r\in\left(0,\frac1\var\right),\\
			&z(0)=0=z\left(\frac1\var\right).
		\end{aligned}
		\right.
\end{equation}
By the definition of the eigenvalue, $\mu_1^{\varepsilon}(\alpha)$ is the minimum value associated with the following functional
\begin{equation}\label{315}
I(z)=\frac{\int_0^{\frac1\var}r^{N-1}|U_\al'|^{N-2}|z'|^2\md r-\frac{1}{N-1}\int_0^{\frac1\var}r^{N\al+N-1}e^{U_\al}z^2\md r}
{\frac{1}{N-1}\int_0^{\frac1\var}r^{N-3}|U_\al'|^{N-2}z^2\md r},
\end{equation}
i.e.,
$$\mu_1^\var(\al):=\min_{z\in \mathcal{B}} I(z).$$
To prove \eqref{312}, we employ the test function
$$z(r)=\frac{r^{\frac{ \al+1}{N-1}}}
{1+r^{\frac{N( \al+1)}{N-1}}}\eta_\var(r)=:g(r)\eta_\var(r),$$
where $\eta_\var(r)\in C_0^\infty(0,\frac1\var)$ is an appropriately chosen cut-off function.
From \eqref{315}, the proof reduces to showing that the numerator in \eqref{315} is strictly negative, which amounts to verifying the inequality, i.e.,
\begin{align}\label{316}
&\quad\int_0^{\frac1\var}r^{N-1}|U_\al'|^{N-2}|g'|^2\eta_\var^2\md r+\int_0^{\frac1\var}r^{N-1}|U_\al'|^{N-2}|g|^2|\eta_\var'|^2\md r+2\int_0^{\frac1\var}r^{N-1}|U_\al'|^{N-2}gg'\eta_\var\eta'_\var\md r\\
&<\frac{1}{N-1}\int_0^{\frac1\var}r^{N\al+N-1}e^{U_\al}g^2\eta_\var^2\md r\nonumber
\end{align}
and $g(r)$ satisfies the equation
\begin{align}\label{317}
-(r^{N-1}|U_\al'|^{N-2}g')'=\frac{1}{N-1}r^{N\al+N-1}e^{U_\al}g-\left(1+\al\right)^{2} r^{N-3}|U_\al'|^{N-2}g.
\end{align}
Multiplying \eqref{317} by $g\eta_\varepsilon^2$, integrating over $(0, \frac{1}{\varepsilon})$, and substituting into \eqref{316}, we obtain that \eqref{316} is equivalent to
\begin{align}\label{318}
\int_0^{\frac1\var}r^{N-1}|U_\al'|^{N-2}|g|^2|\eta_\var'|^2\md r<
\left(1+\al\right)^{2}\int_0^{\frac1\var}r^{N-3}|U_\al'|^{N-2}|g|^2\eta_\var^2\md r.
\end{align}
Taking
\begin{equation*}
\eta_\var=0\ \text{in} \ (0,\var)\ \text{and}\ \left(\frac{1}{2\var},\frac{1}{\var}\right),\ \eta_\var=1\ \text{in} \ \left(2\var,\frac{1}{4\var}\right), \eta_\var\in(0,1)\ \text{in} \ (\var,2\var)\ \text{and}\ \left(\frac{1}{4\var},\frac{1}{2\var}\right),
\end{equation*}
we can derive \eqref{318} and hence \eqref{312}.

As to \eqref{313}, this follows from the monotonicity of eigenvalues with respect to the domain. Let $\mu = \mu_i(\alpha)$ denote the eigenvalue of the problem
\begin{equation}\label{314}
\left\{
\begin{aligned}
&\quad-z''(r)-\frac{z'(r)}{r}\left(1+\frac{N(N-2)}{N-1}
\frac{( \al+1)}{1+r^{\frac{N( \al+1)}{N-1}}}\right)-\frac{N^3( \al+1)^2}{(N-1)^2}\frac{r^{\frac{N( \al+1)}{N-1}-2}}{(1+r^{\frac{N( \al+1)}{N-1}})^2}z(r)\\
&=\frac{\mu}{N-1}\frac{z(r)}{r^2}, \qquad   r\in\left(0,+\infty\right),\\
&z\in \mathcal{B}.
\end{aligned}
\right.
\end{equation}
Then, it follows from the proof of Theorem 5.1 in \cite{CEL} that $\mu_1(\al)=-(N-1)(\al+1)^{2}$ and $\mu_2(\al)=0$. Then, by the monotonicity of eigenvalues with respect to the domain, we derive $\mu_2^\var(\al)\geq \mu_2(\al)=0$. If $\mu_2^\var(\al)=0$, then we can derive a contradiction with the radial non-degeneracy in Lemma \ref{lem31}. Thus $\mu_2^\var(\al)>0$. This finishes our proof of Lemma \ref{lem33}.
\end{proof}

\begin{lem}\label{lem34}
Let $\al>0$ and $\var_n$ be a sequence such that $\var_n\to0$ as $n\to+\infty$.
Let $\mu_1^{\var_n}(\al)$ be the first eigenvalue for \eqref{310} in $\left(0,\frac1{\var_n}\right)$ related to the exponent $\al$. Then
\begin{align}\label{320}
\mu_1^{\var_n}(\al)\to\mu_1(\al)=-(N-1)(1+\al)^{2},  \quad\quad as \ n\to+\infty.
\end{align}	
Furthermore, the convergence of \eqref{320} is uniform in $\al$ on compact sets of $(0,+\infty)$.

Finally, let $z_n(r)$ be the first positive eigen-function of \eqref{310} with respect to $\mu^{\var_n}_1(\al)$ such that $\|z_n\|_{L^\infty}=1$, we get
\begin{align}
&|z_n'(r)|\leq \frac{C}{r^{\al+2}},\quad\quad |z_n(r)|\leq \frac{C}{r^{\al+1}},\label{321}\\
&z_n(r)\to z(r)=\frac{N}{(N-1)^{\frac{N-1}{N}}}\frac{r^{\frac{ \al+1}{N-1}}}
{1+r^{\frac{N( \al+1)}{N-1}}}\label{322}
\end{align}
uniformly in $r$ on compact sets of $[0,+\infty)$.
\end{lem}
\begin{proof}
The convergence in \eqref{320} can be established through  the domain dependence of eigenvalues or by applying Sturm-Liouville theory.

Let $z_n(r)$ denote the first positive  eigen-function associated with $\mu_1^{\var_n}(\alpha)$ for problem \eqref{310} such that $\|z_n\|_{L^\infty} = 1$. Hence $z_n$ satisfies
\begin{equation}\label{323}
\left\{
\begin{aligned}
&-(r^{N-1}|U_\al'|^{N-2}z_n')'
=
\left(
\frac{r^{N\al+N-1}e^{U_\al(r)}}{N-1}+\frac{\mu_1^{\var_n}(\al)}{N-1}\frac{r^{N-1}|U_\al'|^{N-2}}{r^2}
\right)z_n, \quad  r\in\left(0,\frac{1}{\var_n}\right),\\
&z_n\left(\frac{1}{\var_n}\right)=0,\quad\|z_n\|_{L^\infty}=1.
\end{aligned}
\right.
\end{equation}
 On the other hand, $z$ satisfies the corresponding equation in $(0,\infty)$
\begin{equation}\label{z333}
-(r^{N-1}|U_\al'|^{N-2}z')'
=\left(\frac{r^{N\al+N-1}e^{U_\al(r)}}{N-1}+\frac{\mu_1(\al)}{N-1}\frac{r^{N-1}|U_\al'|^{N-2}}{r^2}
\right)z,\end{equation}
where $\mu_1(\al)=-(N-1)(\al+1)^2$.
By the monotonicity of eigenvalues with respect to the domain,
$$
\mu_1^{\var_n}(\al)\geq\mu_1(\al).
$$
Define
\begin{align*}
W_n(r)=
r^{N-1}|U_\al'|^{N-2}\bigl(z(r)z_n'(r)-z'(r)z_n(r)\bigr).
\end{align*}
Thus we obtain
\begin{align}\label{w334}
W_n'(r)=z(r)(r^{N-1}|U_\al'|^{N-2}z_n')'(r)-z_n(r)(r^{N-1}|U_\al'|^{N-2}z')'(r).
\end{align}
Substituting the equations for $z_n$ and $z$, namely \eqref{323} and \eqref{z333}, into \eqref{w334}, we get
\begin{align*}
W_n'(r)=
-\frac{\mu_1^{\var_n}(\al)-\mu_1(\al)}{N-1}
\frac{r^{N-1}|U_\al'|^{N-2}}{r^2}z(r)z_n(r)
\leq0,
\qquad 0<r<\frac{1}{\var_n}.
\end{align*}
The admissible behavior of the eigenfunctions at the origin gives $\lim\limits_{r\to0}W_n(r)=0.$
Consequently, $W_n(r)\leq0$ for $0<r<\frac{1}{\var_n}$, and hence $\left(\frac{z_n}{z}\right)'(r)=\frac{W_n(r)}{r^{N-1}|U_\al'|^{N-2}z(r)^2}\leq0$. Since
$z(1)=\frac{N}{2(N-1)^{\frac{N-1}{N}}}$ and $\|z_n\|_{L^\infty(0,\frac{1}{\var_n})}=1$, we obtain
$$
\frac{z_n(r)}{z(r)}
\leq
\frac{z_n(1)}{z(1)}
\leq\frac{2(N-1)^{\frac{N-1}{N}}}{N},
\qquad1\leq r<\frac{1}{\var_n}.
$$
We conclude that
\begin{equation}\label{Z}
z_n(r)\leq\frac{2(N-1)^{\frac{N-1}{N}}}{N}z(r)\leq\frac{2(N-1)^{\frac{N-1}{N}}}{N}\frac{1}{r^{1+\al}},
\qquad1\leq r<\frac{1}{\var_n}.
\end{equation}

Observe that for sufficiently large $r > 0$ and $\alpha > 0$, the following inequality holds
\begin{align}\label{324}
\frac{r^{N\al+N-1}e^{U_\al(r)}}{N-1}+\frac{\mu_1^{\var_n}(\al)}{N-1}\frac{r^{N-1}|U_\al'|^{N-2}}{r^2}<0,
\end{align}
where we have used the fact that $\mu_1^{\var_n}(\alpha) < 0$ for all sufficiently large $n$.
	
Integrating \eqref{323} on $\left(r, \frac{1}{\varepsilon_n}\right)$, we derive
\begin{align}\label{325}
&\quad r^{N-1}|U'_\al(r)|^{N-2}z'_n(r)\\
&=\left(\frac{1}{\var_n}\right)^{N-1}
\Big|U'_\al\left(\frac{1}{\var_n}\right)\Big|^{N-2}z'_n\left(\frac{1}{\var_n}\right)+\frac{\mu_1^{\var_n}(\al)}{N-1}\int_r^{\frac{1}{\var_n}}t^{N-3}|U'_\al(t)|^{N-2}z_n(t)\md t\nonumber\\
&\quad+\frac{1}{N-1}\int_r^{\frac{1}{\var_n}}t^{N\al+N-1}e^{U_\al(t)}z_n(t)\md t.\nonumber
\end{align}
Noting that $0\leq z_n(r)\leq1$ and $z'_n\left(\frac{1}{\var_n}\right)\leq0$, from \eqref{324}, we get
\begin{align}\label{326}
z'_n(r)<0\quad\quad\text{for}\ r\ \text{large\ enough}.
\end{align}
Integrating \eqref{323} on $\left( 0,r\right)$, we obtain
\begin{align}\label{327}
-r^{N-1}|U'_\al(r)|^{N-2}z'_n(r)=\frac{\mu_1^{\var_n}(\al)}{N-1}\int_0^{r}t^{N-3}|U'_\al(t)|^{N-2}z_n(t)\md  t+\frac{1}{N-1}\int_0^{r}t^{N\al+N-1}e^{U_\al(t)}z_n(t)\md t.
\end{align}
Since $0\leq z_n(r)\leq1$ and $z'_n(r)<0$ for $r$ large enough, from \eqref{Z} and \eqref{327}, we obtain that
\begin{equation}\label{328}
|z_n'(r)|\leq \frac{C}{r^{2+\al}}.
\end{equation}
Consequently, we get \eqref{321}.

By \eqref{310} and \eqref{321}, we obtain $|z''_n(r)|\leq Cr^{-2-\frac{N(\al+1)}{N-1}}$. Applying the Arzel$\acute{a}$-Ascoli's Theorem, it follows that $z_n\to z$ weakly in $ \mathcal{B}$ and uniformly w.r.t $r$ on compact sets of $[0,+\infty)$. Utilizing \eqref{321} once more, we are able to pass to the limit in \eqref{310} and obtain that $z$ is the solution to \eqref{314} associated with the eigenvalue $\mu_1(\al)$. Furthermore, $z\not\equiv0$, in fact, from \eqref{321}, the maximum point of $z_n(r)$ converges to a point $r_0\in[0,+\infty)$ and $|z(r_0)|=1$ due to the uniform convergence.

Finally, we establish the uniform convergence of $\mu_1^{\var_n}(\al)\to\mu_1(\al)$ in $\al$ on compact sets. By multiplying \eqref{323} by $zr^{N-1}|U'_\al|^{N-2}$ and integrating on $\left(0,\frac{1}{\var_n}\right)$, multiplying \eqref{314} by $z_nr^{N-1}|U'_\al|^{N-2}$ and integrating on $\left(0,\frac{1}{\var_n}\right)$, and then subtracting, we derive
\begin{align*}
-\left(\frac{1}{\var_n}\right)^{N-1}	\Big|U'_\al\left(\frac{1}{\var_n}\right)\Big|^{N-2}z'_n\left(\frac{1}{\var_n}\right) z\left(\frac{1}{\var_n}\right)&=\frac{\left(\mu_1^{\var_n}(\al)-\mu_1(\al)\right)}{N-1}\int_0^{\frac{1}{\var_n}}t^{N-3}|U'_\al(t)|^{N-2}z_n(t)z(t)\md t.
\end{align*}
From \eqref{321}, we get
\begin{align*}
-\left(\frac{1}{\var_n}\right)^{N-1}
\Big|U'_\al\left(\frac{1}{\var_n}\right)\Big|^{N-2}z'_n\left(\frac{1}{\var_n}\right) z\left(\frac{1}{\var_n}\right)=O\left(\var_n^{2(\al+1)}\right),
\end{align*}
as $n\to+\infty$, uniformly in $\al$ on compact subsets of $(0,+\infty)$, while
\begin{align*}
\int_0^{\frac{1}{\var_n}}t^{N-3}|U'_\al(t)|^{N-2}z_n(t)z(t)\md t\to\int_0^{+\infty}t^{N-3}|U'_\al(t)|^{N-2}z^2(t)\md t,
\end{align*}
as $n\to+\infty$, uniformly in $\al$ on compact subsets of $(0,+\infty)$. This implies that
\begin{align*}
\sup_{\al\in K}|\mu_1^{\var_n}(\al)-\mu_1(\al)|=o(1),
\end{align*}
as $n\to+\infty$, for any compact set $K\subset(0,+\infty)$. This completes the proof.
\end{proof}

\begin{prop}\label{pp35}
Let $\al>0$ and $\var_n$ be a sequence such that $\var_n\to0$, as $n\to+\infty$. Let $\mu_1^{\var_n}(\al)$ be the first	eigenvalue of \eqref{323} in $\left(0,\frac1{\var_n}\right)$. Then we have that
\begin{align}\label{329}
\frac{\partial \mu_1^{\var_n}(\al)}{\partial \al }\to-2(N-1)( \al+1)=\frac{\partial \mu_1(\al)}{\partial \al }<0
\end{align}	
uniformly on the compact set $K\subset(0,+\infty)$. Furthermore, for any integer $k\geq1$, the equation
\begin{align}\label{330}
-k(N+k-2)=-\la_k=\mu_1^{\var_n}(\al)
\end{align}	
has only one solution $\alpha_k^n$ and
\begin{align}\label{331}
\alpha_k^n\to \alpha(k), \quad as\ n\to+\infty,
\end{align}
where $\alpha(k)=\frac{\sqrt{k(N-1)(k+N-2)}}{N-1}-1$.	
\end{prop}
\begin{proof}
From known results in \cite{K}, we conclude that if $z_n$ denotes the first  eigen-function corresponding to \eqref{323}, then $\frac{\partial z_n}{\partial\al}$ and $\frac{\partial \mu_1^{\var_n}(\al)}{\partial\al}$ are continuous functions. Differentiating both sides of \eqref{323} with respect to $\al$, we obtain
\begin{equation}\label{332}
\left\{\begin{aligned}
&\quad-\left(\frac{\partial z_n}{\partial\al}\right)''-\frac{1}{r}\left(\frac{\partial z_n}{\partial\al}\right)'\left(1+\frac{N(N-2)}{N-1}
\frac{( \al+1)}{1+r^{\frac{N( \al+1)}{N-1}}}\right)\\&
-\frac{z'_n}{r}\frac{\partial}{\partial \al}\left(
\frac{N(N-2)}{N-1}
\frac{( \al+1)}{1+r^{\frac{N( \al+1)}{N-1}}}\right)
-\frac{N^3( \al+1)^2}{(N-1)^2}\frac{r^{\frac{N( \al+1)}{N-1}-2}}{(1+r^{\frac{N( \al+1)}{N-1}})^2}\frac{\partial z_n}{\partial\al}\\
&-\frac{\partial}{\partial \al}\left(\frac{N^3( \al+1)^2}{(N-1)^2}\frac{r^{\frac{N( \al+1)}{N-1}-2}}{(1+r^{\frac{N( \al+1)}{N-1}})^2}\right) z_n\\
&=\frac{\partial \mu_1^{\var_n}(\al)}{\partial \al}\frac{z_n}{(N-1)r^2}+\frac{\mu_1^{\var_n}(\al)}{(N-1)r^2}
\frac{\partial z_n}{\partial \al},
\qquad  r\in\left(0,\frac1{\var_n}\right),\\
&\left(\frac{\partial z_n}{\partial \al}\right)\left(\frac1{\var_n}\right)=0.
\end{aligned}
\right.
\end{equation}
Multiplying \eqref{323} by $\frac{\partial z_n}{\partial \al}$ and \eqref{332} by $z_n$, integrating and subtracting, we get
\begin{align}\label{333e}
&\quad	-\int_0^{\frac1{\var_n}}\frac{\partial}{\partial \al}\left(\frac{N^3( \al+1)^2}{N-1}\frac{r^{\frac{N( \al+1)}{N-1}-2}}{(1+r^{\frac{N( \al+1)}{N-1}})^2}\right) z_n^2r^{N-1}|U'_\al|^{N-2}\md r\\
&-\int_0^{\frac1{\var_n}}\frac{\partial}{\partial \al}\left(
\frac{N(N-2)( \al+1)}{1+r^{\frac{N( \al+1)}{N-1}}}\right) z_n z'_nr^{N-2}|U'_\al|^{N-2}\md r=\frac{\partial \mu_1^{\var_n}(\al)}{\partial \al}\int_0^{\frac1{\var_n}}z_n^2r^{N-3}|U'_\al|^{N-2}
\md r.\nonumber
\end{align}
From Lemma \ref{lem34}, we are able to pass to the limit in \eqref{333e} and deduce that, as $n\rightarrow+\infty$,
\begin{align}\label{334e}
\frac{\partial \mu_1^{\var_n}(\al)}{\partial \al}\to&-\frac{\int_0^{+\infty}\frac{\partial}{\partial \al}\left(\frac{N^3( \al+1)^2}{N-1}\frac{r^{\frac{N( \al+1)}{N-1}-2}}{(1+r^{\frac{N( \al+1)}{N-1}})^2}\right) z^2r^{N-1}|U'_\al|^{N-2}\md r}{\int_0^{+\infty}z^2r^{N-3}|U'_\al|^{N-2}\md r}\\
&\quad-\frac{\int_0^{+\infty}\frac{\partial}{\partial \al}\left(\frac{N(N-2)( \al+1)}{1+r^{\frac{N( \al+1)}{N-1}}}
	\right) z z'r^{N-2}|U'_\al|^{N-2}\md r}{\int_0^{+\infty}z^2r^{N-3}|U'_\al|^{N-2}\md r}\nonumber\\
&=\frac{\partial \mu_1(\al)}{\partial \al}=-2(N-1)( \al+1)\nonumber
\end{align}
uniformly on compact sets of $[0,+\infty)$, which completes the proof of \eqref{329}.

Lastly, due to $\frac{\partial \mu_1^{\var_n}(\al)}{\partial \al}<0$ and $\mu_1^{\var_n}(\al)\to\mu_1(\al)$, we infer that for any integer $k\geq1$, the equation \eqref{330} has exactly one root $\al=\alpha_k^n$. Recalling that $\mu_1(\al)=-(N-1)( \al+1)^2$ and that $\al=\alpha(k)$ is the solution to the algebraic equation
\begin{align}\label{335e}
(N-1)( \al+1)^2=\la_k=k(N+k-2),
\end{align}
we consequently establish that $\alpha(k)=\lim\limits_{n\to+\infty}\alpha_k^n$. This completes the proof.
\end{proof}

\begin{rem}\label{morse-index}
	Based on identities \eqref{39}, \eqref{311} and \eqref{330}, by Lemmas \ref{lem31} and \ref{lem33}, we deduce that the solution $ u_{\varepsilon_n,\al}$ to problem \eqref{31} with parameter $\varepsilon_n$ is degenerate if and only if $\alpha = \alpha_k^n$ for $k = 1,2,\cdots$, i.e., the linearized equation of \eqref{31} at $u_{\varepsilon_n,\al}$ only has trivial solution $0$ if and only if $\alpha\neq\alpha_k^n$. Moreover, at each points $\alpha_k^n$, the Morse index of $u_{\var_n,\alpha}$ changes.

In particular, over the interval $(\alpha_k^n - \delta, \alpha_k^n + \delta)$ where $\delta > 0$ is sufficiently small, the Morse index of $u_{\varepsilon_n,\al}$ increases by exactly $\text{dim} \ \operatorname{Ker}(\Delta_{\mathbb{S}^{N-1}} + \la_k)$.
\end{rem}

\subsection{The bifurcation result in the ball}

In this section, we establish bifurcation results of \eqref{31} in the ball, using some ideas from \cite{GGPS}. To analyze the bifurcation behavior, we introduce the following notation.

We use $u_{n,\al}$ to represent the radial solution of \eqref{31} corresponding to the exponent $\al$, for $\var=\var_n$, and let $B_{\frac{1}{\var_n}}$ be the ball centered at the origin with radius $\frac{1}{\var_n}$. We define the set
\begin{align}\label{336}
	\mathcal{A}(n)=\Big\{&(\al,  u_{n,\al})\in(0,+\infty)\times C^{1,\eta}(\overline B_{\frac{1}{\var_n}}) \ \text{such\  that}\   u_{n,\al} \\ &\text{is\ the\ radial\ solution\ of}\ \eqref{31} \ \text{defined\ in}\ \eqref{32}\Big\}\nonumber.
\end{align}
For the given curve $\mathcal{A}(n)$, we say that a point $(\alpha_j, u_{n,\alpha_j}) \in \mathcal{A}(n)$ is a non-radial bifurcation point, if for every neighborhood of $(\alpha_j, u_{n,\alpha_j})$ in $(0, +\infty) \times  C^{1,\eta}(\overline{B}_{\frac{1}{\var_n}})$, there exists a point $(\alpha, v_{n,\alpha})$ in this neighborhood such that $v_{n,\alpha}$ is a non-radial solution of \eqref{31}.

The subspace $\mathcal{Q}_n \subset  C^{1,\eta}(\overline{B}_{\frac{1}{\var_n}})$ is defined by
\begin{align}\label{337}
	\mathcal{Q}_n=\Big\{&h\in  C^{1,\eta}(\overline B_{\frac{1}{\var_n}}) \ \text{such\  that}\ h(x_1,\cdots,x_N)=h\left( g(x_1,\cdots,x_{N-1}),x_N\right) \\ &\text{for\ any}\ g\in\mathcal{O}(N-1)\Big\}\nonumber,
\end{align}
where $\mathcal{O}(N-1)$ is the orthogonal group in $\R^{N-1}$.

Now we need the following Theorem.

\begin{thm}[\textbf{\cite{AM}}]
	\label{the36}
	Let $T \in C^1(D, X)$ be compact and such that $1$ is not a characteristic value of $T'(x_0)$ $($i.e., $S'(x_0)$ is invertiable$)$ for some $x_0 \in D$. Then, setting $S(x) = x - T(x)$ and $S(x_0) = p$, one has that $x_0$ is an isolated solution of $S(x) = p$ and there holds
	$$
	i(S, x_0) = (-1)^{\beta},
	$$
	where $i(S, x_0)$ is the Morse index and $\beta$ is the sum of the algebraic multiplicities of all the characteristic values of $T'(x_0)$ contained in $(0, 1)$.
\end{thm}

\begin{thm}\label{thm38}
For any $\al\in(0,+\infty)$, let the operator $\mathcal{L}^n(\al,h):(0,+\infty)\times\mathcal{Q}_n\to \mathcal{Q}_n$ be defined by $\mathcal{L}^n(\al,h):=(-\Delta_N)^{-1}(|x|^{N\al}e^{h})$, then the operator $I - \mathcal{L}_{h}^n(\al,  u_{n,\al}):\mathcal{Q}_n\to \mathcal{Q}_n$  is invertible for $\al \neq\alpha_k^n$, $k=1,2,\cdots$, where $\alpha_k^n$ are given by \eqref{331} and
\begin{align*}
\mathcal{L}_{h}^n(\al,  u_{n,\al}):=\lim_{t\to0}\frac{(-\Delta_N)^{-1}(|x|^{N\al}e^{   u_{n,\al}+th})-  u_{n,\al}}{t}.
\end{align*}	
\end{thm}
\begin{proof}
Suppose $v\in \text{Ker}\left( I - \mathcal{L}_{h}^n(\al,  u_{n,\al})\right)$ i.e.,
\begin{align*}
\mathcal{L}_{v}^n(\al,  u_{n,\al}):=\lim_{t\to0}\frac{(-\Delta_N)^{-1}(|x|^{N\al}e^{   u_{n,\al}+tv})-  u_{n,\al}}{t}=v.
\end{align*}	
Thus, for any test function $\varphi\in C_c^\infty(B_{\frac1{\var_n}})$, we have
\begin{align}\label{339q}
&\quad\int_{B_{\frac1{\var_n}}}\left(|\nabla   u_{n,\al}+t\nabla v+o(t)\nabla w|^{N-2}(\nabla   u_{n,\al}+t\nabla v+o(t)\nabla w)-|\nabla   u_{n,\al}|^{N-2}\nabla   u_{n,\al}\right)\cdot\nabla \varphi\md x\\
&=\int_{B_{\frac1{\var_n}}}|x|^{N\al}\left( e^{   u_{n,\al}+tv}-e^{   u_{n,\al}}\right)\varphi\md x\nonumber,
\end{align}
where $w\in X$. On the one hand, from the LHS of the equation \eqref{339q}, we have
\begin{align}\label{340q}
&\quad\int_{B_{\frac1{\var_n}}}\left(|\nabla   u_{n,\al}+t\nabla v+o(t)\nabla w|^{N-2}(\nabla   u_{n,\al}+t\nabla v+o(t)\nabla w)-|\nabla u_n|^{N-2}\nabla   u_{n,\al}\right)\cdot\nabla \varphi\md x\\
&=\int_{B_{\frac1{\var_n}}}\int_0^1(N-2)|\nabla   u_{n,\al}+st\nabla v+o(t)s\nabla w|^{N-4}\left(\left(\nabla   u_{n,\al}+st\nabla v+o(t)s\nabla w\right)\cdot\nabla \varphi\right)\nonumber\\
&\quad\times\left(\left(\nabla   u_{n,\al}+st\nabla v+o(t)s\nabla w\right)\cdot\left( t\nabla v+o(t)\nabla w\right)\right)\nonumber\\
&\quad+|\nabla   u_{n,\al}+st\nabla v+o(t)s\nabla w|^{N-2}\left( t\nabla v+o(t)\nabla w\right)\cdot\nabla \varphi\md s\md x.\nonumber
\end{align}
On the other hand, from the RHS of the equation \eqref{339q}, we obtain
\begin{align}\label{341q}
\int_{B_{\frac1{\var_n}}}|x|^{N\al}\left( e^{u_{n,\al}+tv}-e^{u_{n,\al}}\right)\varphi\md x=t\int_{B_{\frac1{\var_n}}}|x|^{N\al}e^{u_{n,\al}+\theta tv}v\varphi\md x,
\end{align}
where $\theta\in(0,1)$. Combining \eqref{340q}, \eqref{341q}  with \eqref{339q}, and letting $t\to0$, we get,
\begin{align}\label{342qq}
&\quad\int_{B_{\frac1{\var_n}}}|\nabla   u_{n,\al}|^{N-2}\nabla v\cdot\nabla\varphi+(N-2)
|\nabla   u_{n,\al}|^{N-4}(\nabla   u_{n,\al}\cdot\nabla v)\nabla   u_{n,\al}\cdot\nabla\varphi\md x\\
&=\int_{B_{\frac1{\var_n}}}|x|^{N\alpha}e^{  u_{n,\al}}v\varphi\md x,\qquad\forall\ \varphi\in C_c^\infty(B_{\frac1{\var_n}}),\nonumber
	\end{align}
that is, $v$ weakly solves \eqref{38}, which is the linearized equation of \eqref{31} at $  u_{n,\al}$. By Lemmas \ref{lem31}, \ref{lem33} and Remark \ref{morse-index}, we know that, if $\alpha \ne\alpha_k^n$ for $k = 1,2,\cdots $, then $  u_{n,\al}$ is non-degenerate and hence $v=0$. Therefore, $\text{Ker}\left( I - \mathcal{L}_{h}^n(\al,  u_{n,\al})\right)=\{0\}$, i.e., $I - \mathcal{L}_{h}^n(\al,  u_{n,\al})$ is invertible for $\al \neq\alpha_k^n$, for $k = 1,2,\cdots $. This completes our proof.
\end{proof}

\begin{thm}\label{th36}
Fix $n\in \N$ and let $\alpha_k^n$ be given by \eqref{330}. Then the points $(\alpha_k^n, u_{n,\alpha_k^n})$ are nonradial bifurcation points of the curve $\mathcal{A}(n)$. Furthermore, these nonradial solutions bifurcating from $(\alpha_k^n, u_{n,\alpha_k^n})$ lie in the space $\mathcal{Q}_n$ and hence are $\mathcal{O}(N - 1)$-invariant.
\end{thm}
\begin{proof}
From \cite{SW1}, we have the spectral properties of the Laplace-Beltrami operator $\Delta_{\mathbb{S}^{N-1}}$. Each eigenvalue $\lambda_k$ ($k \geq 1$) has a one-dimensional eigenspace $V_k$, $V_k = \operatorname{span}\{\phi_k\}$ where $\phi_k$ is $\mathcal{O}(N-1)$-invariant. This spectral structure implies that
\begin{align}\label{338t}
	m(\alpha_k^n+\delta)-m(\alpha_k^n-\delta)=1,
\end{align}
if $\delta>0$ is sufficiently small, where $m(\al)$ denotes the Morse index of the radial solution $  u_{n,\al}$ in the space $\mathcal{Q}_n$.

Consider the operator $\mathcal{L}^n(\alpha,h): (0,+\infty) \times \mathcal{Q}_n \to \mathcal{Q}_n$ defined by
$$
\mathcal{L}^n(\alpha,h) := (-\Delta_N)^{-1}\left(|x|^{N\alpha}e^{h }\right).
$$
The operator $\mathcal{L}^n$ is compact for each fixed $\alpha \in (0,+\infty)$ and $\alpha \mapsto \mathcal{L}^n(\alpha,\cdot)$ is continuous.

Assume, contrary to claim, that $(\alpha_k^n,  u_{n,\alpha_k^n})$ is not a nonradial bifurcation point. Define the operator
$$
T^n(\alpha,h) := h - \mathcal{L}^n(\alpha,h).
$$
By Lemma~\ref{lem31}, the solution $  u_{n,\al}$ is a isolated radial solution for any $\alpha$. Consequently, there exists $\delta_0 > 0$ such that for all $\delta \in (0, \delta_0)$ and $\rho \in (0, \delta_0)$, the following holds
\begin{align}\label{339t}
	&T^n(\al,h)\ne0,\quad\quad \forall \,\, h\in\mathcal{Q}_n\quad\text{such\ that}\quad\|h-  u_{n,\al}\|_{\mathcal{Q}_n}\leq \rho \\
	&\quad\text{and}\quad h\ne   u_{n,\al},\quad \forall\al\in(\alpha_k^n-\delta,\alpha_k^n+\delta).\nonumber
\end{align}
Moreover, we may choose $\delta_0$ sufficiently small such that the interval $[\alpha_k^n - \delta, \alpha_k^n + \delta]$ contains no other points $\alpha_j^n$ with $j \neq k$.

Define the set $\mathcal{P}$ by
$$\mathcal{P}:=\{(\al,h)\in[\alpha_k^n-\delta,\alpha_k^n+\delta]\times\mathcal{Q}_n:\|h-  u_{n,\al}\|_{\mathcal{Q}_n}\leq\rho\}.$$
Since $T^n(\alpha,\cdot) = \mathrm{id} - \mathcal{L}^n(\alpha,\cdot)$ is a compact perturbation of the identity, we consider the Leray-Schauder degree $\deg(T^n(\alpha,\cdot), \mathcal{P}_\alpha, 0)$, where
$\mathcal{P}_\al:=\{h\in\mathcal{Q}_n\ \text{such\ that}\ (\al,h)\in\mathcal{P}\}$. From \eqref{339t}, the equation $T^n(\alpha,h) = 0$ has no solutions on $[\alpha_k^n - \delta, \alpha_k^n + \delta] \times \partial\mathcal{P}_\alpha$. Consequently, by homotopy invariance of the degree, we have
\begin{align}\label{340t}
\deg(T^n(\al,\cdot),\mathcal{P}_\al,0) \quad\text{is\ constant\ on}\ [\alpha_k^n-\delta,\alpha_k^n+\delta].
\end{align}
	
By Theorem~\ref{thm38}, the operator $I - \mathcal{L}_{h}^n(\alpha,  u_{n,\al}): (0,+\infty) \times \mathcal{Q}_n \to \mathcal{Q}_n$ is invertible at both $\alpha_k^n - \delta$ and $\alpha_k^n + \delta$. Moreover, Theorem~\ref{the36} yields the degree
\begin{align*}
	\deg(T^n(\alpha_k^n - \delta,\cdot),\mathcal{P}_{\alpha_k^n - \delta},0) &= (-1)^{m(\alpha_k^n - \delta)}\end{align*}
	and
	\begin{align*}\deg(T^n(\alpha_k^n + \delta,\cdot),\mathcal{P}_{\alpha_k^n + \delta},0) &= (-1)^{m(\alpha_k^n + \delta)}.
\end{align*}
The choice of $\alpha_k^n$ and the space $\mathcal{Q}_n$, combined with \eqref{338t}, implies
$$
\deg(T^n(\alpha_k^n + \delta,\cdot),\mathcal{P}_{\alpha_k^n + \delta},0) = -\deg(T^n(\alpha_k^n - \delta,\cdot),\mathcal{P}_{\alpha_k^n - \delta},0).
$$
This contradicts \eqref{340t}. Consequently, we conclude that $(\alpha_k^n, u_{n,\alpha_k^n})$ is a nonradial bifurcation point, with the bifurcating solutions being nonradial. This completes our proof.
\end{proof}

\begin{thm}\label{thm39}
Let $k$ be an even integer and $\alpha_k^n$ be defined by equation \eqref{331}. Then, there exist $\left[\frac{N}{2}\right]$ distinct nonradial solutions of \eqref{31} that bifurcate from the pair $(\alpha_k^n, u_{n, \alpha_{k}^n})$, where $[x]$ denotes the largest integer that is less than or equal to $x$.
\end{thm}

\begin{proof}
Consider the subgroups $\mathcal{G}_l$ of $\mathcal{O}(N)$ defined by
\begin{align}\label{347qq}
	\mathcal{G}_l = \mathcal{O}(l) \times \mathcal{O}(N - l) \quad \text{for } 1 \leq l \leq \left[\frac{N}{2}\right].
\end{align}

For even integers $k$, the eigenspace $V_{k,l}$ corresponding to eigenvalue $\lambda_k$ of the Laplace-Beltrami operator on $\mathbb{S}^{N-1}$ satisfies $\mathcal{G}_l$-invariant and $\dim V_{k,l} = 1$,
see \cite{SW2}(also see \cite{W}).

Let $\mathcal{Q}_n^l$ denote the subspace of $ C^{1,\eta}(\overline{B}_{\frac{1}{\var_n}})$ consisting of all $\mathcal{G}_l$-invariant functions, where $\mathcal{G}_l$ is defined in \eqref{347qq}. Similar to the expression in \eqref{338t}, we can obtain that
$$m^l(\alpha_k^n + \delta) - m^l(\alpha_k^n - \delta) = 1,$$
for sufficiently small $\delta > 0$. Here $m^l(\alpha)$ represents the Morse index of the radial
solution $ u_{n,\al}$ within the space $\mathcal{Q}_n^l$.

Applying the same argument as in the proof of Theorem~\ref{th36}, we conclude that $(\alpha_k^n,  u_{n,\alpha_k^n})$ is a bifurcation point. Furthermore, the bifurcating solutions maintain $\mathcal{G}_l$-invariance.

Suppose there exists a solution $h$ that is invariant under both the actions of two groups $\mathcal{G}_{l_1}$ and $\mathcal{G}_{l_2}$ for $l_1 \neq l_2$. Then by the result in \cite{SW2}, $h$ must be radial. However, the radial solutions $  u_{n,\al}$ are isolated. This leads to a contradiction.

As a consequence, there exist at least $\left[\frac{N}{2}\right]$ distinct nonradial solutions to equation \eqref{31} bifurcating from the point $(\alpha_k^n,  u_{n,\alpha_k^n})$.
\end{proof}

Let $\Upsilon_n$ denote the closure in $[0,+\infty) \times \mathcal{Q}_n$ of all solutions to $T^n(\alpha,h) = 0$ different from $  u_{n,\al}$, i.e.,
\begin{equation}\label{338}
\Upsilon_n := \overline{\Big\{(\alpha,h) \in [0,+\infty) \times \mathcal{Q}_n \ \text{such\ that}\ T^n(\al,h)=0\ \text{with}\ h\ne u_{n,\al}\Big\}},
\end{equation}
where $\mathcal{Q}_n$ is defined in \eqref{337} ($\mathcal{Q}_n$ can also be replaced by $\mathcal{Q}_n^l$ defined in the proof of Theorem \ref{thm39}) and $T^n(\alpha,h)$ defined in the proof of Theorem  \ref{th36}.

If $(\alpha_k^n,  u_{n,\alpha_k^n}) \in \mathcal{A}(n)$ is a nonradial bifurcation point, we have
$(\alpha_k^n,  u_{n,\alpha_k^n}) \in \Upsilon_n.$

We denote $\mathcal{C}(\alpha_k^n)\subset\Upsilon_n$ the closed connected component of $\Upsilon_n$ which
contains $(\al^n_k,u_{n,\al^n_k})$ and it is maximal with respect to the inclusion.

\begin{thm}\label{th38}
	Let $\alpha_k^n$ be given by equation \eqref{330} and let $\mathcal{C}(\alpha_k^n)$ be the closed connected component of $\Upsilon_n$ which contains $(\al^n_k,u_{n,\al^n_k})$ and it is maximal with respect to the inclusion. Then either\\
	(i) $\mathcal{C}(\alpha_k^n)$ is unbounded in $[0,+\infty)\times\mathcal{Q}_n$, or \\
	(ii) there exists $\al_i^n$ with $i\ne k$ such that $(\al^n_i,u_{n,\al^n_i})\in\mathcal{C}(\alpha_k^n)$, or\\
	(iii) $\mathcal{C}(\alpha_k^n)$ meets $\{0\}\times\mathcal{Q}_n$.
\end{thm}
\begin{proof}
The proof is analogous to the global bifurcation result presented in \cite{AM}. For more details, please refer to \cite{G1,R1}.
\end{proof}

\begin{rem}\label{rem312}
The result of Theorem \ref{th38} can be applied to each bifurcation point arising from an odd change in the Morse index of the radial solution $  u_{n,\al}$. Subsequently, by making use of Theorem \ref{thm39}, when $k$ is an even number, we can identity $\left[\frac{N}{2}\right]$ distinct continua of nonradial solutions that bifurcate from the point $(\alpha_k^n, u_{n,\alpha_{k}^n})$. Moreover, these continua are global, in accordance with Theorem \ref{th38}.
\end{rem}

\section{Crucial estimates of approximate solutions}
In this section, we establish several crucial estimates for the approximate solution to equation \eqref{31}, where the solution does not necessarily need to be nonradial. Let $ u_n := u_{\var_n,\al_n} $, where $ u_{\var,\al}$ (defined in \eqref{32}) denotes a radial solution of \eqref{31}. Furthermore, let $ v_n $ be a (possibly non-radial) solution of \eqref{31} associated with the parameters $ (\var_n,\alpha_n) $.

First, by applying the super/sub solution method and comparison principle, we can prove the precise uniform asymptotic estimate for approximate solutions $v_n$ on large ball $B_{\frac{1}{\var_n}}$ without using Green function representation formula, which is crucial for all the estimates in Section 4.
\begin{prop}\label{pp41}
	Let $\al_n$ and $\var_n$ be sequences such that $\al_n\to\al>0$ and $\var_n\to0$ as $n\to+\infty$.
	Let $v_n$ be a sequence of solutions to \eqref{31} in $B_{\frac{1}{\var_n}}:=B_{\frac{1}{\var_n}}(0)$, related to the exponent $\al=\al_n$ i.e.,
	\begin{equation}\label{41}
		\begin{cases}
			-\Delta_N v_n=|x|^{N\alpha_n}e^{v_n}, & x\in B_{\frac1{\var_n}}, \\
			v_n=\beta_{\al_n}, & x\in\partial B_{\frac1{\var_n}},
		\end{cases}
	\end{equation}
	where $\displaystyle \beta_{\al_n}:=U_{\al_n}\left(\frac{1}{\var_n}\right)=\ln\frac{C_{N,\alpha_n}\var_n^{\frac{N^2(  \al_n+1)}{N-1}
	}}{(1+\var_n^{\frac{N(  \al_n+1)}{N-1}
		})^N}$. Suppose that $\|e^{v_n}-e^{\beta_{\al_n}}\|_\gamma\leq A$ and $\inf\limits_{\partial B_1}v_n\geq B$ for some positive constant $A$ and constant $B$ independent of $n$, then there exists $n_{0}\geq1$ such that, for any $n\geq n_{0}$,
	\begin{align}\label{42}
		v_n(x)+\gamma_n\ln|x|\leq C_0 \qquad \text{in} \,\,\, B_{\frac1{\var_n}},
	\end{align}
where $C_0$ is a positive constant independent of $n$; moreover, for every $n\geq1$,
	\begin{align}\label{v43}
		v_n(x)+\gamma_n\ln|x|\geq B \qquad \text{in} \,\,\, B_{\frac1{\var_n}}\setminus B_1,
	\end{align}
where $\gamma_n\to\frac{N^2(\al+1)}{N-1}$ as $n\to+\infty$.
\end{prop}
\begin{proof}
Since $\|e^{v_n}-e^{\beta_{\al_n}}\|_\gamma\leq A$ for some positive constant $A$  independent of $n$, we have
\begin{align}\label{43q}
		e^{v_n}\leq C(1+|x|)^{-\gamma},
	\end{align}
and hence
\begin{align}\label{43q+}
v_{n}(x)\leq C-\gamma \ln(1+|x|) \qquad \text{in} \,\,\, B_{\frac1{\var_n}}
\end{align}
for some uniform positive constant $C>0$.
	
	Let $\gamma_n>0$ be such that
	\begin{equation}\label{def519}
		-\gamma_n:=\lim_{|x|\to\frac1{\var_n}}\frac{ v_n(x)-s_n}{\ln|x|},
	\end{equation}
where $B\leq s_n:=\inf\limits_{\partial B_1}v_n\leq C$ for some constant $C$ independent of $n$. From the boundary condition $v_n=\beta_{\al_n}$ on $\partial B_{\frac1{\var_n}}$, we deduce that
\begin{align}\label{w422}
-\gamma_n=\lim_{|x|\to\frac1{\var_n}}\frac{ v_n(x)-s_n}{\ln|x|}=\frac{\beta_{\alpha_n}-s_n}{\ln\frac{1}{\var_n}}
=\frac{\ln C_{N,\al_n}-N\ln\lr1+\lr\frac1{\var_n}\rr^{\frac{N(\alpha_n+1)}{N-1}}\rr-s_n}{\ln\left(\frac{1}{\var_n}\right)}.
\end{align}
It follows immediately that
\begin{equation}\label{equ416}
\gamma_n\to\frac{N^2(\al+1)}{N-1}, \quad  \text{as} \ n\to+\infty.
	\end{equation}
We will show that
\begin{equation}\label{u522}
		 v_n(x)\ge B-\gamma_n\ln  |x|\quad\text{in }B_{\frac1{\var_n}}\setminus B_1.
	\end{equation}
In fact, for $1<R\leq\frac1{\var_n}$, from  the definition of $-\gamma_n$ in \eqref{def519}, we can take $R$ large enough such that $\inf\limits_{\partial B_{R}}\frac{v_n(x)-s_n}{\ln  R}<0$,  we have
$$-\Delta_N(v_n-s_n)\ge-\Delta_N\left(\lr\inf_{\partial B_{R}}\frac{v_n(x)-s_n}{\ln  R}\rr\ln  |x|\right)=0 \quad\text{in } B_{R}\setminus\overline{B_1},$$
the comparison principle gives that
\begin{equation}\label{u48}\inf_{\partial B_{R}}\frac{v_n(x)-s_n}{\ln  R}=\inf_{1<|x| \le R}\frac{v_n(x)-s_n}{\ln  |x|} \quad \text{for}\ 1<R\leq\frac1{\var_n}.\end{equation}
Therefore, by the definition of $-\gamma_n$ and \eqref{u48}, we derive that $\lim\limits_{R\to\frac1{\var_n}}\inf\limits_{1< |x| \le R}\frac{v_n(x)-s_n}{\ln  |x|}=-\gamma_n$, which confirms \eqref{u522}.

Next, we will show that, there exists a positive constant $C_0$ independent of $n$ such that
	\begin{equation}\label{u524}
	v_n(x)\le C_0-\gamma_n\ln  |x|\quad\text{in }B_{\frac1{\var_n}}.
	\end{equation}
From \eqref{equ416}, for sufficiently large $n\geq n_{0}$, we have $\delta_{0,n}:=\gamma_n-\gamma>0$, then $-\gamma=-\gamma_n+\delta_{0,n}<-N(\al_n+1)$. Then from \eqref{43q}, we get that
	\begin{equation}\label{u525}
		v_n(x)\le C_1-\gamma \ln|x|\quad\text{in }B_{\frac1{\var_n}}\setminus\overline{B_1},
	\end{equation}
	for some positive constant $C_1$ independent of $n$. Now, we prove \eqref{u524} by constructing suitable super-solutions to \eqref{41} in an exterior domain and then applying the comparison principle, as described below.

To construct a super-solution of \eqref{41}, we introduce $\phi=\phi(t)$, a smooth function for $t>0$ with $\phi':=\frac{d\phi}{dt}<0$. Define $h(x)=\phi(|x|)$ for $x\in B_{\frac1{\var_n}}\setminus\{0\}$, it is straightforward to verify that
$$-\Delta_N h=t^{1-N}\left(t^{N-1}(-\phi')^{N-1}\right)'.$$
From \eqref{u525}, we note that $|x|^{N\al_n}e^{v_n(x)}\le C_2|x|^{N\al_n-\gamma }$ for $|x|>1$, where $C_2=e^{C_1}$ is a positive constant independent of $n$. Thus, given $1<t_0<\frac1{\var_n}$ big enough independent of $n$, consider the ODE
	\begin{equation}\label{eq526}
		t^{1-N}\left(t^{N-1}(-\phi')^{N-1}\right)'=t^{N\al_n-\gamma }\quad\text{for }t>t_0.\\
	\end{equation}
	By integration, we can conclude that the function
	\begin{equation*}
		\phi_{a,b}(t)=-(\gamma -N(\al_n+1))^{\frac{1}{1-N}}\int_{t_0}^t\frac{(a-\tau^{N(\al_n+1)-\gamma })^{\frac{1}{N-1}}}{\tau}\,\md\tau+b
	\end{equation*}
	solves \eqref{eq526} for any constants $a\ge (t_0)^{N(\al_n+1)-\gamma }$, $b=\phi(t_0)$ independent of $n$.

Now, let $h_{a,b}(x)=C_2^{\frac{1}{N-1}}\phi_{a,b}(|x|)$, we have
	\begin{equation}\label{u527}
		-\Delta_N h_{a,b}=C_2|x|^{N\al_n-\gamma }\ge |x|^{N\al_n}e^{v_n}=-\Delta_N v_n\quad\text{in }B_{\frac1{\var_n}}\setminus\overline{B_{t_0}}.
	\end{equation}
	Additionally, we find that for $x$ with $|x|\ge t_0$,
	\begin{align}\label{u528}
		h_{a,b}(x)&\ge-\left(\frac{aC_2}{\gamma -N(\al_n+1)}\right)^{\frac{1}{N-1}}\int_{t_0}^{|x|} \frac{1}{\tau}\,\md\tau+C_2^{\frac{1}{N-1}}b \\
		&=-\left(\frac{aC_2}{\gamma -N(\al_n+1)}\right)^{\frac{1}{N-1}}\ln  |x|+ C_3+C_2^{\frac{1}{N-1}}b, \notag
	\end{align}
	where $C_3=C_3(a,t_0)$ is a positive constant independent of $n$. Specifically, we have
	\begin{equation}\label{v529}
		h_{a,b}(x)= C_2^{\frac{1}{N-1}}b,\quad\text{when }\ |x|=t_0.
	\end{equation}
	Using these results, we now show by selecting appropriate $a$ and $b$ that $h_{a,b}(x)\ge v_n(x)$ for $|x|\ge t_0$.
	
Indeed, for each $\delta\in(0,\delta_{0,n})$, \eqref{def519} implies there exists $R_n(\delta)<\frac1{\var_n}$ big enough such that
	\begin{equation}\label{u530}
		v_n(x)\le (-\gamma_n+\delta)\ln  |x|+s_n\quad\text{for }\ R_n(\delta)<|x|\leq\frac1{\var_n}.
	\end{equation}
	Take $$a_\delta=\frac{(\gamma_n-\delta)^{N-1}(\gamma -N(\al_n+1))}{C_2}$$
	and note that $$\frac{\gamma ^{N-1}(\gamma -N(\al_n+1))}{C_2}\le a_\delta\le\frac{\gamma_n^{N-1}(\gamma -N(\al_n+1))}{C_2}.$$
	It also follows that the constant $C_3(a_\delta,t_0)$ as in \eqref{u528} is uniformly bounded. Now, fix $b$ such that
	$$C_2^{\frac{1}{N-1}}b=\max\left\{-C_3(a_\delta,t_0)\,,\,\sup_{n}\max_{\partial B_{t_0}}v_n\,,\,\sup_{n}s_n\right\}$$
	holds for any $\delta\in(0,\delta_{0,n})$. In view of \eqref{u528}, \eqref{v529} and \eqref{u530}, we get, for each $\delta\in(0,\delta_{0,n})$,
	\begin{equation*}
		h_{a_\delta,b}(x)\ge v_n(x),\quad\text{when }|x|= t_0\text{ and when }|x|>R_n(\delta).
	\end{equation*}
	Recall that \eqref{u527} holds. By applying the comparison principle, we therefore conclude that
	$$h_{a_\delta,b}(x)\ge v_n(x)\quad\text{on }B_{\frac1{\var_n}}\setminus B_{t_0}$$
	for each $\delta\in(0,\delta_{0,n}).$

Note that, for $\delta\in(0,\delta_{0,n})$, we have
	\begin{align*}
		h_{a_\delta,b}(x)-C_2^{\frac{1}{N-1}}b&=-\left(\frac{C_2}{\gamma -N(\al_n+1)}\right)^{\frac{1}{N-1}}\int_{t_0}^{|x|} \frac{(a_\delta-\tau^{N(\al_n+1)-\gamma })^{\frac{1}{N-1}}}{\tau}\,\md\tau\\
		&=-\left(\frac{a_\delta C_2}{\gamma -N(\al_n+1)}\right)^{\frac{1}{N-1}}\int_{t_0}^{|x|} \frac{1}{\tau}\,\md\tau+I\\
		&=(-\gamma_n+\delta)(\ln|x|-\ln  t_0)+I,
	\end{align*}
	where $$I=\left(\frac{C_2}{\gamma -N(\al_n+1)}\right)^{\frac{1}{N-1}}\int_{t_0}^{|x|} \frac{(a_\delta)^{\frac{1}{N-1}}-(a_\delta-\tau^{N(\al_n+1)-\gamma })^{\frac{1}{N-1}}}{\tau}\,\md\tau.$$
By choosing $1<t_0<\frac1{\var_n}$ sufficiently large (still fixed) such that $a_\delta\ge 2(t_0)^{N(\al_n+1)-\gamma }$ for any $\delta\in(0,\delta_{0,n})$.  Since for $\tau>t_0$, we get
	\begin{align*}
		0\le(a_\delta)^{\frac{1}{N-1}}-(a_\delta-\tau^{N(\al_n+1)-\gamma })^{\frac{1}{N-1}}&\le\frac{\tau^{N(\al_n+1)-\gamma }}{N-1}(a_\delta-\tau^{N(\al_n+1)-\gamma })^{\frac{2-N}{N-1}}
		\le\frac{\tau^{N(\al_n+1)-\gamma }}{N-1}t_0^{\frac{(N(\al_n+1)-\gamma )(2-N)}{N-1}}.
	\end{align*}
By Bernoulli's inequality and the fact that $a_\delta\ge 2(t_0)^{N(\al_n+1)-\gamma }$, we conclude that $I\le C_4$ for some constant $C_4$ independent of $\delta$ and $n$. Thus, we infer that
\begin{equation}\label{a}
  v_n(x)\le h_{a_\delta,b}(x)\le(-\gamma_n+\delta)(\ln|x|-\ln  t_0)+C_4+C_2^{\frac{1}{N-1}}b
\end{equation}
for $|x|\ge t_0$. By letting $\delta\to0$ in \eqref{a} and combining with \eqref{43q+}, we get \eqref{u524}. This completes our proof of Proposition \ref{pp41}.
\end{proof}

Furthermore, based on Proposition \ref{pp41}, we can show the following uniform fast decay estimate for $|\nabla v_n(x)|$.
\begin{prop}\label{pro43}
Let $\al_n$ and $\var_n$ be sequences such that $\al_n\to\al>0$ and $\var_n\to0$ as $n\to+\infty$.
Let $v_n$ be a sequence of solutions of \eqref{41} in $B_{\frac{1}{\var_n}}$, corresponding to the exponent $\al_n$. If $\|e^{v_n}-e^{\beta_{\al_n}
}\|_\gamma\leq A$ and $\inf\limits_{\overline{B_1}\setminus B_{\frac{1}{2}}}v_n\geq B$ for some positive constant $A$ and constant $B$ independent of $n$, then there exists $C>0$ independent of $n$ and $n_0\geq1$ such that, for any $n\geq n_0$,
	\begin{align}\label{q421}
		|\nabla v_n(x)|\leq\frac{C}{|x|}, \qquad \forall \,\,  x\in \overline{B_{\frac{1}{\var_n}}}\setminus\overline{B_{1}}.
	\end{align}
\end{prop}
\begin{proof}
For all $\frac{1}{2}<R\leq\frac{1}{3\var_n}$, $\forall \,\, 1<|y|<\frac{1}{\var_{n} R}$, define $v_n^R(y):=v_n(Ry)+\gamma_n\ln R$. From Proposition \ref{pp41}, we have $|v_n(x)+\gamma_n\ln |x||\leq C$ for some uniform constant $C>0$, any $|x|\geq \frac{1}{2}$ and $n\geq n_0$, and thus
\begin{align}\label{q422}
v_n^R(y)\leq C-\gamma_n\ln|Ry|+\gamma_n\ln R\leq C-\gamma_n\ln|y|\leq C, \qquad \forall \,\, 1<|y|<\frac{1}{\var_{n}R},
\end{align}
\begin{align}\label{q422+}
v_n^R(y)\geq -C-\gamma_n\ln|Ry|+\gamma_n\ln R
\geq -C-\gamma_n\ln|y|
\geq -C-\gamma_n\ln3\geq -C, \quad \forall \,\, 1<|y|<3.
\end{align}
As a consequence, we have $|v_n^R(y)|\leq C$ for some constant $C$ independent of $n$ and any $y\in \mathbb{R}^{N}$ such that $1<|y|<3$. Moreover, $v_n^R(y)$ satisfies
\begin{align}\label{q423}
	-\Delta_N v_n^R(y)=R^{N}|Ry|^{N\alpha_n}e^{v_n(Ry)}\leq C_1 R^{N}|Ry|^{N\alpha_n}(1+R|y|)^{-\gamma_n}\leq C_1, \quad \forall \,\, 1<|y|<\frac{1}{\var_{n}R},
\end{align}
where $C_1$ is independent of $n$, $R$ and $y$. Thus, from the gradient regularity estimate in \cite{Tp}, we get
$$\|\nabla_y  v_n^R(y)\|_{L^\infty(B_2(0)\setminus B_{\frac{3}{2}}(0))}\leq C_2,$$
where $C_2$ is independent of $n$ and $R$. For all $1<|x|\leq\frac{1}{2\var_n}$, let $x=Ry$, where $\frac{3}{2}<|y|<2$, so we obtain, for any $n\geq n_0$,
$$|\nabla_x v_n(x)|=R^{-1}|\nabla_y  v_n^R(y)|\leq C_2R^{-1}\leq\frac{2C_2}{|x|}.$$

For all $\frac{1}{2\var_n}<|x|\leq\frac{1}{\var_n}$, we take $R=\frac{1}{2\var_n}$ and let $x=Ry=\frac{y}{2\var_n}$, where $1<|y|\leq2$. Similar to the proof of Proposition \ref{pro26}, by considering the estimates \eqref{q422}, \eqref{q422+} and \eqref{q423} on $\{1<|y|<2\}$ for $V_n:=v_n^{R}$ with $R=\frac{1}{2\var_n}$, from the interior and boundary gradient regularity estimate (c.f. e.g. Theorems 8.32 and 8.33 in \cite{GT}, Theorem 1 in \cite{Lieb}), we have, for every $n\geq n_0$,
$$\|\nabla_y V_n(y)\|_{L^\infty(\overline{B_{2}(0)}\setminus B_{1}(0))}\leq C_3,$$
where $C_3$ is independent of $n$ and $n_0\geq 1$. Thus we obtain, for any $n\geq n_0$,
$$|\nabla_x v_n(x)|=2\var_n|\nabla_y  V_n(y)|\leq 2C_3\var_n\leq\frac{2C_3}{|x|}.$$
This gives the desired estimate \eqref{q421} and hence completes our proof of Proposition \ref{pro43}.
\end{proof}

Let $\al_n$ and $\var_n$ be sequences such that $\al_n\to\al>0$ and $\var_n\to0$ as $n\to+\infty$. Let $v_n$ be a sequence of nonradial solutions of \eqref{31} with $\var=\var_n$ related to the exponent $\al_n$ (i.e., \eqref{41}), and let $w_n:=\frac{u_n-v_n}{\|u_n-v_n\|_{L^{\infty}(\R^N)}}$, where $u_n$ is the radial solution of \eqref{31} with $\var=\var_n$ related to the exponent $\al_n$ given by \eqref{32}.

\medskip

In the following Proposition, we first derive the uniform upper bound for global integral of $\frac{|\nabla w_n|^2}{|x|^{N-2}}$. Then, by careful choice of a series of cut-off functions $\varphi^{(k)}_R$ ($k\geq1$) and a dyadic iteration method, we can improve the uniform upper bound for global integral of $|x|^{-(N-2)}|\nabla w_n|^2$ and obtain the uniform $(\ln R)^{-1}$-type decay estimate on the integral of $|x|^{-(N-2)}|\nabla w_n|^2$ outside of the ball $B_R(0)$, which is crucial in our subsequent proof.
\begin{prop}\label{ppp43}
If $\|e^{v_n}-e^{\beta_{\al_n}}\|_\gamma\leq A$ and $\inf\limits_{\overline{B_1}}v_n\geq B$ for some positive constant $A$ and constant $B$ independent of $n$, then there exists a constant $C>0$ independent of $n$ and $n_0\geq1$ such that
	\begin{align}	
		\int_{B_{\frac{1}{\var_n}}\setminus B_1}|x|^{-(N-2)}|\nabla w_n|^2\md x&\leq C, \qquad \forall \,\, n\geq n_0.
\label{aa438}
	\end{align}
Furthermore, for any $R\geq 2$,
\begin{align}	
		\int_{B_{\frac{1}{\var_n}}\setminus B_R}|x|^{-(N-2)}|\nabla w_n|^2\md x&\leq
		\frac{C}{\ln R}, \qquad \forall \,\, n\geq n_0,
\label{ln 420}
	\end{align}
where $C>0$ is a constant independent of $n$ and $R$.
\end{prop}

\begin{proof}
By direct calculations, we have
\begin{align*}
	\nabla U_{\al_n} = -\frac{N^2( \al_n+1)}{N-1}\frac{|x|^{\frac{N( \al_n+1)}{N-1}-1}}
	{1 + |x|^{\frac{N( \al_n+1)}{N-1}}}\frac{x}{|x|},
\end{align*}
and hence, for any $n\geq1$,
\begin{align}\label{n4}
	\frac{\hat{C}}{|x|}\leq|\nabla U_{\al_n}(x)|\leq\frac{C}{|x|},\qquad \forall \,\, |x|\geq1,
\end{align}
where $C>0$ and $\hat{C}>0$ are constants independent of $n$. Since Proposition \ref{pp41} and $\inf\limits_{\overline{B_1}}v_n\geq B$ imply that, there exists $n_0\geq1$ such that
$$\left\|v_n+\frac{\gamma_n(N-1)}{N(\al_n+1)}\ln\left(1+|x|^{\frac{N(\alpha_n+1)}{N-1}}\right)\right\|_{L^{\infty}(B_{\frac{1}{\var_{n}}})}\leq C, \quad\, \forall \,\, n\geq n_0$$
for some uniform constant $C$ (w.r.t. $n$), by the formula of $\gamma_n$ in \eqref{w422}, we obtain
\begin{align}\label{w421}
&\quad \|v_n(x)-u_n(x)\|_{L^\infty(B_{\frac{1}{\var_{n}}})} \\
&\leq\left\|v_n-\ln C_{N,\al_n}+N\ln\left(1+|x|^{\frac{N(\alpha_n+1)}{N-1}}\right)\right\|_{L^\infty(B_{\frac{1}{\var_{n}}})} \nonumber\\
&\leq\Big\|v_n+\frac{\gamma_n(N-1)}{N(\al_n+1)}\ln\left(1+|x|^{\frac{N(\alpha_n+1)}{N-1}}\right)-\ln C_{N,\al_n}\nonumber \\
&\qquad+\left(N-\frac{\gamma_n(N-1)}{N(\al_n+1)}\right)\ln\left(1+|x|^{\frac{N(\alpha_n+1)}{N-1}}\right)\Big\|_{L^\infty(B_{\frac{1}{\var_{n}}})}\nonumber\\
&\leq C+\left|\lr\frac{N^2(\alpha_n+1)}{N-1}-\gamma_n\rr\ln\left(\frac1{\var_n}\right)\right|+o_{n}(1)\ln2 \nonumber\\
&\leq C+\left|\frac{N^2(\alpha_n+1)}{N-1}\ln\left(\frac1{\var_n}\right)+\ln C_{N,\al_n}-N\ln\lr1+\lr\frac1{\var_n}\rr^{\frac{N(\alpha_n+1)}{N-1}}\rr-s_n\right|\nonumber\\
&\leq C+|\ln C_{N,\al_n}-s_n|+N\ln\lr1+\var_n^{\frac{N(\alpha_n+1)}{N-1}}\rr\nonumber\\
&\leq C+N\var_n^{\frac{N(\alpha_n+1)}{N-1}}\leq C, \qquad \forall \,\, n\geq n_0\nonumber
\end{align}
for some uniform positive constant $C$ independent of $n$. Let
\begin{align}\label{n0}
\zeta_n(x):=\int_0^1e^{tu_n+(1-t)v_n}\md t.
\end{align}
From \eqref{w421}, we have, for any $n\geq n_0$,
\begin{align}\label{xi425}
\frac{1}{C}e^{U_{\al_n}}\leq e^{-\|u_{n}-v_{n}\|_{L^{\infty}}}e^{u_n}\leq Ce^{U_{\al_n}}\leq \zeta_n(x)\leq e^{\|u_{n}-v_{n}\|_{L^{\infty}}}e^{u_n}\leq Ce^{U_{\al_n}},\quad\forall\ x\in B_{\frac{1}{\varepsilon_n}},
\end{align}
where $C$ is a positive constants independent of $n$.

By \eqref{n4} and \eqref{xi425}, and noting that $|w_n|\leq1$, we obtain that, for any $n\geq n_0$,
	\begin{align}\label{a438}	
		&\quad\int_{B_{\frac{1}{\var_n}}\setminus B_1}|x|^{-(N-2)}|\nabla w_n|^2\md x\\
		&\leq
		C\int_{B_{\frac{1}{\var_n}}\setminus B_1}|\nabla U_{\al_n}|^{N-2}|\nabla w_n|^2\md x
		\leq C\int_{B_{\frac{1}{\var_n}}}|\nabla w_n|^2\left(|\nabla u_n|^{N-2}+|\nabla v_n|^{N-2}\right)\md x\nonumber\\
		&\leq \frac{C}{\|u_n-v_n\|^2_{L^{\infty}(\R^N)}}\int_{B_{\frac{1}{\var_n}}}\nabla (u_n-v_n)\cdot\left(|\nabla u_n|^{N-2}\nabla u_n-|\nabla v_n|^{N-2}\nabla v_n\right)\md x\nonumber\\
        &= \frac{C}{\|u_n-v_n\|^2_{L^{\infty}(\R^N)}}\int_{B_{\frac{1}{\var_n}}}|x|^{N\al_n}\left(e^{u_{n}}-e^{v_{n}}\right)(u_{n}-v_{n})\md x\nonumber\\
		&= C\int_{B_{\frac{1}{\var_n}}}|x|^{N\al_n} \zeta_n(x)|w_n|^2\md x\nonumber\\
		&\leq C\int_{B_{\frac{1}{\var_n}}}|x|^{N\al_n} e^{U_{\al_n}}|w_n|^2\md x
		\leq C\int_{B_{\frac{1}{\var_n}}}|x|^{N\al_n} e^{U_{\al_n}}\md x\nonumber\\
		&\leq C\int_{B_{\frac{1}{\var_n}}}\frac{C_{N,\al_n}|x|^{N\al_n}}{\lr1+|x|^{\frac{N(\al_n+1)}{N-1}}\rr^N} \md x\leq C_{0},\nonumber
	\end{align}
where $C_{0}$ is a positive constant independent of $n$. Thus \eqref{aa438} holds.

For any $R\geq2$ and every $k\geq1$, let
\begin{equation}\label{r426}
\varphi^{(k)}_R(x)=\begin{cases}
			0, & |x|\leq R^{2^{-k}}, \\
			\frac{2^{k}\ln|x|-\ln R}{\ln R}, & R^{2^{-k}}<|x|<R^{2^{-k+1}}, \\
			1, & |x|\geq R^{2^{-k+1}}.
		\end{cases}
	\end{equation}
From \eqref{r426} and $|w_n|\leq1$, by Proposition \ref{pro43} and \eqref{aa438}, we have, for any $n\geq n_0$, $R\geq2$ and $k\geq1$,
\begin{align}	\label{w427}
		&\quad\int_{B_{\frac{1}{\var_n}}\setminus B_{R^{2^{-k+1}}}}|x|^{-(N-2)}|\nabla w_n|^2\md x\\
&\leq
		C\int_{B_{\frac{1}{\var_n}}}|\nabla U_{\al_n}|^{N-2}|\nabla w_n|^2\varphi^{(k)}_R(x)\md x\nonumber\\
		&\leq C\int_{B_{\frac{1}{\var_n}}}|\nabla w_n|^2\left(|\nabla u_n|^{N-2}+|\nabla v_n|^{N-2}\right)\varphi^{(k)}_R(x)\md x\nonumber\\
		&\leq \frac{C}{\|u_n-v_n\|^2_{L^{\infty}(\R^N)}}\int_{B_{\frac{1}{\var_n}}}\varphi^{(k)}_R\nabla (u_n-v_n)\cdot\left(|\nabla u_n|^{N-2}\nabla u_n-|\nabla v_n|^{N-2}\nabla v_n\right)\md x\nonumber\\
		&\leq \int_{B_{\frac{1}{\var_n}}}\lr\nabla\lr (u_n-v_n)\varphi^{(k)}_R
		\rr- (u_n-v_n)\nabla\varphi^{(k)}_R\rr\cdot\left(|\nabla u_n|^{N-2}\nabla u_n-|\nabla v_n|^{N-2}\nabla v_n\right)\md x\nonumber\\
		&\quad \times\frac{C}{\|u_n-v_n\|^2_{L^{\infty}(\R^N)}}\nonumber\\
		&\leq C\int_{B_{\frac{1}{\var_n}}}|x|^{N\al_n} \zeta_n(x)|w_n|^2\varphi^{(k)}_R\md x+C\int_{B_{\frac{1}{\var_n}}}|\nabla\varphi^{(k)}_R||w_n|\left(|\nabla u_n|^{N-2}+|\nabla v_n|^{N-2}\right)|\nabla w_n|\md x\nonumber\\
		&\leq C\int_{B_{\frac{1}{\var_n}}\setminus B_{R^{2^{-k}}}}|x|^{N\al_n}e^{U_{\al_n}}\md x+C\int_{B_{R^{2^{-k+1}}}\setminus B_{R^{2^{-k}}}}\frac{2^{k}}{\ln R}\frac{1}{|x|^{N-1}}|w_n||\nabla w_n|\md x\nonumber\\
&\leq C\int_{B_{\frac{1}{\var_n}}\setminus B_{R^{2^{-k}}}}\frac{C_{N,\al_n}|x|^{N\al_n}}{\lr1+|x|^{\frac{N(\al_n+1)}{N-1}}\rr^N} \md x\nonumber\\
&\quad +\frac{1}{2}\int_{B_{R^{2^{-k+1}}}\setminus B_{R^{2^{-k}}}}|x|^{-(N-2)}|\nabla w_n|^2\md x+C\frac{4^{k}}{(\ln R)^{2}}\int_{B_{R^{2^{-k+1}}}\setminus B_{R^{2^{-k}}}}\frac{|w_{n}|^{2}}{|x|^{N}}\md x \nonumber\\
&\leq\frac{C}{R^{\frac{N(\al_n+1)}{2^{k}(N-1)}}}+\frac{C2^{k}}{\ln R}+\frac{1}{2}\int_{B_{R^{2^{-k+1}}}\setminus B_{R^{2^{-k}}}}|x|^{-(N-2)}|\nabla w_n|^2\md x \nonumber \\
&\leq\frac{C}{R^{\frac{N}{2^{k}(N-1)}}}+\frac{C2^{k}}{\ln R}+\frac{1}{2}\int_{B_{\frac{1}{\var_n}}\setminus B_{R^{2^{-k}}}}\frac{|\nabla w_n|^2}{|x|^{N-2}}\md x- \frac{1}{2}\int_{B_{\frac{1}{\var_n}}\setminus B_{R^{2^{-k+1}}}}\frac{|\nabla w_n|^2}{|x|^{N-2}}\md x\nonumber\\
&\leq\frac{C2^{k}}{\ln R}+\frac{1}{2}\int_{B_{\frac{1}{\var_n}}\setminus B_{R^{2^{-k}}}}|x|^{-(N-2)}|\nabla w_n|^2\md x- \frac{1}{2}\int_{B_{\frac{1}{\var_n}}\setminus B_{R^{2^{-k+1}}}}|x|^{-(N-2)}|\nabla w_n|^2\md x,\nonumber
	\end{align}
where $C$ is a positive constant independent of $n$, $R$ and $k$. It follows from \eqref{w427} that, for any $R\geq2$ and every $k\geq1$,
\begin{equation}\label{b0}
  \int_{B_{\frac{1}{\var_n}}\setminus B_{R^{2^{-k+1}}}}|x|^{-(N-2)}|\nabla w_n|^2\md x\leq \frac{C2^{k}}{\ln R}+\frac{1}{3}\int_{B_{\frac{1}{\var_n}}\setminus B_{R^{2^{-k}}}}|x|^{-(N-2)}|\nabla w_n|^2\md x.
\end{equation}
For any $n\geq n_0$, $R\geq2$ and $k\geq1$, let us define
\begin{equation}\label{b1}
  \Lambda_{R}(k):=\int_{B_{\frac{1}{\var_n}}\setminus B_{R^{2^{-k+1}}}}|x|^{-(N-2)}|\nabla w_n|^2\md x,
\end{equation}
then \eqref{b0} implies that
\begin{equation}\label{b2}
  \Lambda_{R}(k)\leq C\frac{2^{k}}{\ln R}+\frac{1}{3}\Lambda_{R}(k+1).
\end{equation}
From \eqref{a438}, we deduce that, for any $n\geq n_0$ and $R\geq2$, $\Lambda_{R}(k)$ is monotonically increasing w.r.t. $k\geq1$ and
\begin{equation}\label{b3}
  \lim\limits_{k\rightarrow+\infty}\Lambda_{R}(k)=\int_{B_{\frac{1}{\var_n}}\setminus B_{1}}|x|^{-(N-2)}|\nabla w_n|^2\md x\leq C_{0}.
\end{equation}
Therefore, \eqref{b1}, \eqref{b2} and \eqref{b3} imply that, for any $n\geq n_0$ and $R\geq2$,
\begin{align}	\label{w428}
		\Lambda_{R}(1)&=\int_{B_{\frac{1}{\var_n}}\setminus B_{R}}|x|^{-(N-2)}|\nabla w_n|^2\md x\\
&\leq \frac{2C}{\ln R}+\frac{1}{3}\Lambda_{R}(2)\leq \frac{2C}{\ln R}+\frac{1}{3}\frac{2^{2}C}{\ln R}+\frac{1}{3^{2}}\Lambda_{R}(3) \nonumber\\
		&\leq \cdots \leq C\sum\limits_{k=0}^{K}\frac{1}{3^{k}}\frac{2^{k+1}}{\ln R}+\frac{1}{3^{K+1}}\Lambda_{R}(K+2)\nonumber\\
		&\leq \frac{C}{\ln R}\sum\limits_{k=0}^{+\infty}\left(\frac{2}{3}\right)^{k}+\lim\limits_{K\rightarrow+\infty}\frac{1}{3^{K+1}}\Lambda_{R}(K+2)=\frac{C}{\ln R},\nonumber
	\end{align}
where $C$ is a positive constant independent of $n$ and $R$. This completes our proof of Proposition \ref{ppp43}.
\end{proof}

Now, setting $f = e^{u} - e^{\beta_{\alpha_{n}}}$,where $u$ is a solution of \eqref{41}, then $f$ satisfies
\begin{equation}\label{f474}
	\begin{cases}
		-\Delta_N f = |x|^{N\alpha_n}\left(f + e^{\beta_{\alpha_{n}}}\right)^{N} - (N-1)\frac{|\nabla f|^{N}}{f + e^{\beta_{\alpha_{n}}}}, & x \in B_{\frac{1}{\varepsilon_n}}, \\[2mm]
		f = 0, & x \in \partial B_{\frac{1}{\varepsilon_n}}.
	\end{cases}
\end{equation}
Let $f_n = e^{u_{n}} - e^{\beta_{\alpha_{n}}}$ and $g_{n} = e^{v_{n}} - e^{\beta_{\alpha_{n}}}$ be a radial and a non-radial solution of \eqref{f474}, respectively. Defining
\begin{equation}\label{h449}
	h_{n}(x) := \frac{f_{n} - g_{n}}{\|f_{n} - g_{n}\|_{L^{\infty}(\mathbb{R}^N)}}
	\quad \text{and} \quad
	\psi_n(x) := \frac{u_n - v_n}{\|e^{u_n} - e^{v_n}\|_{L^{\infty}(\mathbb{R}^N)}},
\end{equation}
we obtain
\begin{equation}\label{h450}
	h_n(x) = \zeta_n(x)\psi_n(x),
\end{equation}
where $\zeta_n(x)$ is defined in \eqref{n0}.

For any $R>0$, we can overcome the lack of critical weighted Sobolev embedding inequality and prove the following (almost critical) truncated weighted Hardy-Sobolev inequality on exterior domain $\R^N\setminus B_{R}(0)$ instead, which is important for us to derive a uniform upper bound estimate for the $L^2$-integral average of $\psi_{n}$ in the annulus $B_{8R}(0)\setminus B_{R}(0)$, and to prove Propositions \ref{pp42} and \ref{pp44}.
\begin{prop}\label{tws+}
Assume $N\geq2$ and $R>1$. For any $\varphi\in \dot{H}^{1}(\R^N\setminus \overline{B_{R}(0)};\,|x|^{-(N-2)})\cap L^2(\R^N\setminus \overline{B_{R}(0)};\,|x|^{-N}[\ln|x|]^{-2})$ such that $\varphi|_{\partial B_{R}}=0$, we have
	\begin{align}\label{aa436+}
		\int_{\R^N\setminus B_{R}(0)} \frac{1}{[\ln |x|]^{2}}\frac{|\varphi|^2}{|x|^{N}}\md x\leq
		4\int_{\R^N\setminus B_{R}(0)}\frac{|\nabla\varphi|^{2}}{|x|^{N-2}}\md x.
		\end{align}
\end{prop}
\begin{proof}
To prove \eqref{aa436+}, we can assume by approximation that $\varphi\in C_{c}^1(\R^N)$ such that $\varphi|_{\partial B_{R}}=0$. From Fubini's theorem and using polar coordinates, we obtain
	\begin{align}\label{a436+}	
	&\quad \int_{\R^N\setminus B_{R}(0)}	\frac{1}{[\ln |x|]^{2}}\frac{|\varphi|^2}{|x|^{N}}\md x \\ &=\int_{\Sm^{N-1}}\int_R^{+\infty}r^{-1}[\ln r]^{-2}|\varphi(r\theta)|^2\md r\md\theta \nonumber \\
		&\leq 2\int_{\Sm^{N-1}}\int_R^{+\infty}r^{-1}[\ln r]^{-2}\int_R^{r}|\varphi(t\theta)||\nabla\varphi(t\theta)|\md t\md r\md\theta\nonumber\\ &=2\int_{\Sm^{N-1}}\int_R^{+\infty}|\varphi(t\theta)||\nabla\varphi(t\theta)|\int_t^{+\infty}r^{-1}[\ln r]^{-2}\md r\md t\md\theta\nonumber\\
		&=2\int_{\Sm^{N-1}}\int_R^{+\infty}\frac{1}{\ln t}|\varphi(t\theta)||\nabla\varphi(t\theta)|\md t\md\theta\nonumber\\
		&\leq 2\lr\int_{\Sm^{N-1}}\int_R^{+\infty}\frac{1}{t[\ln t]^{2}}|\varphi(t\theta)|^2\md t\md\theta\rr^{\frac12\nonumber}\lr\int_{\Sm^{N-1}}\int_R^{+\infty}t|\nabla\varphi(t\theta)|^2\md t\md\theta\rr^{\frac12}\nonumber \\
&= 2\lr\int_{\R^N\setminus B_{R}(0)}\frac{1}{[\ln |x|]^{2}}\frac{|\varphi|^2}{|x|^{N}}\md x\rr^{\frac12}\lr\int_{\R^N\setminus B_{R}(0)}\frac{|\nabla\varphi|^{2}}{|x|^{N-2}}\md x\rr^{\frac12}.\nonumber
	\end{align}	
Hence \eqref{aa436+} follows from \eqref{a436+} immediately. This finishes our proof.
\end{proof}

By considering the equation satisfied by $\psi_n$, using the harmonic replacement and the truncated weighted Hardy-Sobolev inequality \eqref{aa436+}, we overcame the nonlinear nature of the $N$-Laplacian $\Delta_N$ and the divergence (to $-\infty$) at $\infty$ of the solutions $u_n$ and $v_n$, and proved the following uniform upper bound on $L^2$-integral average of $\psi_n$.
\begin{prop}\label{qpp44}
	Assume $\al>0$ and $ N\geq2$. If $\|e^{v_n}-e^{\beta_{\al_n}
}\|_\gamma\leq A$ and $\inf\limits_{\overline{B_1}}v_n\geq B$ for some positive constant $A$ and constant $B$ independent of $n$, then there exists $n_0\geq1$ such that, for any $R_0>8$ sufficiently large,
	\begin{align}\label{wL2+}
		R^{-N}\int_{B_{8R}\setminus B_{R}}|\psi_n|^2\md x\leq C[\ln R_{0}]^{2} + C [\ln R_{0}]^{2} R_0^{-\frac{N(\al+1)}{N-1}}[\ln R]^{2}, \qquad \forall \,\, R\geq R_0,\ \forall \,\, n\geq n_0,
	\end{align}
where $C>0$ is a constant independent of $n$, $R$ and $R_{0}$.
\end{prop}
\begin{proof}
First, we will deduce the equation satisfied by $\psi_{n}$. To this end, let us define $b(x)=|x|^{N-2}x$, $b_j=|x|^{N-2}x_j$, so
$$\partial_{x_j}(b_i)=(N-2)|x|^{N-4}x_ix_j+|x|^{N-2}\delta_{ij}:=A_{ij},$$	
and
\begin{align}
\sum_{i,j=1}^{N}A_{ij}\xi_i\xi_j&=\sum_{i,j=1}^{N}\left((N-2)|x|^{N-4}x_ix_j+|x|^{N-2}\delta_{ij}\right)\xi_i\xi_j\label{432qq}\\
&=\sum_{i,j=1}^{N}(N-2)|x|^{N-4}x_i\xi_ix_j\xi_j+|x|^{N-2}|\xi|^2\nonumber\\
&=|x|^{N-2}|\xi|^2\left(1+(N-2)\left|\frac{x}{|x|}\cdot\frac{\xi}{|\xi|}\right|^2\right)\geq|x|^{N-2}|\xi|^2.\nonumber
\end{align}
Consequently, we have
\begin{align*}
&\quad|\nabla u_n|^{N-2}\partial_{x_{i}} u_n-|\nabla v_n|^{N-2}\partial_{x_{i}} v_n\\
&=\sum_{j=1}^{N}\int_0^1\Big[ (N-2)|t\nabla u_n+(1-t)\nabla v_n|^{N-4}\partial_{x_i}(tu_n+(1-t)v_n)
\partial_{x_j}(tu_n+(1-t)v_n)\\
&\quad+|t\nabla u_n+(1-t)\nabla v_n|^{N-2}\delta_{ij}\Big]\md t \, \partial_{x_j}(u_n-v_n),\qquad i=1,\cdots,N.
\end{align*}
Defining
\begin{align*}
a_{ij}^n(x):&=\int_0^1\Big[(N-2)|t\nabla u_n+(1-t)\nabla v_n|^{N-4}\partial_{x_i}(tu_n+(1-t)v_n)
\partial_{x_j}(tu_n+(1-t)v_n)\\
&\quad+|t\nabla u_n+(1-t)\nabla v_n|^{N-2}\delta_{ij}\Big]\md t,
\end{align*}
then we have that, the equation
\begin{align*}
-\text{div}(|\nabla u_n|^{N-2}\nabla u_n-|\nabla v_n|^{N-2}\nabla v_n  )=|x|^{N\al_n}\left( e^{u_n}-e^{v_n}\right)
\end{align*}
is equivalent to
\begin{align}\label{e2}
-\sum_{i=1}^{N}\partial_{x_i}\left( \sum_{j=1}^{N}a_{ij}^n(x)\partial_{x_j}(u_n-v_n)\right)=|x|^{N\al_n}\left( e^{u_n}-e^{v_n}\right).
\end{align}	
By the definition of $\psi_n$, we have 	
\begin{align*}
-\sum_{i=1}^{N}\partial_{x_i}\left( \sum_{j=1}^{N}a_{ij}^n(x)\partial_{x_j}\psi_n\right)&=|x|^{N\al_n}\int_0^1e^{tu_n+(1-t)v_n}\md t \frac{u_n-v_n}{\|e^{u_n}-e^{v_n}\|_{L^{\infty}(\mathbb{R}^{N})}}=:|x|^{N\al_n}\zeta_n(x)\psi_n.
\end{align*}	
Therefore, $\psi_n$ satisfies the equation
\begin{align}\label{w}
\begin{cases}
\displaystyle-\sum_{i=1}^{N}\partial_{x_i}\left( \sum_{j=1}^{N}a_{ij}^n(x)\partial_{x_j}\psi_n\right)=|x|^{N\al_n}\zeta_n(x)\psi_n, &x\in B_{\frac{1}{\var_n}},\\
\psi_n=0, &x\in\partial B_{\frac{1}{\var_n}}.
\end{cases}
\end{align}

On the one hand, we have, for any $\xi,\,\eta\in\R^N$,
\begin{equation}\label{n2}
  \left|\sum_{i,j=1}^{N}a_{ij}^n(x)\xi_i\eta_j\right|\leq(N-1)\int_0^1|(1-t)\nabla u_n+t\nabla v_n|^{N-2}\md t |\xi||\eta|.
\end{equation}
On the other hand, from \eqref{432qq}, we get
\begin{align}\label{n3}
\sum_{i,j=1}^{N}a_{ij}^n(x)\xi_i\xi_j
\geq \displaystyle \int_0^1|(1-t)\nabla u_n+t\nabla v_n|^{N-2}\md t |\xi|^2.
\end{align}

Let
\begin{equation}\label{n1}
  \eta_n(x):=\int_0^1|(1-t)\nabla u_n+t\nabla v_n|^{N-2}\md t.
\end{equation}
We have the following estimates on $\eta_n(x)$ and $\zeta_n(x)$:
\begin{align}\label{435}
\eta_n(x)\geq C(N) |\nabla u_n|^{N-2},
\end{align}
\begin{align}\label{436}
\eta_n(x)\leq
\left(|\nabla u_n|+|\nabla v_n|\right)^{N-2}
\end{align}
and
\begin{align}\label{437}
\frac{1}{C}e^{U_{\al_n}}\leq\zeta_n(x)\leq Ce^{U_{\al_n}}.
\end{align}

The estimate \eqref{436} can be easily deduced from the definition \eqref{n1} of $\eta_n$. The estimate \eqref{437} follows directly from \eqref{xi425}.  We only need to show \eqref{435}.

In fact, if $|\nabla u_n|\geq|\nabla(v_n-u_n)|$, we derive
	\begin{align*}
		\eta_n(x)&\geq\int_0^1\Big||\nabla u_n|-t|\nabla (v_n-u_n)|\Big|^{N-2}\md t\geq\int_0^{\frac12}\left(|\nabla u_n|-t|\nabla (v_n-u_n)|\right)^{N-2}\md t\\&\geq\int_0^{\frac12}\left(|\nabla u_n|-\frac12|\nabla (v_n-u_n)|\right)^{N-2}\md t\geq\left(\frac{1}{2}\right)^{N-1}|\nabla u_n|^{N-2},
	\end{align*}
if $|\nabla u_n|<|\nabla(v_n-u_n)|$, we get
	\begin{align*}
	\eta_n(x)&\geq\int_0^1\Big||\nabla u_n|-t|\nabla (v_n-u_n)|\Big|^{N-2}\md t\\
	&=\int_0^{\frac{|\nabla u_n|}{|\nabla (v_n-u_n)|}}\lr|\nabla u_n|-t|\nabla (v_n-u_n)|\rr^{N-2}\md t+\int_{\frac{|\nabla u_n|}{|\nabla (v_n-u_n)|}}^1\lr t|\nabla (v_n-u_n)|-|\nabla u_n|\rr^{N-2}\md t\\
	&=\frac{1}{N-1}\frac{|\nabla u_n|^{N-1}+(|\nabla(v_n-u_n)|-|\nabla u_n|)^{N-1}}{|\nabla(v_n-u_n)|}\\
	&\geq\frac{2^{2-N}}{N-1}|\nabla(v_n-u_n)|^{N-2}
	\geq\frac{2^{2-N}}{N-1}|\nabla u_n|^{N-2}.
	\end{align*}
This implies \eqref{435}.

From \eqref{n2}, \eqref{n3}, \eqref{435}, \eqref{436}, \eqref{n4} and Proposition \ref{pro43}, we deduce that, for any $\xi,\eta\in\R^N$ and $n\geq n_0$,
\begin{equation}\label{n7}
  |a_{ij}^n(x)|\leq (N-1)\eta_n(x)\leq \frac{C_{1}}{|x|^{N-2}},
\end{equation}
\begin{equation}\label{n5}
  \left|\sum_{i,j=1}^{N}a_{ij}^n(x)\xi_i\eta_j\right|\leq C_{1}\frac{|\xi||\eta|}{|x|^{N-2}},
\end{equation}
and
\begin{align}\label{n6}
\sum_{i,j=1}^{N}a_{ij}^n(x)\xi_i\xi_j\geq C_{2} \frac{|\xi|^2}{|x|^{N-2}},
\end{align}
where $C_{1}>0$ and $C_{2}>0$ are uniform constants independent of $n$.

\medskip

Choose arbitrarily $R_0 > 1$ large enough. By \eqref{xi425}, \eqref{h449}  and \eqref{h450}, on $\partial B_{R_0}$ we have
$$
|\psi_n(x)| = \frac{|h_n(x)|}{\zeta_n(x)} \le Ce^{- U_{\al_n}(x)}.
$$
Since $\alpha_n \to \alpha$, there exists a constant $M_0 = M_0(R_0) > 0$ (to be determined later), independent of $n$, such that
\begin{equation}\label{464p}
\|\psi_n\|_{L^\infty(\partial B_{R_0})} \le M_0.
\end{equation}

Let $H_n$ be the harmonic replacement of $\psi_n$ in $B_{\frac{1}{\var_n}} \setminus \overline{B_{R_0}}$, namely,
\begin{equation}\label{465p}
H_n - \psi_n \in H_0^1(B_{\frac{1}{\var_n}} \setminus \overline{B_{R_0}}),
\end{equation}
and $H_n$ satisfies
\begin{equation}\label{466p}
\int_{B_{\frac{1}{\var_n}} \setminus \overline{B_{R_0}}} \sum_{i,j=1}^{N}a_{ij}^n(x)\partial_j H_n  \partial_i \phi\mathrm{d}x = 0
\quad \text{for all } \phi \in H_0^1(B_{\frac{1}{\var_n}} \setminus \overline{B_{R_0}}).
\end{equation}
For each fixed $n$, the domain $B_{\frac{1}{\var_n}} \setminus \overline{B_{R_0}}$ is a bounded annulus and $a_{ij}^n(x)$ is bounded and uniformly elliptic on $B_{\frac{1}{\var_n}} \setminus \overline{B_{R_0}}$. Consequently, the existence and uniqueness of $H_n$ follow directly from the Lax-Milgram theorem.

Since $\psi_n = 0$ on $\partial B_{\frac{1}{\var_n}}$, the boundary values of $H_n$ on $\partial \left(B_{\frac{1}{\var_n}} \setminus \overline{B_{R_0}}\right)$ satisfy $
|H_n| \le M_0. $
Testing \eqref{466p} with $(H_n - M_0)_+ \in H_0^1(B_{\frac{1}{\var_n}} \setminus \overline{B_{R_0}})$, we obtain
$$
\int_{\{H_n > M_0\}}\sum_{i,j=1}^{N}  a_{ij}^n(x)\partial_j H_n \partial_i H_n \,\mathrm{d}x = 0.
$$
By the strict ellipticity of $ a_{ij}^n(x)$, it follows that $(H_n - M_0)_+ = 0$. Analogously, taking $(-M_0 - H_n)_+$ as a test function yields $H_n \ge -M_0$. Therefore, we establish
\begin{equation}\label{467p}
\|H_n\|_{L^\infty(B_{\frac{1}{\var_n}} \setminus \overline{B_{R_0}})} \le M_0.
\end{equation}
Subtracting \eqref{466p} from the weak formulation of \eqref{w}, we deduce that
\begin{equation}\label{468p}
\int_{B_{\frac{1}{\var_n}} \setminus \overline{B_{R_0}}} \sum_{i,j=1}^{N}a_{ij}^n(x) \partial_j(\psi_n - H_n)\partial_i\phi \,\mathrm{d}x
= \int_{B_{\frac{1}{\var_n}} \setminus \overline{B_{R_0}}} |x|^{N\alpha_n}\zeta_n\psi_n \phi \,\mathrm{d}x \quad \text{for all } \phi \in H_0^1(B_{\frac{1}{\var_n}} \setminus \overline{B_{R_0}}).
\end{equation}
Testing \eqref{468p} with $\phi = \psi_n - H_n \in H_0^1(B_{\frac{1}{\var_n}} \setminus \overline{B_{R_0}})$ and applying \eqref{n6}, we obtain
\begin{align}\label{469p}
&\quad C_2\int_{B_{\frac{1}{\var_n}} \setminus \overline{B_{R_0}}} |x|^{-(N-2)} |\nabla (\psi_n-H_n)|^2 \,\md  x\\
&\quad\le
\int_{B_{\frac{1}{\var_n}} \setminus \overline{B_{R_0}}} |x|^{N\alpha_n}\zeta_n(\psi_n-H_n)^2\,\mathrm{d}x+ \int_{B_{\frac{1}{\var_n}} \setminus \overline{B_{R_0}}} |x|^{N\alpha_n}\zeta_nH_n(\psi_n-H_n) \,\mathrm{d}x.\nonumber
\end{align}
In view of $\al_n \to \al$, as $n\to+\infty$, up to increasing $n_0$, we can assume that
$\frac{N(\al_n+1)}{N-1} \ge \frac{N(\al+2)}{2(N-1)}$ for every $n \ge n_0.$ Therefore, we deduce that
\begin{equation}\label{470p}
0 \le |x|^{N\alpha_n}\zeta_n \le C |x|^{-N-\frac{N(\al+2)}{2(N-1)}}, \qquad |x| \ge 1, \ n \ge n_0.
\end{equation}
For any $|x| \ge R_0$, we have $
|x|^{-N-\frac{N(\al+2)}{2(N-1)}} \le R_0^{-\frac{N(\al+1)}{2(N-1)}} |x|^{-N}[\ln|x|]^{-2}.$
Consequently, by \eqref{470p} and the truncated weighted Hardy-Sobolev inequality \eqref{aa436+}, it follows that
\begin{align}\label{471p}
\int_{B_{\frac{1}{\var_n}} \setminus \overline{B_{R_0}}} |x|^{N\alpha_n}\zeta_n(\psi_n-H_n)^2\,\mathrm{d}x
&\le C R_0^{-\frac{N(\al+1)}{2(N-1)}}\int_{B_{\frac{1}{\var_n}} \setminus \overline{B_{R_0}}} [\ln|x|]^{-2}|x|^{-N}(\psi_n-H_n)^2 \,\mathrm{d}x \nonumber \\
&\le C R_0^{-\frac{N(\al+1)}{2(N-1)}}\int_{B_{\frac{1}{\var_n}} \setminus \overline{B_{R_0}}} |x|^{-(N-2)} |\nabla (\psi_n-H_n)|^2 \,\md  x.
\end{align}

On the other hand, a direct calculation yields
\begin{equation}\label{472p}
\int_{B_{\frac{1}{\var_n}} \setminus \overline{B_{R_0}}} |x|^{N\alpha_n}\zeta_n \,\mathrm{d}x
\le C \int_{\{|x| \ge R_0\}} |x|^{-N-\frac{N(\al+1)}{2(N-1)}} \,\mathrm{d}x
\le C R_0^{-\frac{N(\al+1)}{2(N-1)}}.
\end{equation}
In view of \eqref{467p}, H\"{o}lder's inequality, \eqref{471p}, and \eqref{472p}, we have
\begin{align}\label{473p}
&\quad \left| \int_{B_{\frac{1}{\var_n}} \setminus \overline{B_{R_0}}} |x|^{N\alpha_n}\zeta_nH_n(\psi_n-H_n) \,\mathrm{d}x \right| \\
&\le M_0 \left( \int_{B_{\frac{1}{\var_n}} \setminus \overline{B_{R_0}}} |x|^{N\alpha_n}\zeta_n \,\mathrm{d}x \right)^{\frac12} \left( \int_{B_{\frac{1}{\var_n}} \setminus \overline{B_{R_0}}} |x|^{N\alpha_n}\zeta_n(\psi_n-H_n)^2 \,\mathrm{d}x \right)^{\frac12} \nonumber \\
&\le C M_0 R_0^{-\frac{N(\al+1)}{2(N-1)}} \lr\int_{B_{\frac{1}{\var_n}} \setminus \overline{B_{R_0}}} |x|^{-(N-2)} |\nabla (\psi_n-H_n)|^2 \,\md  x\rr^{\frac12}. \nonumber
\end{align}
We now choose $R_0> 8$ sufficiently large such that $
C R_0^{-\frac{N(\al+1)}{2(N-1)}}\le \frac{C_2}{2}$ in \eqref{471p}.
Combining \eqref{469p}, \eqref{471p} and \eqref{473p}, we deduce that
\begin{align}\label{p474}
\int_{B_{\frac{1}{\var_n}} \setminus \overline{B_{R_0}}} |x|^{-(N-2)} |\nabla (\psi_n-H_n)|^2 \,\md  x\le C M_0^{2} R_0^{-\frac{N(\al+1)}{N-1}}.
\end{align}
For any $x \in B_{8R} \setminus B_R$ with $R > R_0$ sufficiently large, applying the truncated weighted Hardy-Sobolev inequality \eqref{aa436+}, we obtain from \eqref{p474} that
\begin{align}\label{475p}
R^{-N} \int_{B_{8R}\setminus B_{R}}(\psi_n-H_n)^2 \,\mathrm{d}x
&\le 4[\ln R]^2 \int_{B_{\frac{1}{\var_n}} \setminus \overline{B_{R_0}}}[\ln|x|]^{-2}|x|^{-N}(\psi_n-H_n)^2 \,\mathrm{d}x \nonumber \\
&\le 16 [\ln R]^{2} \int_{B_{\frac{1}{\var_n}} \setminus \overline{B_{R_0}}} |x|^{-(N-2)} |\nabla (\psi_n-H_n)|^2 \,\md  x\\
&\le C M_0^{2} R_0^{-\frac{N(\al+1)}{N-1}}[\ln R]^{2}. \nonumber
\end{align}

On the other hand, in view of \eqref{467p}, we have
\begin{equation}\label{476p}
R^{-N} \int_{B_{8R}\setminus B_{R}} H_n^2 \,\mathrm{d}x \le C M_0^2.
\end{equation}
Combining \eqref{475p} and \eqref{476p} yields that, for any $n\geq n_0$ and $R_0>8$ sufficiently large,
\begin{equation}\label{477p}
R^{-N} \int_{B_{8R}\setminus B_{R}} \psi_n^2 \,\mathrm{d}x
\le CM_{0}^{2} + C M_0^{2} R_0^{-\frac{N(\al+1)}{N-1}}[\ln R]^{2}, \qquad \forall \,\, R\geq R_0,
\end{equation}
where $C>0$ is independent of $n$, $R$, $R_0$ and $M_{0}(R_0)$. Next, we will determine $M_0(R_0)$.

Now, due to the arbitrariness of $R_0$, taking $R_0=100$ (say), for any $|x|>2R_0=200$, we can choose $R>\frac{2R_{0}}{7}=\frac{200}{7}$ such that $x\in B_{7R}\setminus B_{2R}$. Thanks to the $(\ln R)^{2}$-type uniform decay estimate \eqref{477p} for the $L^2$-integral average of $\psi_n$, by using the Sobolev embedding inequality for $N\geq3$ and the Moser-Trudinger inequality for $N=2$ and applying the De Giorgi-Moser-Nash iteration argument, we can show via exactly the same method as Proposition \ref{ppq43} that (see \eqref{425hh} and \eqref{425hw}), for any $n\geq n_0$,
\begin{align}\label{e0+}
		|\psi_n(x)|\leq \|\psi_n\|_{L^\infty(B_{7R}\setminus B_{2R})}&\leq\frac{C}{R^{\frac{N}{2}}}\left(\int_{B_{\frac{15}{2}R}\setminus B_{\frac32R}} |\psi_n|^{2}\md x\right)^{\frac{1}{2}}
		\leq CR^{-\frac N 2+\frac {N}{2}}\lr [\ln R]^{2}\rr^{\frac{1}{2}}\\
		&\leq C\ln|x|,    \qquad \forall \,\, |x|>2R_0=200, \nonumber
	\end{align}
where $C>0$ is independent of $n$ and $n_0\geq1$.

Therefore, by \eqref{464p} and \eqref{e0+}, we can determine $M_0=M_0(R_0)=C\ln R_0$, and hence \eqref{477p} yields that, for any $n\geq n_0$ and any $R_0>8$ sufficiently large,
\begin{equation}\label{477p+}
R^{-N} \int_{B_{8R}\setminus B_{R}} \psi_n^2 \,\mathrm{d}x
\le C[\ln R_{0}]^{2} + C [\ln R_{0}]^{2} R_0^{-\frac{N(\al+1)}{N-1}}[\ln R]^{2}, \qquad \forall \,\, R\geq R_0,
\end{equation}
where $C>0$ is independent of $n$, $R$, $R_0$. This finishes our proof.
\end{proof}

By considering the equation satisfied by $\psi_n$ directly, using the uniform upper bound \eqref{wL2+} on $L^2$-integral average of $\psi_n$ and applying the De Giorgi-Moser-Nash iteration argument, we overcame the nonlinear nature of the $N$-Laplacian $\Delta_N$, absence of Kelvin type transforms, lack of the Green integral representation formula and the divergence (to $-\infty$) at $\infty$ of the solutions $u_n$ and $v_n$, and proved the following uniform fast decay estimate for $h_n$.
\begin{prop}\label{ppq43}
Let $\al_n$ and $\var_n$ be sequences such that $\al_n\to\al>0$ and $\var_n\to0$ as $n\to+\infty$. Let $v_n$ be a sequence of nonradial solutions of \eqref{41} in $B_{\frac{1}{\var_n}}$ related to the exponent $\al_n$, and $u_n$ be a sequence of radial solution of \eqref{41} related to the exponent $\al_n$ given by \eqref{32}. If $\|e^{v_n}-e^{\beta_{\al_n}
}\|_\gamma\leq A$ and $\inf\limits_{\overline{B_1}}v_n\geq B$ for some positive constant $A$ and constant $B$ independent of $n$, then there exists $n_0\geq1$ such that, for every $n\geq n_0$ and any $R_0>8$ sufficiently large,
\begin{align}\label{412}
	|h_n(x)|\leq C\left(\ln R_{0} + (\ln R_{0}) R_0^{-\frac{N(\al+1)}{2(N-1)}}\ln |x|\right)e^{U_{\al_n}(x)}, \qquad \,\,\, \forall \,\,  |x| > 2R_0,
\end{align}
where $C>0$ independent of $n$, $n_0$ and $R_0$.
\end{prop}

\begin{proof}
For any $x \in \mathbb{R}^N$ with $|x| > 2R_0$, we can choose $R > R_0$ sufficiently large such that $x \in B_{7R} \setminus B_{2R}$. Recall that $\psi_n$ satisfies the equation \eqref{w}. By \eqref{n5} and \eqref{n6}, we get, for any $\xi\in\R^N$ and $n\geq n_0$,
\begin{equation}\label{n5+}
  \left|\sum_{i,j=1}^{N}a_{ij}^n(x)\partial_{x_j}\psi_n\xi_i\right|\leq C_{1}\frac{|\nabla \psi_n||\xi|}{|x|^{N-2}},
\end{equation}
and
\begin{align}\label{n6+}
\sum_{i,j=1}^{N}a_{ij}^n(x)\partial_{x_i}\psi_n\partial_{x_j}\psi_n\geq C_{2} \frac{|\nabla \psi_n|^2}{|x|^{N-2}},
\end{align}
where $C_{1}>0$ and $C_{2}>0$ are uniform constants independent of $n$.

Now we take
	\begin{align*}
		F(\overline w_n)
		=\begin{cases}
			\overline w_n^q,& \overline w_n\leq l,\\
			q l^{q-1}\overline w_n-(q-1) l^q, &\overline w_n> l
		\end{cases}
	\end{align*}
	for some $q\geq1,\,l\in\R^+$, and
	\begin{align*}
		G(w_n)
		=\text{sgn}(\psi_n)F(\overline \psi_n)F'(\overline \psi_n),
	\end{align*}
where $\overline \psi_n=|\psi_n|$. So
	\begin{align*}
		F'(\overline \psi_n)
		=\begin{cases}
			q\overline \psi_n^{q-1},& \overline \psi_n\leq l,\\
			q l^{q-1}, &\overline \psi_n> l
		\end{cases}
	\end{align*}
	and
	\begin{align*}
		G'(\psi_n)
		=\begin{cases}
			q^{-1}(2q-1)\left( F'(\overline \psi_n)\right)^2,& \overline w_n\leq l,\\
			\left( F'(\overline \psi_n)\right)^2, &\overline w_n> l.
		\end{cases}
	\end{align*}

For any $\delta\in(0,\frac{R}{2})$, we can choose cut-off function $\varphi(x)$ satisfying $0\leq\varphi(x)\leq1$, $\varphi(x)=1$ in $B_{8R-\delta}\setminus B_{R+\delta}$, $\varphi(x)=0$ in $(\R^N\setminus B_{8R})\cup B_{R}$, $0<\varphi(x)<1$ in $(B_{R+\delta}\setminus B_{R})\cup (B_{8R}\setminus B_{8R-\delta})$ and $|\nabla\varphi|\leq\frac{C}{\delta}$.
Let $$\kappa_n(x):=\varphi^2(x)G(\psi_n),$$
then
	$$\partial_{x_i}\left(\kappa_n(x)\right)=2\varphi\partial_{x_i}\varphi G(\psi_n)+\varphi^2G'(\psi_n)\partial_{x_i} \psi_n.$$	
We get
	\begin{align}
		&\quad \sum_{i,j=1}^{N}a_{ij}^n(x)\partial_{x_j}\psi_n\partial_{x_i}\kappa_n-|x|^{N\al_n}\zeta_n(x)\psi_n\kappa_n\label{423yy}\\
		&=\sum_{i,j=1}^{N}a_{ij}^n(x)\partial_{x_j}\psi_n\left(2\varphi\partial_{x_i}\varphi G(\psi_n)+\varphi^2G'(\psi_n)\partial_{x_i} \psi_n\right)-|x|^{N\al_n}\zeta_n(x)\psi_n\varphi^2(x)G(\psi_n)\nonumber\\
		&\geq \sum_{i,j=1}^{N}\varphi^2G'(\psi_n)a_{ij}^n(x)\partial_{x_j}\psi_n\partial_{x_i}\psi_n-2\varphi|\nabla\varphi||G(\psi_n)||(N-1)\eta_n(x)||\nabla \psi_n|\nonumber\\
		&\quad-|x|^{N\al_n}\zeta_n(x)|\psi_n|\varphi^2(x)|G(\psi_n)|\nonumber\\
		&\geq C(N)\varphi^2|F'(\overline \psi_n)|^2|\eta_n(x)||\nabla \psi_n|^2-2\varphi|\nabla\varphi||F(\overline \psi_n)||F'(\overline \psi_n)||(N-1)\eta_n(x)||\nabla \psi_n|\nonumber\\
		&\quad-|x|^{N\al_n}\zeta_n(x)|\psi_n|\varphi^2(x)|F(\overline \psi_n)||F'(\overline \psi_n)|\nonumber\\
		&= C(N)\varphi^2|\eta_n(x)||F'(\overline \psi_n)\nabla \overline \psi_n|^2-2(N-1)|\nabla\varphi F(\overline \psi_n)||\varphi F'(\overline \psi_n)\eta_n(x)\nabla \overline \psi_n|\nonumber\\
		&\quad-|x|^{N\al_n}\zeta_n(x)|\varphi F(\overline \psi_n)||\varphi F'(\overline \psi_n)\overline \psi_n|\nonumber.
	\end{align}	
	Let $z_n=F(\overline \psi_n)$, and note that $\overline \psi_nF'(\overline \psi_n)\leq qF(\overline \psi_n)$. From
	\begin{align*}
		\int_{B_{8R}\setminus B_{R}}\sum_{i,j=1}^{N} a_{ij}^n(x)\partial_{x_j}\psi_n\partial_{x_i}\kappa_n-|x|^{N\al_n}\zeta_n(x)\psi_n\kappa_n(x)\md x=0,
	\end{align*}	
by \eqref{423yy}, we get
	\begin{align*}
		&\quad \int_{B_{8R}\setminus B_{R}} |\eta_n(x)||\varphi\nabla z_n|^2\md x\\&\leq	C\int_{B_{8R}\setminus B_{R}}|\nabla\varphi \cdot z_n||\varphi\eta_n\cdot\nabla z_n|+Cq|x|^{N\alpha_n}\zeta_n(x)|\varphi z_n|^2\md x\\
		&\leq  C\left(	\int_{B_{8R}\setminus B_{R}}|\eta_n(x)||\nabla\varphi|^2 | z_n|^2\md x\right)^{\frac12}\left(	\int_{B_{8R}\setminus B_{R}}|\eta_n(x)||\varphi|^2|\nabla z_n|^2\md x\right)^{\frac12}\\
		&\quad+Cq\int_{B_{8R}\setminus B_{R}}|x|^{N\alpha_n}\zeta_n(x)|\varphi z_n|^2\md x.
	\end{align*}	
Then, we get
\begin{align*}
	\int_{B_{8R}\setminus B_{R}} |\eta_n(x)||\varphi\nabla z_n|^2\md x
	\leq  C\int_{B_{8R}\setminus B_{R}}|\eta_n(x)||\nabla\varphi|^2 | z_n|^2\md x+Cq\int_{B_{8R}\setminus B_{R}}|x|^{N\alpha_n}\zeta_n(x)|\varphi z_n|^2\md x.
\end{align*}	
By Proposition \ref{pp41}, Proposition \ref{pro43}, \eqref{435}, \eqref{436} and \eqref{437}, we derive
	\begin{align*}
\int_{B_{8R}\setminus B_{R}} |\eta_n(x)||\varphi\nabla z_n|^2\md x\geq \frac{C}{R^{N-2}}\int_{B_{8R}\setminus B_{R}} |\varphi\nabla z_n|^2\md x
	\end{align*}		
and
\begin{align*}
	&\quad\int_{B_{8R}\setminus B_{R}}|\eta_n||\nabla\varphi|^2 | z_n|^2\md x+Cq\int_{B_{8R}\setminus B_{R}}|x|^{N\alpha_n}\zeta_n(x)|\varphi z_n|^2\md x\\
	&\leq\frac{C}{R^{N-2}}\int_{B_{8R}\setminus B_{R}}|\nabla\varphi|^2 | z_n|^2\md x+\frac{Cq}{R^{\frac{N(N+\al_n)}{N-1}}}\int_{B_{8R}\setminus B_{R}}|\varphi z_n|^2\md x.
\end{align*}		
	Furthermore, we have
	\begin{align}\label{424rr}
		\int_{B_{8R}\setminus B_{R}} |\varphi\nabla z_n|^2\md x
		\leq C	\int_{B_{8R}\setminus B_{R}}|\nabla\varphi|^2 | z_n|^2\md x+\frac{Cq}{R^{\frac{3N+N\al_n-2}{N-1}}}\int_{B_{8R}\setminus B_{R}}|\varphi z_n|^2\md x.
	\end{align}	
\smallskip

Next, we will discuss $N\geq3$ and $N=2$ separately.

\smallskip

If $N\geq 3$, by the Sobolev embedding inequality, \eqref{424rr}, $|\varphi|\leq1$ and $|\nabla\varphi|\leq\frac{C}{\delta}$, we get
	\begin{align*}
		\left(\int_{B_{8R}\setminus B_{R}} |\varphi z_n|^{\frac{2N}{N-2}}\md x\right)^{\frac{N-2}{N}}&\leq C\int_{B_{8R}\setminus B_{R}} |\varphi\nabla z_n|^2\md x+C	\int_{B_{8R}\setminus B_{R}}|\nabla\varphi|^2 | z_n|^2\md x\\
		&\leq C\left(\frac{1}{\delta^2}+\frac{q}{R^{\frac{3N+N\al_n-2}{N-1}}}\right)	\int_{B_{8R}\setminus B_{R}} | z_n|^2\md x.
	\end{align*}	
So, we have
	\begin{align*}
		\left(\int_{B_{8R-\delta}\setminus B_{R+\delta}} | z_n|^{\frac{2N}{N-2}}\md x\right)^{\frac{N-2}{N}}
		\leq Cq\left(\frac{1}{\delta^2}+\frac{1}{R^{\frac{3N+N\al_n-2}{N-1}}}\right)	\int_{B_{8R}\setminus B_{R}} | z_n|^2\md x
	\end{align*}	
and
	\begin{align*}
		\left(\int_{B_{8R-\delta}\setminus B_{R+\delta}} | \overline \psi_n|^{q2^*}\md x\right)^{\frac{1}{q2^*}}\leq \left(\frac{Cq}{\delta^2}\right)	^{\frac{1}{2q}}\left(\int_{B_{8R}\setminus B_{R}} | \overline \psi_n|^{2q}\md x\right)^{\frac{1}{2q}}.
	\end{align*}	
Take
	$$\gamma=2, \,\,\,\, \chi=\frac{N}{N-2},\,\,\,\, \gamma_i=\chi^i\gamma,\,\,\,\, h_i=\sum_{j=0}^{i}\frac{R}{2^{j+1}} \,\,\,\, \text{for}\ i=0,1,\cdots,$$
thus we obtain
		\begin{align*}
		&\quad\left(\int_{B_{8R-h_{i+1}}\setminus B_{R+h_{i+1}}} | \overline \psi_n|^{\gamma_{i+1}}\md x\right)^{\frac{1}{\gamma_{i+1}}}\\&\leq\left(\frac{C\chi^i}{(h_i-h_{i+1})^2}\right)	^{\frac{1}{\gamma_{i}}}\left(\int_{B_{8R-h_{i}}\setminus B_{R+h_{i}}} | \overline \psi_n|^{\gamma_{i}}\md x\right)^{\frac{1}{\gamma_{i}}}\\
		&\leq \left(\frac{C\chi^i}{(\frac{R}{2^{i+2}})^2}\right)	^{\frac{1}{\gamma_{i}}}\left(\int_{B_{8R-h_{i}}\setminus B_{R+h_{i}}} | \overline \psi_n|^{\gamma_{i}}\md x\right)^{\frac{1}{\gamma_{i}}}\\
&\leq\cdots\cdots \nonumber\\
		&\leq \frac{C^{\sum^{i}_{k=0}\frac{1}{\gamma_k}}\chi^{\sum^{i}_{k=0}\frac{k}{\gamma_k}}4^{\sum^{i}_{k=0}\frac{k+2}{\gamma_k}}}{R^{\sum^{i}_{k=0}\frac{2}{\gamma_k}}}\left(\int_{B_{\frac{15}{2}R}\setminus B_{\frac32R}} | \overline \psi_n|^{\gamma_{0}}\md x\right)^{\frac{1}{\gamma_{0}}}.
	\end{align*}	
Letting $i\to+\infty$, by \eqref{wL2+}, we have, for any $n\geq n_0$, and any $R_0>8$ sufficiently large,
	\begin{align}\label{425hh}
		|\psi_n(x)|&\leq \|\overline \psi_n\|_{L^\infty(B_{7R}\setminus B_{2R})}\leq\frac{C}{R^{\frac{N}{2}}}\left(\int_{B_{\frac{15}{2}R}\setminus B_{\frac32R}} | \overline \psi_n|^{2}\md x\right)^{\frac{1}{2}}\\
		&\leq CR^{-\frac N 2+\frac {N}{2}}\lr C[\ln R_{0}]^{2} + C [\ln R_{0}]^{2} R_0^{-\frac{N(\al+1)}{N-1}}[\ln R]^{2} \rr^{\frac{1}{2}}\nonumber\\
		&\leq C\left(\ln R_{0} + (\ln R_{0}) R_0^{-\frac{N(\al+1)}{2(N-1)}}\ln |x|\right), \qquad\forall \,\, |x|>2R_0.  \nonumber
	\end{align}
	
\smallskip

For $N = 2$, we employ the Moser-Trudinger inequality \cite[Theorem 7.15]{GT} instead of the Sobolev embedding.
We derive that, for any integer $K \ge 2$,
\begin{align}\label{a2}
&\quad \int_{B_{8R}\setminus B_{R}}\frac{1}{K!}\left(\frac{|\varphi z_{n}|}{c_{1}\|D(\varphi z_{n})\|_{L^{2}(B_{8R}\setminus B_{R})}}\right)^{2K}\mathrm{d}x \\
  \nonumber &\leq\int_{B_{8R}\setminus B_{R}}\exp\left[\left(\frac{|\varphi z_{n}|}{c_{1}\|D(\varphi z_{n})\|_{L^{2}(B_{8R}\setminus B_{R})}}\right)^{2}\right]\mathrm{d}x\leq CR^{2}.
\end{align}
Therefore, by \eqref{424rr}, \eqref{a2}, $|\varphi|\leq1$ and $|\nabla\varphi|\leq\frac{C}{\delta}$, we have
	\begin{align*}
		\lr\int_{B_{8R}\setminus B_{R}} |\varphi z_n|^{2K}\md x\rr^{\frac{1}{K}}&\leq CR^{\frac{2}{K}}\left(\int_{B_{8R}\setminus B_{R}} |\varphi\nabla z_n|^2\md x+	\int_{B_{8R}\setminus B_{R}}|\nabla\varphi|^2 | z_n|^2\md x\right)\\
		&\leq CR^{\frac{2}{K}}\left(\frac{1}{\delta^2}+\frac{q}{R^{4+2\al_n}}\right)		\int_{B_{8R}\setminus B_{R}} | z_n|^2\md x.
	\end{align*}	
	Thus
	\begin{align*}
		\lr\int_{B_{8R-\delta}\setminus B_{R+\delta}} | z_n|^{2K}\md x\rr^{\frac{1}{K}}
		\leq CqR^{\frac{2}{K}}\lr\frac{1}{\delta^2}+\frac{1}{R^{4+2\al_n}}\rr	\int_{B_{8R}\setminus B_{R}} | z_n|^2\md x,
	\end{align*}	
and hence
	\begin{align*}
		\lr\int_{B_{8R-\delta}\setminus B_{R+\delta}} | \overline \psi_n|^{2Kq}\md x\rr^{\frac{1}{2Kq}}\leq \lr\frac{CqR^{\frac{2}{K}}}{\delta^2}\rr	^{\frac{1}{2q}}\lr\int_{B_{8R}\setminus B_{R}} | \overline \psi_n|^{2q}\md x\rr^{\frac{1}{2q}}.
	\end{align*}	
	Let
	$$\gamma=2, \,\,\,\, \chi=K,\,\,\,\, \gamma_i=\chi^i\gamma,\,\,\,\, h_i=\sum_{j=0}^{i}\frac{R}{2^{j+1}} \,\,\,\,\,\, \text{for}\,\, i=0,1,2,\cdots,$$
	then, we get
		\begin{align*}
		&\quad\lr\int_{B_{8R-h_{i+1}}\setminus B_{R+h_{i+1}}} | \overline \psi_n|^{\gamma_{i+1}}\md x\rr^{\frac{1}{\gamma_{i+1}}}\\&\leq\lr\frac{C\chi^{i}R^{\frac{2}{K}}}{(h_{i+1}-h_{i})^2}\rr^{\frac{1}{\gamma_{i}}}\lr\int_{B_{8R-h_{i}}\setminus B_{R+h_{i}}} | \overline \psi_n|^{\gamma_{i}}\md x\rr^{\frac{1}{\gamma_{i}}}\\
		&\leq \lr\frac{C\chi^{i}R^{\frac{2}{K}}}{(\frac{R}{2^{i+2}})^2}\rr	^{\frac{1}{\gamma_{i}}}\lr\int_{B_{8R-h_{i}}\setminus B_{R+h_{i}}} | \overline \psi_n|^{\gamma_{i}}\md x\rr^{\frac{1}{\gamma_{i}}}\\
&\leq \cdots\cdots \nonumber \\
&\leq \frac{C^{\sum^{i}_{k=0}\frac{1}{\gamma_k}}\chi^{\sum^{i}_{k=0}\frac{k}{\gamma_k}}4^{\sum^{i}_{k=0}\frac{k+2}{\gamma_k}}}{R^{\left(1-\frac{1}{K}\right)\sum^{i}_{k=0}\frac{2}{\gamma_k}}}\lr\int_{B_{\frac{15}{2}R}\setminus B_{\frac32R}} | \overline \psi_n|^{\gamma_{0}}\md x\rr^{\frac{1}{\gamma_{0}}}.
	\end{align*}
Letting $i\to+\infty$, from \eqref{wL2+}, we get	
	\begin{align}\label{425hw}
		|\psi_n(x)|\leq \|\overline \psi_n\|_{L^\infty(B_{7R}\setminus B_{2R})}&\leq\frac{C}{R}\lr\int_{B_{\frac{15}{2}R}\setminus B_{\frac32R}} | \overline \psi_n|^{2}\md x\rr^{\frac{1}{2}}\\
		&\leq CR^{-1+ 1}\lr C[\ln R_{0}]^{2} + C [\ln R_{0}]^{2} R_0^{-\frac{N(\al+1)}{N-1}}[\ln R]^{2} \rr^{\frac12}\nonumber\\
		&\leq C\left(\ln R_{0} + (\ln R_{0}) R_0^{-\frac{N(\al+1)}{2(N-1)}}\ln |x|\right),\qquad \forall \,\, |x|>2R_0.\nonumber
	\end{align}	

\medskip

Finally, combining \eqref{xi425}, \eqref{425hh} and \eqref{425hw}, we obtain, for every $n\geq n_0$ and any $R_0>8$ sufficiently large,
$$
\begin{aligned}
|h_n(x)|= \zeta_n(x) |\psi_n(x)|\leq C\left(\ln R_{0} + (\ln R_{0}) R_0^{-\frac{N(\al+1)}{2(N-1)}}\ln |x|\right)e^{U_{\al_n}(x)}, \qquad \forall \,\, |x|>2R_0,
\end{aligned}
$$
where $C>0$ independent of $n$, $n_0$ and $R_0$. We finishes our proof of Proposition \ref{ppq43}.
\end{proof}

For any $R>0$, we can overcome the lack of critical weighted Sobolev embedding inequality and prove the following (almost critical) truncated weighted Hardy-Sobolev inequality on exterior domain $\R^N\setminus B_{R}(0)$ instead, which is important for us to derive a uniform decay estimate for the $L^2$-integral average of $w_{n}$ in the annulus $B_{8R}(0)\setminus B_{R}(0)$ near $\partial B_{\frac{1}{\var_n}}$.
\begin{prop}\label{tws}
Assume $N\geq2$ and $R>0$. For any $\varphi\in \dot{H}^{1}(\R^N\setminus \overline{B_{R}(0)};\,|x|^{-(N-2)}\left[\ln(R^{-1}|x|)\right]^{2})\cap L^2(\R^N\setminus \overline{B_{R}(0)};\,|x|^{-N})$, we have
	\begin{align}\label{aa436}
		\int_{\R^N\setminus B_{R}(0)}	\frac{|\varphi|^2}{|x|^{N}}\md x\leq
		4\int_{\R^N\setminus B_{R}(0)}\frac{|\nabla\varphi|^{2}}{|x|^{N-2}}\left[\ln\left(\frac{|x|}{R}\right)\right]^{2}\md x.
		\end{align}
\end{prop}
\begin{proof}
To prove \eqref{aa436}, we can assume by approximation that $\varphi\in C_{c}^1(\R^N)$. From Fubini's theorem and using polar coordinates, we obtain
		\begin{align}\label{a436}	
	&\quad \int_{\R^N\setminus B_{R}(0)}	\frac{|\varphi|^2}{|x|^{N}}\md x =\int_{\Sm^{N-1}}\int_R^{+\infty}r^{-1}|\varphi(r\theta)|^2\md r\md\theta  \\
		&\leq 2\int_{\Sm^{N-1}}\int_R^{+\infty}r^{-1}\int_r^{+\infty}|\varphi(t\theta)||\nabla\varphi(t\theta)|\md t\md r\md\theta\nonumber\\ &=2\int_{\Sm^{N-1}}\int_R^{+\infty}|\varphi(t\theta)||\nabla\varphi(t\theta)|\int_R^{t}r^{-1}\md r\md t\md\theta\nonumber\\
		&=2\int_{\Sm^{N-1}}\int_R^{+\infty}\ln\left(\frac{t}{R}\right)|\varphi(t\theta)||\nabla\varphi(t\theta)|\md t\md\theta\nonumber\\
		&\leq 2\lr\int_{\Sm^{N-1}}\int_R^{+\infty}t^{-1}|\varphi(t\theta)|^2\md t\md\theta\rr^{\frac12\nonumber}\lr\int_{\Sm^{N-1}}\int_R^{+\infty}t\left[\ln\left(\frac{t}{R}\right)\right]^{2}|\nabla\varphi(t\theta)|^2\md t\md\theta\rr^{\frac12}\nonumber \\
&= 2\lr\int_{\R^N\setminus B_{R}(0)}	\frac{|\varphi|^2}{|x|^{N}}\md x\rr^{\frac12}\lr\int_{\R^N\setminus \overline{B_{R}(0)}}\frac{|\nabla\varphi|^{2}}{|x|^{N-2}}\left[\ln\left(\frac{|x|}{R}\right)\right]^{2}\md x\rr^{\frac12}.\nonumber
	\end{align}
Hence \eqref{aa436} follows from \eqref{a436} immediately. This finishes our proof.
\end{proof}

By using the uniform $(\ln R)^{-1}$-type decay estimate \eqref{ln 420} on the integral of $|x|^{-(N-2)}|\nabla w_n|^2$ and applying the truncated weighted Hardy-Sobolev inequality \eqref{aa436}, we overcame the nonlinear nature of the $N$-Laplacian $\Delta_N$ and the divergence (to $-\infty$) at $\infty$ of the solutions $u_n$ and $v_n$, and proved the following uniform $(\ln R)^{-1}$-type decay estimate for the $L^2$-integral average of $w_n$ on annulus $B_{8R}(0)\setminus B_{R}(0)$ near $\partial B_{\frac{1}{\var_n}}$.
\begin{prop}\label{qpp44'}
	Assume $\al>0$ and $ N\geq2$. If $\|e^{v_n}-e^{\beta_{\al_n}
}\|_\gamma\leq A$ and $\inf\limits_{\overline{B_1}}v_n\geq B$ for some positive constant $A$ and constant $B$ independent of $n$, then we have, there exists $n_0\geq1$ such that
	\begin{align}\label{wL2}
		\frac{1}{R^N}\int_{B_{8R}(0)\setminus B_{R}(0)}|w_n|^2\md x\leq \frac{C}{\ln R}, \qquad \forall \,\, \frac{1}{14\var_n}<R<\frac{1}{2\var_n},\ \forall \,\, n\geq n_0,
	\end{align}
where $C>0$ is a constant independent of $n$ and $R$.
\end{prop}
\begin{proof}
By applying the truncated weighted Hardy-Sobolev inequality \eqref{aa436} with $\varphi=w_n$ and noting that $0\leq \ln\left(\frac{|x|}{R}\right)\leq \ln(14)$ for any $x\in B_{\frac{1}{\var_n}}(0)\setminus B_{R}(0)$, we deduce from Proposition \ref{ppp43} that
\begin{align}	\label{h1}
\frac{1}{(8R)^{N}}\int_{B_{8R}(0)\setminus B_{R}(0)}|w_{n}|^2\md x&\leq\int_{B_{8R}(0)\setminus B_{R}(0)}	\frac{|w_{n}|^2}{|x|^{N}}\md x
\leq\int_{\R^N\setminus B_{R}(0)}	\frac{|w_{n}|^2}{|x|^{N}}\md x \\
		&\leq 4\int_{B_{\frac{1}{\var_n}}(0)\setminus B_{R}(0)}\frac{|\nabla w_{n}|^{2}}{|x|^{N-2}}\left[\ln\left(\frac{|x|}{R}\right)\right]^{2}\md x \nonumber \\
&\leq 4[\ln(14)]^{2}\int_{B_{\frac{1}{\var_n}}(0)\setminus B_{R}(0)}\frac{|\nabla w_{n}|^{2}}{|x|^{N-2}}\md x\leq \frac{C}{\ln R}, \nonumber
\end{align}
where $C>0$ is a constant independent of $n$ and $R$. This finishes our proof of Proposition \ref{qpp44'}.
\end{proof}

By considering the equation satisfied by $w_n$ directly, using the uniform $(\ln R)^{-1}$-type decay estimate \eqref{wL2} on $L^2$-integral average of $w_n$ and applying the De Giorgi-Moser-Nash iteration argument, we overcame the nonlinear nature of the $N$-Laplacian $\Delta_N$, absence of Kelvin type transforms, lack of the Green integral representation formula and the divergence (to $-\infty$) at $\infty$ of the solutions $u_n$ and $v_n$, and proved the following uniform decay estimate for $w_n$.
\begin{prop}\label{p1}
Let $\al_n$ and $\var_n$ be sequences such that $\al_n\to\al>0$ and $\var_n\to0$ as $n\to+\infty$. If $\|e^{v_n}-e^{\beta_{\al_n}
}\|_\gamma\leq A$ and $\inf\limits_{\overline{B_1}}v_n\geq B$ for some positive constant $A$ and constant $B$ independent of $n$, then there exists a positive constant $C$ independent of $n$ and $n_0\geq1$ such that, for every $n\geq n_0$,
\begin{align}\label{412'}
	|w_n(x)|\leq \frac{C}{\sqrt{\ln|x|}}, \qquad \ \forall \,\, x\in \mathbb{R}^{N} \,\,\, \text{such that} \,\, \frac{1}{2\var_n}<|x|<\frac{1}{\var_n}.
\end{align}
\end{prop}
\begin{proof}
For any $x\in \mathbb{R}^{N}$ such that $\frac{1}{2\var_n}<|x|<\frac{1}{\var_n}$, there exists $\frac{1}{14\var_n}<R<\frac{1}{2\var_n}$ such that $x\in B_{7R}\setminus B_{2R}$. Thanks to the $(\ln R)^{-1}$-type uniform decay estimate \eqref{wL2} for the $L^2$-integral average of $w_n$ in Proposition \ref{qpp44'}, by using the Sobolev embedding inequality for $N\geq3$ and the Moser-Trudinger inequality for $N=2$ and applying the De Giorgi-Moser-Nash iteration argument, we can show via exactly the same method as Proposition \ref{ppq43} that, for any $n\geq n_0$,
\begin{align}\label{e0}
		|w_n(x)|\leq \|w_n\|_{L^\infty(B_{7R}\setminus B_{2R})}&\leq\frac{C}{R^{\frac{N}{2}}}\left(\int_{B_{\frac{15}{2}R}\setminus B_{\frac32R}} |w_n|^{2}\md x\right)^{\frac{1}{2}}
		\leq CR^{-\frac N 2+\frac {N}{2}}\lr\frac{1}{\ln R}\rr^{\frac{1}{2}}
		\leq \frac{C}{\sqrt{\ln|x|}},
	\end{align}
where $C>0$ is independent of $n$ and $n_0\geq1$. This finishes our proof of Proposition \ref{p1}.
\end{proof}

Furthermore, based on Proposition \ref{p1}, by using rescaling arguments and the interior and boundary gradient regularity
estimates, we can prove the following uniform fast decay estimate for $\nabla w_n(x)$, which is crucial in the proof of Propositions \ref{pp42} and \ref{pp44}.
\begin{prop}\label{pro26}
Let $\al_n$ and $\var_n$ be sequences such that $\al_n\to\al>0$ and $\var_n\to0$ as $n\to+\infty$. If $\|e^{v_n}-e^{\beta_{\al_n}}\|_\gamma\leq A$ and $\inf\limits_{\overline{B_1}}v_n\geq B$ for some positive constant $A$ and constant $B$ independent of $n$, then there exists $C>0$ independent of $n$ and $n_0\geq 1$ such that
	\begin{align}\label{g4}
		|\nabla w_n(x)|\leq\frac{C\var_n}{\sqrt{|\ln\var_n|}},
		\qquad \, \forall \,\,   x\in \overline{B_{\frac{1}{\var_n}}}\setminus B_{\frac{3}{4\var_n}}, \quad \forall \,\, n\geq n_0.
	\end{align}
\end{prop}
\begin{proof}
From Proposition \ref{p1}, we infer that, there exists a uniform constant $C_0>0$ independent of $n$ and $n_0\geq1$ such that, for every $n\geq n_0$, $$|w_n(x)|\leq\frac{C_{0}}{\sqrt{|\ln\var_n|}}, \qquad \ \forall \,\, x\in \mathbb{R}^{N} \,\,\, \text{such that} \,\, \frac{1}{2\var_n}\leq|x|\leq\frac{1}{\var_n}.$$

From \eqref{e2}, we know that $w_n$ satisfies the equation
\begin{align}\label{w'}
\begin{cases}
\displaystyle-\sum_{i=1}^{N}\partial_{x_i}\left( \sum_{j=1}^{N}a_{ij}^n(x)\partial_{x_j}w_n\right)=|x|^{N\al_n}\zeta_n(x)w_n, &x\in B_{\frac{1}{\var_n}},\\
w_n=0, &x\in\partial B_{\frac{1}{\var_n}}.
\end{cases}
\end{align}

For all $\frac{1}{2}\leq|y|\leq1$, define $W_n(y):=\sqrt{|\ln\var_n|}w_n\left(\frac{y}{\var_n}\right)$ and hence $|W_n(y)|\leq C_0$. Moreover, by \eqref{w'}, $W_n(y)$ satisfies
\begin{equation}\label{g1}
	-\sum_{i=1}^{N}\partial_{y_i}\left( \sum_{j=1}^{N}a_{ij}^n\left(\frac{y}{\var_n}\right)\partial_{y_j}W_n\right)=\left(\frac{1}{\var_n}\right)^{2}\left|\frac{y}{\var_n}\right|^{N\al_n}\zeta_n\left(\frac{y}{\var_n}\right)W_n(y),  \qquad \forall \,\, \frac{1}{2}<|y|<1.
\end{equation}
Define
$$\widehat{a}_{ij}^n(y):=\left(\frac{1}{\var_n}\right)^{N-2}a_{ij}^n\left(\frac{y}{\var_n}\right),$$
then it follows from \eqref{n7}, \eqref{n5} and \eqref{n6} that
\begin{equation}\label{g3}
	|\widehat{a}_{ij}^n(y)|\leq \frac{C_{1}\left(\frac{1}{\var_n}\right)^{N-2}}{|y/\var_n|^{N-2}}\leq 2^{N-2}C_1,  \qquad \forall \,\, \frac{1}{2}<|y|<1,
\end{equation}
\begin{equation}\label{g2}
	C_{2}|\xi|^{2}\leq \frac{C_{2}}{|y|^{N-2}}|\xi|^{2}\leq\sum_{i=1}^{N}\sum_{j=1}^{N}\widehat{a}_{ij}^n(y)\xi_i\xi_j\leq \frac{C_{1}}{|y|^{N-2}}|\xi|^{2}\leq 2^{N-2}C_{1}|\xi|^{2},  \qquad \forall \,\, \frac{1}{2}<|y|<1,
\end{equation}
where $C_1$ and $C_2$ are independent of $n$ and $y$. Therefore, $W_n(y)$ satisfies the following uniformly elliptic equation
\begin{equation}\label{g31}
	-\sum_{i=1}^{N}\partial_{y_i}\left( \sum_{j=1}^{N}\widehat{a}_{ij}^n(y)\partial_{y_j}W_n\right)=\left(\frac{1}{\var_n}\right)^{N}\left|\frac{y}{\var_n}\right|^{N\al_n}\zeta_n\left(\frac{y}{\var_n}\right)W_n(y),  \qquad \forall \,\, \frac{1}{2}<|y|<1
\end{equation}
with the right-hand side satisfying
\begin{align}\label{g32}
	&\quad \left|\left(\frac{1}{\var_n}\right)^{N}\left|\frac{y}{\var_n}\right|^{N\al_n}\zeta_n\left(\frac{y}{\var_n}\right)W_n(y)\right|\leq C\left(\frac{1}{\var_n}\right)^{N}\left|\frac{y}{\var_n}\right|^{N\al_n}e^{U_{\al_n}\left(\frac{y}{\var_n}\right)} \\
&\leq C\left(\frac{1}{\var_n}\right)^{N}\frac{\left|\frac{y}{\var_n}\right|^{N\al_n}}{\lr 1+\left|\frac{y}{\var_n}\right|^{\frac{N(\al_n+1)}{N-1}}\rr^N}
\leq C\left|\frac{y}{\var_n}\right|^{-\frac{N(\al_n+1)}{N-1}} \nonumber \\
&\leq C_{3}\var_n^{\frac{N}{N-1}}\leq C_{3},  \qquad \forall \,\, \frac{1}{2}<|y|<1, \nonumber
\end{align}
where $C_3$ is independent of $n$ and $y$. For $\frac{1}{2}\leq|y|\leq1$, let
\begin{equation}
\widetilde u_n(y):=u_n\left(\frac{y}{\var_n}\right)-\beta_{\al_n},
\qquad
\widetilde v_n(y):=v_n\left(\frac{y}{\var_n}\right)-\beta_{\al_n}.
\label{eq:tilde-uv}
\end{equation}
Then, both $\widetilde u_n$ and $\widetilde v_n$ vanish on $\partial B_1$, and
\begin{equation}\label{e3}
\begin{split}
-\Delta_N\widetilde u_n
&=\left(\frac{1}{\var_n}\right)^{N(1+\alpha_n)}e^{\beta_{\al_n}}|y|^{N\alpha_n}e^{\widetilde u_n},\\
-\Delta_N\widetilde v_n
&=\left(\frac{1}{\var_n}\right)^{N(1+\alpha_n)}e^{\beta_{\al_n}}|y|^{N\alpha_n}e^{\widetilde v_n}.
\end{split}
\end{equation}
One can easily verify that $\widetilde u_n$, $\widetilde v_n$ and the right-hand side of \eqref{e3} are uniformly bounded w.r.t. $n$. Therefore, the boundary gradient regularity estimate (c.f. e.g. Theorem 1 in \cite{Lieb}) gives some $\eta\in(0,1)$, independent of $n$, such that
\begin{equation}
\|\widetilde u_n\|_{C^{1,\eta}
(\overline{B_1\setminus B_{2/3}})}
+\|\widetilde v_n\|_{C^{1,\eta}
(\overline{B_1\setminus B_{2/3}})}
\leq C.
\label{eq:boundary-C1a}
\end{equation}
Therefore, noting that
\begin{align*}
\widehat{a}_{ij}^n(x):&=\int_0^1\Big[(N-2)|t\nabla \widetilde{u}_n+(1-t)\nabla \widetilde{v}_n|^{N-4}\partial_{x_i}(t\widetilde{u}_n+(1-t)\widetilde{v}_n)
\partial_{x_j}(t\widetilde{u}_n+(1-t)\widetilde{v}_n)\\
&\quad+|t\nabla \widetilde{u}_n+(1-t)\nabla \widetilde{v}_n|^{N-2}\delta_{ij}\Big]\md t,
\end{align*}
we have
\begin{equation}\label{e4}
  \|\widehat{a}_{ij}^n\|_{C^{0,\eta}(\overline{B_1\setminus B_{2/3}})}\leq (N-1)\left(\|\widetilde u_n\|_{C^{1,\eta}
(\overline{B_1\setminus B_{2/3}})}+\|\widetilde v_n\|_{C^{1,\eta}(\overline{B_1\setminus B_{2/3}})}\right)^{N-2}\leq C_{4},
\end{equation}
where $C_4$ is independent of $n$. Consequently, from the interior and boundary gradient regularity estimate (c.f. e.g. Theorems 8.32 and 8.33 in \cite{GT}, Theorem 1 in \cite{Lieb}), we have, for every $n\geq n_0$,
$$\|\nabla_y W_n(y)\|_{L^\infty(\overline{B_{1}(0)}\setminus B_{\frac{3}{4}}(0))}\leq C_5,$$
where $C_5$ is independent of $n$ and $n_0\geq 1$.

For all $x\in \overline{B_{\frac{1}{\var_n}}}\setminus B_{\frac{3}{4\var_n}}$, let $x=\frac{y}{\var_n}$, where $\frac{3}{4}\leq |y|\leq1$, thus we get
$$|\nabla_x w_n(x)|=\frac{\var_n}{\sqrt{|\ln\var_n|}}|\nabla_y W_n(y)|\leq \frac{C_{5}\var_n}{\sqrt{|\ln\var_n|}}.$$
This gives the desired estimate \eqref{g4} and hence completes the proof of Proposition \ref{pro26}.
\end{proof}

By using the uniform fast decay estimate \eqref{g4} for $\nabla w_n(x)$, the truncated weighted Hardy-Sobolev inequality \eqref{aa436+} and the Pohozaev identity \eqref{poh}, we overcame the invariance of the total mass $\int_{\mathbb{R}^{N}}|x|^{N\al}e^{U_{\la,\al}}\mathrm{d}x$ with respect to $\lambda>0$, the lack of critical weighted Sobolev embedding inequality and the lack of the Green integral representation formula, and proved the rigidity of the limiting scale.
\begin{prop}[Rigidity of the limiting scale]\label{pp42}
	Let $\al_n$ and $\var_n$ be sequences such that $\al_n\to\al>0$ and $\var_n\to0$ as $n\to+\infty$.
	Let $v_n$ be a sequence of solutions to \eqref{41} in $B_{\frac{1}{\var_n}}$, corresponding to the exponent $\al_n$. If $e^{v_n(x)}-e^{\beta_{\al_n}}\to e^{U_{\lambda,\al}}$ in $X$ as $n\to+\infty$, then we have $\la=1$.
\end{prop}
\begin{proof}
Since $e^{v_n(x)} - e^{\beta_{\alpha_n}} \to e^{U_{\lambda,\al}} \ \text{in } X \ \text{as } n\to+\infty$ for some $\lambda>0$, we have
$$\|e^{v_n(x)} - e^{U_{\lambda,\al}}\|_{L^\infty(B_4)}=\|e^{U_{\lambda,\al}}\lr e^{v_n(x)-U_{\lambda,\al}} - 1\rr\|_{L^\infty(B_4)}\to0,\quad \text{as } n\to+\infty\ \text{for some } \la>0.$$
Thus we get $\|v_n(x)-U_{\lambda,\al}\|_{L^\infty(B_4)}\to0,\ \text{as } n\to+\infty\ \text{for some } \la>0$, then there exists a constant $B$ independent of $n$ such that $\inf\limits_{\overline{B_1}}v_n\geq B$. We may assume below that $\|u_n-v_n\|_{L^\infty(B_{\frac{1}{\var_n}})}>0$ for all sufficiently large $n$, or else Proposition \ref{pp32} implies that $\lambda=1$ and the proof is finished.

Since $\beta_{\al_n}\to-\infty$, on every compact set $K\Subset\R^N$, so $e^{v_n(x)} - e^{\beta_{\alpha_n}} \to e^{U_{\lambda,\al}} \ \text{in } X \ \text{as } n\to+\infty$ implies
$$
e^{v_n}
=\bigl(e^{v_n}-e^{\beta_n}\bigr)+e^{\beta_n}
\longrightarrow e^{U_{\lambda,\alpha}}
\quad\text{uniformly on }K,
$$
and hence
\begin{equation}\label{e5}
v_n\longrightarrow U_{\lambda,\alpha}
\quad\text{locally uniformly in }\R^N,
\end{equation}
i.e., $v_n\rightarrow U_{\lambda,\alpha}$ in $L^{\infty}_{loc}(\mathbb{R}^{N})$. By local $C^{1,\eta}$-estimates for the $N$-Laplace equation, one has
\begin{equation}\label{e6}
v_n\longrightarrow U_{\lambda,\alpha}
\quad\text{in }C^1_{\mathrm{loc}}(\R^N).
\end{equation}
Consequently, we have
\begin{equation}\label{e7}
u_n-v_n\longrightarrow D_{\lambda}(|x|):=U_\alpha(|x|)-U_{\lambda,\alpha}(|x|) \quad\text{in }C^1_{\mathrm{loc}}(\R^N).
\end{equation}
Therefore, we have
\begin{equation}
D_\lambda(|x|)=-N(1+\alpha)\ln\lambda+N\ln\left[\frac{1+\lambda^{\frac{N(1+\alpha)}{N-1}} |x|^{\frac{N(1+\alpha)}{N-1}}}{1+|x|^{\frac{N(1+\alpha)}{N-1}}}\right],
\end{equation}
and hence
\begin{equation}
\frac{\partial}{\partial r}D_\lambda=\frac{N^2(1+\alpha)}{N-1}\left(\lambda^{\frac{N(1+\alpha)}{N-1}}-1\right)\frac{|x|^{\frac{N(1+\alpha)}{N-1}-1}}{(1+|x|^{\frac{N(1+\alpha)}{N-1}})(1+\lambda^{\frac{N(1+\alpha)}{N-1}} |x|^{\frac{N(1+\alpha)}{N-1}})},
\end{equation}
\begin{equation}
\lim_{|x|\to0}D_\lambda(|x|)=-N(1+\alpha)\ln\lambda,
\qquad
\lim_{|x|\to\infty}D_\lambda(|x|)=\frac{N(1+\alpha)}{N-1}\ln\lambda,
\end{equation}
and $D_{\lambda}$ change signs and $D_{\lambda}=0$ on $\partial B_{r(\lambda)}$ for some radius $r(\lambda)>0$. It follows that, for any sufficiently large radius $R$, one has
\begin{align}
\lambda>1&\quad\Longrightarrow\quad
u_n-v_n>0\quad\text{on }\partial B_R,
\label{e8}\\
\lambda<1&\quad\Longrightarrow\quad
u_n-v_n<0\quad\text{on }\partial B_R
\label{e9}
\end{align}
for every sufficiently large $n$.

Now, let $\nu:=\frac{x}{|x|}$ denotes the outward unit normal vector and define
\begin{equation}\label{e10}
z_n:=-\frac{1}{\var_n}\partial_\nu u_n\left(\frac{\theta}{\var_n}\right),
\qquad
\rho_n(\theta):=-\frac{1}{\var_n}\partial_\nu v_n\left(\frac{\theta}{\var_n}\right),
\quad \forall \,\, \theta\in\partial B_1.
\end{equation}
The explicit expression for $u_n$ yields
\begin{equation}\label{e11}
z_n
=\frac{N^2(1+\alpha_n)}{N-1}\frac{\left(\frac{1}{\var_n}\right)^{\frac{N(1+\alpha_n)}{N-1}}}{1+\left(\frac{1}{\var_n}\right)^{\frac{N(1+\alpha_n)}{N-1}}}
\longrightarrow \frac{N^2(1+\alpha)}{N-1},  \qquad \text{as} \,\, n\rightarrow+\infty.
\end{equation}
From \eqref{w421} and \eqref{g4} in Proposition \ref{pro26}, we deduce that, there exists $C>0$ independent of $n$ and $n_0\geq 1$ such that, for any $n\geq n_0$,
\begin{align}\label{e1}
\|z_n-\rho_n(\theta)\|_{L^\infty(\partial B_{1})}&\leq \frac{1}{\var_n}\|\nabla(u_n-v_n)\|_{L^\infty\left(\partial B_{\frac{1}{\var_n}}\right)} \\
&\leq \frac{1}{\var_n}\|u_n-v_n\|_{L^{\infty}(B_{\frac{1}{\var_n}})}\|\nabla w_n\|_{L^\infty\left(\overline{B_{\frac{1}{\var_n}}}\setminus B_{\frac{3}{4\var_n}}\right)} \nonumber \\
&\leq \frac{\|u_n-v_n\|_{L^{\infty}(B_{\frac{1}{\var_n}})}}{\var_n}\frac{C\var_n}{\sqrt{|\ln\var_n|}} \nonumber\\
&\leq\frac{C}{\sqrt{|\ln\var_n|}}\longrightarrow0, \qquad \text{as} \,\, n\rightarrow+\infty. \nonumber
\end{align}

\medskip

Suppose first that $\lambda>1$. We take $R$ large enough such that \eqref{e8} holds, i.e., $u_n-v_n>0$ on $\partial B_R$.  Since
$u_n-v_n=0$ on $\partial B_{\frac{1}{\var_n}}$ and $u_n-v_n>0$ on $\partial B_R$, its negative part $(u_n-v_n)_-:=\max\{-(u_n-v_n),0\}$ belongs to $H_0^1(B_{\frac{1}{\var_n}}\setminus\overline{B_R})$. Testing \eqref{e2} with $(u_n-v_n)_-$ gives
\begin{equation}\label{e12}
\sum\limits_{i,j=1}^{N}\int_{B_{\frac{1}{\var_n}}\setminus B_R}a^n_{ij}\partial_i (u_n-v_n)_-\partial_j (u_n-v_n)_-\dd x
=\int_{B_{\frac{1}{\var_n}}\setminus B_R}|x|^{N\al_n}\zeta_n(x)[(u_n-v_n)_-]^2\dd x.
\end{equation}
Since $\al_n\rightarrow\al$, one has $|x|^{N\al_n}\zeta_n(x)\leq C|x|^{N\al_n}e^{U_{\al_n}}\leq C|x|^{-N-\frac{N(\al+2)}{2(N-1)}}\leq CR^{-\frac{N(\al+1)}{2(N-1)}} |x|^{-N}[\ln|x|]^{-2}$ for $|x|\geq R$ and $n\geq n_0$. By using \eqref{n6}, and the truncated weighted Hardy-Sobolev inequality \eqref{aa436+} on annulus in Proposition \ref{tws+}, we get, for $n\geq n_0$,
\begin{align}\label{e13}
&\quad C_{2}\int_{B_{\frac{1}{\var_n}}\setminus B_R} |x|^{-(N-2)}|\nabla (u_n-v_n)_-|^2\dd x \\
&\leq\int_{B_{\frac{1}{\var_n}}\setminus B_R} |x|^{N\al_n}\zeta_n(x)[(u_n-v_n)_-]^2\dd x \nonumber \\
&\leq CR^{-\frac{N(\al+1)}{2(N-1)}}\int_{B_{\frac{1}{\var_n}}\setminus B_R} |x|^{-N}[\ln|x|]^{-2}[(u_n-v_n)_-]^2\dd x\nonumber\\
&\leq CR^{-\frac{N(\al+1)}{2(N-1)}}
\int_{B_{\frac{1}{\var_n}}\setminus B_R} |x|^{-(N-2)}|\nabla (u_n-v_n)_-|^2\dd x.\nonumber
\end{align}
Since Proposition \ref{ppp43} implies that, for $n\geq n_{0}$,
$$\int_{B_{\frac{1}{\var_n}}\setminus B_R} |x|^{-(N-2)}|\nabla (u_n-v_n)_-|^2\dd x\leq \frac{C}{\ln R}\|u_n-v_n\|^{2}_{L^\infty}\leq \frac{C}{\ln R},$$
by fixing $R$ larger if necessary such that $CR^{-\frac{N(\al+1)}{2(N-1)}}\leq\frac{C_{2}}{2}$ in \eqref{e13}, we get $(u_n-v_n)_-=0$. Thus $u_n-v_n\geq0$ in the annulus $B_{\frac{1}{\var_n}}\setminus B_{R}$. The strong maximum principle and the Hopf boundary point lemma imply that
\begin{equation}
u_n-v_n>0 \quad\text{in }B_{\frac{1}{\var_n}}\setminus\overline{B_R},
\qquad
\partial_\nu(u_n-v_n)<0\quad\text{on }\partial B_{\frac{1}{\var_n}},
\end{equation}
and hence
$$z_n-\rho_n(\theta)=-\frac{1}{\var_n}\partial_\nu(u_n-v_n)\left(\frac{\theta}{\var_n}\right)>0, \qquad \forall \,\, \theta\in\partial B_1.$$
Therefore, we have proved that
\begin{equation}\label{e14}
\lambda>1\quad\Longrightarrow\quad
\rho_n(\theta)<z_n
\quad\text{for every }\theta\in\partial B_1(0).
\end{equation}
If $\lambda<1$, applying the same argument to $-(u_n-v_n)$ gives us that
\begin{equation}\label{e15}
\lambda<1\quad\Longrightarrow\quad
\rho_n(\theta)>z_n
\quad\text{for every }\theta\in\partial B_1(0).
\end{equation}

Now we apply the Pohozaev identity in Lemma \ref{lem-p} to both $v_n$ and $u_n$ with $\var=\var_n$ and $\al=\al_n$ and obtain that
\begin{equation}\label{e16}
\frac{N-1}{N}\int_{\partial B_1}\rho_n^N(\theta)\dd\theta
=N(1+\al_n)\int_{\partial B_1}\rho_n^{N-1}(\theta)\dd\theta
-|\partial B_1|\left(\frac{1}{\var_n}\right)^{N(1+\alpha_n)}e^{\beta_{\al_n}},
\end{equation}
\begin{equation}\label{e17}
\frac{N-1}{N}|\partial B_1|z_n^N
=N(1+\al_n)|\partial B_1|z_n^{N-1}
-|\partial B_1|\left(\frac{1}{\var_n}\right)^{N(1+\alpha_n)}e^{\beta_{\al_n}}.
\end{equation}
Define
\begin{equation}
F_n(t):=\frac{N-1}{N}t^N-N(1+\al_n)t^{N-1},
\qquad t>0.
\end{equation}
Subtracting \eqref{e17} from \eqref{e16} gives the following identity
\begin{equation}\label{e18}
\int_{\partial B_1}
\bigl(F_n(\rho_n(\theta))-F_n(z_n)\bigr)\dd\theta=0.
\end{equation}
Moreover, one has
\begin{equation}\label{e19}
F_n'(t)=(N-1)t^{N-2}(t-N(1+\al_n)).
\end{equation}
From \eqref{e11} and \eqref{e1}, we deduce that, there exists $\delta_0>0$ such that,
for every sufficiently large $n\geq n_0$,
\begin{equation}\label{e20}
z_n\geq N(1+\al_n)+2\delta_0,
\qquad
\rho_n(\theta)\geq N(1+\al_n)+\delta_0
\quad\text{for all }\theta\in\partial B_1.
\end{equation}
Thus $F_n$ is strictly increasing on an interval containing $z_n$ and the entire range of $\rho_n$.

If $\lambda>1$, then \eqref{e14} and strict monotonicity imply
$$
F_n(\rho_n(\theta))-F_n(z_n)<0
\quad\text{for every }\theta\in\partial B_1,
$$
which contradicts \eqref{e18}.  If $\lambda<1$, then \eqref{e15} gives the opposite strict inequality, again contradicting \eqref{e18}. Both alternatives are impossible, hence we must have
$$
\lambda=1.
$$
This concludes our proof of Proposition \ref{pp42}.
\end{proof}

Finally, by directly analyzing the problems satisfied by $h_n$ and $\psi_n$, using the uniform fast decay estimate \eqref{412} for $h_n$, classification result for $\ker \mathcal{L}_{\mathcal{T},\,e^{U_\alpha}}$ in Theorem \ref{c13}, the truncated weighted Hardy-Sobolev inequality \eqref{aa436+} and the Pohozaev identity \eqref{poh}, we successfully overcame the nonlinear nature of the $N$-Laplacian $\Delta_N$, the lack of critical weighted Sobolev embedding inequality, and the lack of the Green function representation formula, and proved the uniform lower bound $\|e^{u_n}-e^{v_n}\|_{X} \geq C_0$ in the following proposition, this result shows the limit of non-radial approximate solutions is non-radial solution to \eqref{11} in $\R^N$, which plays a crucial role in proving the bifurcation result in $\R^N$ in Section 5.
\begin{prop}\label{pp44}
	Let $\al_n$ and $\var_n$ be sequences such that $\al_n\to\al>0$ and $\var_n\to0$ as $n\to+\infty$.
	Let $v_n$ be a sequence of nonradial solutions of \eqref{41} in $B_{\frac{1}{\var_n}}$ related to the exponent $\al_n$. If $\al\ne\alpha(k)$ for all $k\in\N$ and if $\|e^{v_n}-e^{\beta_{\al_n}}\|_\gamma\leq A$ and $\inf\limits_{\overline{B_1}}v_n\geq B$ for some positive constant $A$ and constant $B$ independent of $n$, then there exists a constant $C>0$ independent of $n$ and $n_0\geq1$ such that, for every $n\geq n_0$,
	\begin{align}\label{413}
		\|e^{u_n}-e^{v_n}\|_{X}\geq C,
	\end{align}
	where $u_n$ is the radial solution of \eqref{41} related to the exponent $\al_n$ given by \eqref{32}.
\end{prop}
\begin{proof}
We suppose on the contrary that there exists a sequence nonradial solutions $v_n$ of \eqref{41} in $B_{\frac{1}{\var_n}}$ that is related to the exponent $\al_n$ such that
	\begin{align}\label{p414}
		\|e^{u_n}-e^{v_n}\|_{X}\to0, \quad\text{as}\ n\to+\infty,
	\end{align}
from the definition of $X$, we have
\begin{align}\label{equ446}
		\|e^{u_n}-e^{v_n}\|_{L^{\infty}_\gamma(\R^N)}\to0, \quad\text{as}\ n\to+\infty.
\end{align}
Thus we obtain
\begin{align}\label{equ437}
		u_n-v_n\to 0 \quad \text{in} \,\,\, L^{\infty}_{loc}(\R^N), \quad\text{as}\ n\to+\infty,
\end{align}
furthermore, by local $C^{1,\eta}$-estimates for the $N$-Laplace equation, one has
\begin{align}\label{equ437'}
		u_n-v_n\to 0 \quad \text{in} \,\,\, C^{1}_{loc}(\R^N), \quad\text{as}\ n\to+\infty.
\end{align}
Note that $h_n=\frac{f_n-g_n}{\|f_n-g_n\|_{L^{\infty}(\R^N)}}$, $h_n$ satisfies the following equation in the weak sense	
\begin{align}\label{e22}
		&\quad\int_{B_{\frac{1}{\var_n}}}\int_0^1\Big(|(1-t)\nabla f_n+t\nabla g_n|^{N-2}\nabla h_n+(N-2)\times\\
		&\lr|(1-t)\nabla f_n+t\nabla g_n|^{N-4}\left[((1-t)\nabla f_n+t\nabla g_n)\cdot\nabla h_n\right]((1-t)\nabla f_n+t\nabla g_n)\rr\Big)\cdot\nabla\varphi\md t\md x\nonumber\\
		&+N(N-1)\int_{B_{\frac{1}{\var_n}}}\int_0^1
\frac{|(1-t)\nabla f_n+t\nabla g_n|^{N-2}\left[((1-t)\nabla f_n+t\nabla g_n)\cdot\nabla h_n\right]}{(1-t)e^{u_n}+te^{v_n}}
 \varphi \md t\md x\nonumber\\
		&=N\int_{B_{\frac{1}{\var_n}}}
|x|^{N\alpha_n}\int_0^1((1-t)e^{u_n}+te^{v_n})^{N-1}\md t \, h_n\varphi\md x\nonumber\\
&\quad+(N-1)\int_{B_{\frac{1}{\var_n}}}\int_0^1
\frac{|(1-t)\nabla f_n+t\nabla g_n|^N}{((1-t)e^{u_n}+te^{v_n})^2}
\md t \,
h_n\varphi\md x,\quad\forall\varphi\in C_0^\infty(B_{\frac{1}{\var_n}})\nonumber.
	\end{align}

Since $\|h_n\|_{L^{\infty}(\R^N)}=1$, through the standard regularity theorem, we can deduce from \eqref{36} and \eqref{p414} that $h_n\to h$ in $C_{loc}^{1,\eta}(\R^N)$ and $h$ satisfies the equation
	\begin{align}\label{425p}
		&\quad\int_{\R^N}\left(|\nabla e^{U_\al}|^{N-2}\nabla h+(N-2)
		|\nabla e^{U_\al}|^{N-4}\left(\nabla e^{U_\al}\cdot\nabla h\right)\nabla e^{U_\al}\right)\cdot\nabla\varphi\md x\\
		&+N(N-1)\int_{\R^N}\frac{|\nabla e^{U_\al}|^{N-2}(\nabla e^{U_\al}\cdot\nabla h)}{e^{U_\al}}\varphi\md x\nonumber\\
		&=\int_{\R^N}\lr N|x|^{N\alpha}e^{(N-1)U_\al}h+(N-1) \frac{|\nabla e^{U_\al}|^N h}{e^{2U_\al}} \rr\varphi\md x,\quad\forall\varphi\in C_0^\infty(\R^N).\nonumber
	\end{align}
From Proposition \ref{ppq43}, we deduce that $h\in C^{1}_{loc}(\R^N)$, and for any $R_0>8$ large enough,
\begin{equation}\label{a1}
  |h(x)|\leq C\left(\ln R_{0} + (\ln R_{0}) R_0^{-\frac{N(\al+1)}{2(N-1)}}\ln |x|\right)e^{U_{\al}(x)}, \qquad \forall \,\, |x|>2R_0,
\end{equation}
and hence
$$0\leq \liminf\limits_{R\rightarrow+\infty}\sup\limits_{|x|=R} \frac{e^{-U_\al}|h(x)|}{\ln |x|}\leq (\ln R_{0}) R_0^{-\frac{N(\al+1)}{2(N-1)}}\rightarrow0, \qquad \text{as} \,\, R_{0}\rightarrow+\infty,$$
namely, $\liminf\limits_{R\rightarrow+\infty}\sup\limits_{|x|=R}\frac{e^{-U_\al}|h(x)|}{\ln |x|}=0$. Because $\al\ne\alpha(k)$, by Theorem \ref{c13}, we get
	\begin{align}\label{425}
		h(x)=A\frac{(N-1)-|x|^{\frac{N( \al+1)}{N-1}}}
		{\lr1+|x|^{\frac{N(1+\al)}{N-1}}\rr^{N+1}} \qquad\text{for\ some}\ A\in\R.
	\end{align}

If we assume that
$$A=0.$$
Then, we have
	\begin{align}\label{p440}h_n\to h\equiv 0\quad \text{in}\ C^{1,\eta}_{loc}(\R^N). \end{align}
	Let $x_n\in B_{\frac{1}{\var_n}}$ be such that $|h_n(x_n)|=1=\|h_n\|_{L^{\infty}(\mathbb{R}^{N})}$. From the uniform decay estimate in Proposition \ref{ppq43}, we know the sequence $\{x_n\}$ is bounded, so there exists some point $x_0\in\R^N$ such that $x_n\to x_0$ (up to subsequence) and thus $|h(x_0)|=1$. This contradicts \eqref{p440}.

\medskip

Thus it follows that $A \neq 0$. We will also derive contradiction in similar way as $\lambda\neq 1$ in the proof of Proposition \ref{pp42}. It follows from $h_{n}\rightarrow A\frac{(N-1)-|x|^{\frac{N( \al+1)}{N-1}}}
		{\lr1+|x|^{\frac{N(1+\al)}{N-1}}\rr^{N+1}}$ that, for any sufficiently large radius $R$ (to be fixed later), one has
\begin{align}
Ah_n<0\quad\text{on }\partial B_R,
\label{e8'}
\end{align}
for every sufficiently large $n$. Recall that $h_n=\zeta_n\psi_n$ and $\frac{1}{C}e^{U_{\alpha_n}}\leq \zeta_n\leq Ce^{U_{\alpha_n}}$, we have
\begin{align}
A\psi_n<0\quad\text{on }\partial B_R
\label{e8+}
\end{align}
for every sufficiently large $n$.

Suppose first that $A>0$. We take $R$ large enough such that \eqref{e8+} holds, i.e., $\psi_n:=\frac{u_n-v_n}{\|f_n-g_n\|_{L^{\infty}(\R^N)}}<0$ on $\partial B_R$. Since $\psi_n=0$ on $\partial B_{\frac{1}{\var_n}}$ and $\psi_n<0$ on $\partial B_R$, its positive part $(\psi_n)_+:=\max\{\psi_n,0\}$ belongs to $H_0^1(B_{\frac{1}{\var_n}}\setminus\overline{B_R})$. Testing equation \eqref{w} with $(\psi_n)_+$ gives
\begin{equation}\label{e12+}
\sum\limits_{i,j=1}^{N}\int_{B_{\frac{1}{\var_n}}\setminus B_R}a^n_{ij}\partial_i (\psi_n)_+\partial_j (\psi_n)_+\dd x
=\int_{B_{\frac{1}{\var_n}}\setminus B_R}|x|^{N\al_n}\zeta_n(x)[(\psi_n)_+]^2\dd x.
\end{equation}
Since $\al_n\rightarrow\al$, one has $|x|^{N\al_n}\zeta_n(x)\leq C|x|^{N\al_n}e^{U_{\al_n}}\leq C|x|^{-N-\frac{N(\al+2)}{2(N-1)}}\leq CR^{-\frac{N(\al+1)}{2(N-1)}} |x|^{-N}[\ln|x|]^{-2}$ for $|x|\geq R$ and $n\geq n_0$. By using \eqref{n6}, and the truncated weighted Hardy-Sobolev inequality \eqref{aa436+} on annulus in Proposition \ref{tws+}, we get, for $n\geq n_0$,
\begin{align}\label{e13+}
&\quad C_{2}\int_{B_{\frac{1}{\var_n}}\setminus B_R} |x|^{-(N-2)}|\nabla (\psi_n)_+|^2\dd x \\
&\leq\int_{B_{\frac{1}{\var_n}}\setminus B_R} |x|^{N\al_n}\zeta_n(x)[(\psi_n)_+]^2\dd x  \nonumber \\
&\leq CR^{-\frac{N(\al+1)}{2(N-1)}}\int_{B_{\frac{1}{\var_n}}\setminus B_R} |x|^{-N}[\ln|x|]^{-2} [(\psi_n)_+]^2\dd x\nonumber\\
&\leq CR^{-\frac{N(\al+1)}{2(N-1)}}
\int_{B_{\frac{1}{\var_n}}\setminus B_R} |x|^{-(N-2)}|\nabla (\psi_n)_+|^2\dd x.\nonumber
\end{align}
Since Proposition \ref{ppp43} implies that, for $n\geq n_{0}$,
$$\int_{B_{\frac{1}{\var_n}}\setminus B_R} |x|^{-(N-2)}|\nabla (\psi_n)_+|^2\dd x\leq \frac{C}{\ln R}\frac{\|u_n-v_n\|^{2}_{L^\infty}}{\|f_n-g_n\|^{2}_{L^\infty}}<+\infty,$$
by fixing $R$ larger if necessary such that $CR^{-\frac{N(\al+1)}{2(N-1)}}\leq\frac{C_{2}}{2}$ in \eqref{e13+}, we get $(\psi_n)_+=0$. Thus $\psi_n\leq0$ in the annulus $B_{\frac{1}{\var_n}}\setminus B_{R}$. The strong maximum principle and the Hopf boundary point lemma imply that
\begin{equation}
\psi_n<0 \quad\text{in }B_{\frac{1}{\var_n}}\setminus\overline{B_R},
\qquad
\partial_\nu\psi_n>0\quad\text{on }\partial B_{\frac{1}{\var_n}},
\end{equation}
and hence
$$z_n-\rho_n(\theta)=-\frac{1}{\var_n}\partial_\nu\psi_n\left(\frac{\theta}{\var_n}\right)\|f_n-g_n\|_{L^\infty}<0, \qquad \forall \,\, \theta\in\partial B_1.$$
Therefore, we have proved that
\begin{equation}\label{e14+}
A>0\quad\Longrightarrow\quad
\rho_n(\theta)>z_n
\quad\text{for every }\theta\in\partial B_1(0).
\end{equation}
If $A<0$, by applying the same argument to $-\psi_n$, we derive that
\begin{equation}\label{e15+}
A<0\quad\Longrightarrow\quad
\rho_n(\theta)<z_n
\quad\text{for every }\theta\in\partial B_1(0).
\end{equation}

Note that, by \eqref{e19}, $F_n(t)$ is strictly increasing for $t\in [N(1+\al_n),+\infty)$. Therefore, from \eqref{e20}, we infer that $F_n$ is strictly increasing on an interval containing $z_n$ and the entire range of $\rho_n$, which combining with \eqref{e14+} and \eqref{e15+} contradicts the integral identity \eqref{e18}, i.e.,
\begin{equation*}
\int_{\partial B_1}\bigl(F_n(\rho_n(\theta))-F_n(z_n)\bigr)\dd\theta=0.
\end{equation*}
Thus we have obtained contradictions for both $A>0$ and $A<0$.

\medskip

As a consequence, we finally derive the uniform lower bound estimate \eqref{413}. This finishes our proof of Proposition \ref{pp44}.
\end{proof}

\begin{rem}\label{bl}
The assumption $\inf\limits_{\overline{B_1}}v_n\geq B$ for some constant $B$ independent of $n$ in Proposition \ref{pp44} can be deduced from $e^{v_n(x)}-e^{\beta_{\al_n}}\to L(x)$ in $X$ for some positive function $0<L(x)\in C_{loc}(\R^N)$. In fact, $\|e^{v_n(x)}-e^{\beta_{\al_n}}-L(x)\|_{X}\rightarrow0$ implies that $\|e^{v_n(x)}-L(x)\|_{L^{\infty}(B_1)}\rightarrow0$, and hence $\inf\limits_{\overline{B_1}}v_n\geq \ln\Big[\frac{1}{2}\min\limits_{\overline{B_1}} L(x)\Big]=:B$ for $n\geq n_0$ large enough.
\end{rem}

\section{The bifurcation result: completion of our proof of the main Theorem \ref{th14}}

In this section, we will finish the proof of Theorem \ref{th14}.
For this purpose, let $\{\varepsilon_n\}$ be a sequence such that $\varepsilon_n\rightarrow0$. In Section 3, for all $k\in\mathbb{N}$, we have shown that $(\alpha_k^n,u_{n,\alpha_k^n})$ are nonradial bifurcation points, generating continua $ \mathcal{C}(\alpha_k^n)$ in the space $(0,+\infty)\times \mathcal{Q}_n$ where $\alpha_k^n$ denotes the unique root of equation \eqref{330}. Furthermore, when $k$ is an even integer, these continua also exist in $(0,+\infty)\times \mathcal{Q}^l_n$, with $ \mathcal{Q}^l_n$ defined in the proof of Theorem \ref{thm39}. These continua $ \mathcal{C}(\alpha_k^n)$ are global and satisfy the well-known Rabinowitz alternative Theorem (see Theorem \ref{th38}). However, since $u_{n,\alpha_k^n}$ does not belong to the space $X$, it is actually $e^{u_{n,\alpha_k^n}}-e^{\beta_{\al_n}}\in X$. Therefore, we apply the global bifurcation theory to $e^{u_{n,\alpha_k^n}}-e^{\beta_{\al_n}}$ in $\mathbb{R}^N$. In equation \eqref{31}, by setting  $f=e^{u}-e^{\beta_{\al_n}}$, then $f$ satisfies the following problem
 \begin{equation}\label{54eq}
		\begin{cases}
			-\Delta_N f=|x|^{N\alpha_n}(f+e^{\beta_{\al_n}})^N-\frac{(N-1)|\nabla f|^N}{f+e^{\beta_{\al_n}}}, & x\in B_{\frac1{\var_n}}, \\
			f=0, & x\in\partial B_{\frac1{\var_n}}.
		\end{cases}
	\end{equation}
 Define radial function
 \begin{equation}\label{f52}
	 f_{n,\alpha}=e^{u_{\var_n,\al_n}}-e^{\beta_{\al_n}}=\left\{
	\begin{aligned}
		&\frac{C_{N,\alpha_n}}{\left(1+|x|^{\frac{N(\al_n+1)}{N-1}}\right)^N}-\frac{C_{N,\alpha_n}\var_n^{\frac{N^2(\al_n+1)}{N-1}}}{\left(1+\var_n^{\frac{N(\al_n+1)}{N-1}}\right)^N},&  \text{if}\ |x|\leq\frac{1}{\var_{n}}, \\
		&0, & \text{if}\ |x|>\frac{1}{\var_{n}},
	\end{aligned}
	\right.
\end{equation}
where $u_{\var,\al}$ as defined in \eqref{32} is the radial solution of \eqref{31} and $e^{\beta_\al}=e^{U_\al\lr\frac1\var\rr}=\frac{C_{N,\alpha}\var^{\frac{N^2(  \al+1)}{N-1}}}{\left(1+\var^{\frac{N(  \al+1)}{N-1}}\right)^N}$.

Then, analogous to the proof of Lemma \ref{lem31}, we can show that $ f_{n,\alpha}$ forms a sequence of non-degenerate radial solutions to equation \eqref{54eq} in the space of radial functions. As $n\to+\infty$, the related limiting problem of \eqref{54eq} is
\begin{equation}\label{55eq}
		-\Delta_N g=|x|^{N\alpha}g^N-\frac{(N-1)|\nabla g|^N}{g}, \ \ x\in \R^N.
	\end{equation}
Let $v_n$ be the non-radial solutions of \eqref{31} on the large ball $B_{\frac{1}{\var_n}}$. By setting $f_n=e^{v_{n}}-e^{\beta_{\al_n}}$, within this transformed framework, we then show that the non-radial solution $f_n$ converges in the $\|\cdot\|_X$ norm to the whole space $\R^N$.
Define the set
\begin{align}\label{e52}
	 \hat{\mathcal{A}}(n)=\Big\{&(\al, f_{n,\alpha})\in(0,+\infty)\times C_0^{1,\eta}(\overline B_{\frac{1}{\var_n}}) \ \text{such\  that}\  f_{n,\alpha} \\ &\text{is\ the\ radial\ solution\ of}\ \eqref{54eq}\ \text{defined\ in} \ \eqref{f52}\Big\}\nonumber.
\end{align}
For the given curve $\hat{\mathcal{A}}(n)$, we say that a point $(\alpha_j, f_{n,\al_j}) \in  \hat{\mathcal{A}}(n)$ is a nonradial bifurcation point, if for every neighborhood of $(\alpha_j, f_{n,\al_j})$ in $(0, +\infty) \times C_0^{1,\eta}(\overline{B}_{\frac{1}{\varepsilon_n}})$, there exists a point $(\alpha, e^{v_{n,\alpha}}-e^{\beta_{\al_n}})$ in this neighborhood such that $e^{v_{n,\alpha}}-e^{\beta_{\al_n}}$ is a nonradial solution of \eqref{54eq}.

The subspace $ \hat{\mathcal{Q}}_n \subset C_0^{1,\eta}(\overline{B}_{\frac{1}{\var_n}})$ is defined by
\begin{align}\label{e54}
	 \hat{\mathcal{Q}}_n=\Big\{&h\in C_0^{1,\eta}(\overline B_{\frac{1}{\var_n}}) \ \text{such\  that}\ h(x_1,\cdots,x_N)=h\left( g(x_1,\cdots,x_{N-1}),x_N\right) \\ &\text{for\ any}\ g\in\mathcal{O}(N-1)\Big\}\nonumber,
\end{align}
where $\mathcal{O}(N-1)$ is the orthogonal group in $\R^{N-1}$.
Consider the operator $ \hat{\mathcal{L}}^n(\alpha,h): (0,+\infty) \times  \hat{\mathcal{Q}}_n \to  \hat{\mathcal{Q}}_n$ defined by
$$
 \hat{\mathcal{L}}^n(\alpha,h) := (-\Delta_N)^{-1}\left(|x|^{N\alpha}(h+e^{\beta_\al})^N-\frac{(N-1)|\nabla h|^N}{h+e^{\beta_\al}}\right)
$$
The operator $ \hat{\mathcal{L}}^n$ is compact for each fixed $\alpha \in (0,+\infty)$ and $\alpha \mapsto  \hat{\mathcal{L}}^n(\alpha,\cdot)$ is continuous. Define the operator
$$
\hat T^n(\alpha,h) :=h -  \hat{\mathcal{L}}^n(\alpha,h).
$$

Let $ \hat{\Upsilon}_n$ denote the closure in $[0,+\infty) \times  \hat{\mathcal{Q}}_n$ of all solutions to $\hat T^n(\alpha,h) = 0$ different from $ f_{n,\alpha}$, i.e.,
\begin{equation}\label{e55}
 \hat{\Upsilon}_n := \overline{\Big\{(\alpha,h) \in [0,+\infty) \times  \hat{\mathcal{Q}}_n \ \text{such that } \hat T^n(\alpha,h) = 0\ \text{with}\ h \neq  f_{n,\alpha}\Big\}}.
\end{equation}
If $(\alpha_k^n,  f_{n,\alpha_k^n}) \in \hat{\mathcal{A}}(n)$ is a nonradial bifurcation point, we have
$(\alpha_k^n,  f_{n,\alpha_k^n}) \in \hat{\Upsilon}_n.$
We denote $\widetilde{\mathcal{C}}(\alpha_k^n)\subset\hat{\Upsilon}_n$ the closed connected component of $\hat{\Upsilon}_n$ which
contains $(\al^n_k,f_{n,\al^n_k})$ and it is maximal with respect to the inclusion.


Define the space
\begin{align}\label{q51}
\hat{\mathcal{Q}}=\hat{\mathcal{Q}}_\al:=\Big\{&h\in X_\al \ \text{such\  that}\ h(x_1,\cdots,x_N)=h\lr g(x_1,\cdots,x_{N-1}),x_N\rr \\ &\text{for\ any}\ g\in\mathcal{O}(N-1)\Big\}\nonumber.
\end{align}
If there is no need to emphasize the dependency on $\al$, we will abbreviate $\hat{\mathcal{Q}}_\al$ as $\hat{\mathcal{Q}}$.

By extending the function by zero outside $B_{\frac{1}{\varepsilon_n}}$, we can deduce via regularity theorems that $ \widetilde{\mathcal{C}}(\alpha_k^n)$ lies in the space
$$ \hat{\mathcal{H}}:=\bigcup\limits_{\al\in(0,+\infty)}\bigcup\limits_{g\in\hat{\mathcal{Q}}_\al}\{(\al,g)\}.$$
Let $\hat{\mathcal{C}}(\alpha_k^{n})$ denote the maximal connected component bifurcating from $(\alpha_k^{n},  f_{n,\alpha_k^n})$ in the space $\hat{\mathcal{H}}$.

Furthermore, by Proposition \ref{pp32}, $f_{n,\alpha_k^n}\to e^{U_{\al(k)}}$ in $\hat{\mathcal{Q}}$ (a subset of $X$) as $n \to +\infty$. From \cite{SW1}, restricting the Morse index of $e^{U_\alpha}$ to $\hat{\mathcal{Q}}$ leads to an increase of $1$ as $\alpha$ crosses $\alpha(k)$, i.e.,
\begin{align*}
m(\alpha(k)+\delta)-m(\alpha(k)-\delta)=1,
\end{align*}
where $m$ denotes the Morse index of $e^{U_\alpha}$ in the space $\hat{\mathcal{Q}}$. Our goal is to use this change in the Morse index of $e^{U_{\alpha}}$ within the space $\hat{\mathcal{Q}}$ and prove that, as $n\to+\infty$, these continua $\hat{\mathcal{C}}(\alpha_k^n)$ converge in a suitable sense to continua of non-radial solutions of \eqref{55eq} bifurcating from $(\alpha(k),e^{U_{\alpha(k)}})$ in the product space $\hat{\mathcal{H}}=\bigcup\limits_{\al\in(0,+\infty)}\bigcup\limits_{g\in\hat{\mathcal{Q}}_\al}\{(\al,g)\}$. To do this, we adopt certain ideas previously used in \cite{AG,GP}.

To establish the bifurcation result, we require the following topological result (see Theorem 9.1 in \cite{W}).

\begin{lem}[\textbf{\cite{W}}]\label{th51}
	Let $X_n$ be a sequence of connected subsets of a metric space $X$. Let $\lim\inf(X_n)$ $(\lim\sup(X_n)$, resp.$)$ denote the set of all $x\in X$ such that any neighborhood of $x$ intersects all except finitely many of $X_n$ $($infinitely many of $X_n$, resp.$)$. If\\
	(i) $\lim\inf(X_n)\ne\emptyset$, \\
	(ii) $\bigcup X_n$ is precompact, \\
	then $\lim\sup(X_n)$ is nonempty, compact and connected.
\end{lem}

Let $n$ be sufficiently large and let $\alpha_k^{n}$ be given by \eqref{330}, such that $(\alpha_k^{n},  f_{n,\alpha_k^n})$ is a bifurcation point for problem \eqref{54eq}. Fix $\delta>0$ such that no other exponent $\alpha(j)$ (with $j\neq k$) exists in the interval $[\alpha(k)-\delta,\alpha(k)+\delta]$. Define
$$ \hat{\mathcal{H}}_n:= \hat{\mathcal{C}}(\alpha_k^n)\bigcap B_{\delta, \hat{\mathcal{H}}}(\alpha_k^n,  f_{n,\alpha_k^n}),$$
where
$$B_{\delta, \hat{\mathcal{H}}}(\alpha_k^n,  f_{n,\alpha_k^n}):=\{(\alpha, h)\in  \hat{\mathcal{H}}\ \text{such that}\ |\alpha-\alpha_k^n|+\| h-   f_{n,\alpha_k^n}\|_{X_{\alpha_k^n+\delta}}<\delta\},$$
and the space $X$ along with its norm are defined by \eqref{19} and \eqref{d18}. Let $ \hat{\mathcal{S}}^n_k$ be the maximal connected component of $ \hat{\mathcal{H}}_n$ containing $(\alpha_k^n,  f_{n,\alpha_k^n})$. Then, $ \hat{\mathcal{S}}^n_k\neq\emptyset$ and $(\alpha(k),e^{U_{\alpha(k)}})\in\liminf\limits_{n}( \hat{\mathcal{S}}^n_k)$.

\begin{lem}\label{th52}
For every $k\in\mathbb{N}$, the set $\bigcup\limits_{n} \hat{\mathcal{S}}^n_k$ is precompact in $ \hat{\mathcal{H}}$.
\end{lem}
\begin{proof}
Let $(\alpha^{(k)}_j,h_j^{(k)})$ be a sequence in $\bigcup\limits_{n} \hat{\mathcal{S}}^n_k$. Then, there $(\alpha^{(k)}_j,h_j^{(k)})\in \hat{\mathcal{S}}_k^{n(j)}$ for some $n(j)\geq1$. First, we consider the case: $n(j)\to+\infty$ as $j\to+\infty$.

Based on the definition of $ \hat{\mathcal{S}}_k^{n(j)}$, the functions $h_j^{(k)}$ satisfy equation \eqref{54eq} with $\var = \var_{n(j)}$ and $\alpha = \alpha_j^{(k)}$. Since $|\alpha_j^{(k)}-\alpha_{k}^{n(j)}|+|\alpha_{k}^{n(j)}-\alpha(k)|\leq2\delta$ for $n(j)$ sufficiently large, we can deduce that $\alpha_j^{(k)}\in(\alpha(k) - 2\delta,\alpha(k) + 2\delta)$. Consequently, if required, by passing to a subsequence for $j$, we can conclude that $\alpha_j^{(k)}\to\hat{\alpha}$, where $\hat{\alpha}\in[\alpha(k) - 2\delta,\alpha(k) + 2\delta]$. Additionally, since $h_j^{(k)}\in \hat{\mathcal{Q}}$ and $\|h_j^{(k)}-f_{n(j),\alpha_k^{n(j)}}\|_X<\delta$, it implies that $\|h_j^{(k)}\|_X<\delta+\sup\limits_j \|f_{n(j),\alpha_k^{n(j)}}\|_X$. From Proposition \ref{pp32}, we can infer that $\|f_{n(j),\alpha_k^{n(j)}}\|_X\leq C$ with $C$ independent of $j$. As a consequence, there exists a positive constant $C > 0$ such that $\|h_j^{(k)}\|_\gamma\leq C$ and $\|h_j^{(k)}\|_{1,N}\leq C$ for all $j$. Then, by taking a subsequence, we can claim that $h_j^{(k)}\to  \hat{h}$ weakly in $  W^{1,N}(\mathbb{R}^N)$ and almost everywhere in $\mathbb{R}^N$. Furthermore, by applying equation \eqref{54eq}, we can show that $h_j^{(k)}\to  \hat{h}$ in $C^{1}_{loc}(\mathbb{R}^N)$, where $ \hat{h}$ is a solution of equation \eqref{55eq} with the exponent $\alpha=\hat{\alpha}$.

Moreover, based on Proposition \ref{pp41}, there exists a positive constant $C>0$ such that
\begin{align}\label{52}
h_j^{(k)}(x)\leq\frac{C}{(1 + |x|)^{\gamma_{n(j)}}}
\ \ \text{for every } x\in\mathbb{R}^N \text{ and for every } j\in\mathbb{N}.
\end{align}
Then, we get
\begin{align}\label{53}
|h_j^{(k)}(x)-\hat{h}(x)| \leq h_j^{(k)}(x)+\hat{h}(x)\leq\frac{C}{(1 + |x|)^{\gamma_{n(j)}}}+\frac{C}{(1 + |x|)^{\frac{N^2(\al+1)}{N-1}}},
\end{align}
which means that for every $\varepsilon>0$, there exists $r>0$ such that for any $j>0$ large enough, $$(1 + |x|)^\gamma |h_j^{(k)}-  \hat{h} | <\varepsilon, \ \text{if}\ |x|>r.$$

Due to the uniform convergence of $h_j^{(k)}$ to $ \hat{h}$ on compact subsets of $\mathbb{R}^N$, we can conclude that $(1 + |x|)^\gamma |h_j^{(k)}-  \hat{h} | <\varepsilon$ in $B_r(0)$ if $j$ is large enough. That is, $h_j^{(k)}\to  \hat{h}$ in $L^\infty_{\gamma}(\mathbb{R}^N)$.

Furthermore, from \eqref{52}, we obtain
\begin{align}\label{54}
	N\int_{B_{\frac{1}{\var_{n(j)}}}}|\nabla h_j^{(k)}|^N\mathrm{d}x=&\int_{B_{\frac{1}{\var_{n(j)}}}}|x|^{N\alpha^{(k)}_j}\left( h_j^{(k)}+e^{\beta_{\al_j^{(k)}}}\right)^{N+1}\mathrm{d}x \\ &-e^{N\beta_{\al_j^{(k)}}}\int_{B_{\frac{1}{\var_{n(j)}}}}|x|^{N\alpha^{(k)}_j}\left( h_j^{(k)}+e^{\beta_{\al_j^{(k)}}}\right)\mathrm{d}x. \nonumber
\end{align}
Taking the limit for $j$ in \eqref{54} based on \eqref{42}, we get
\begin{align}\label{55}
	N\int_{B_{\frac{1}{\var_{n(j)}}}}|\nabla h_j^{(k)}|^N\mathrm{d}x\to\int_{\mathbb{R}^N}|x|^{N\hat{\alpha}}\hat{h}^{N+1}\mathrm{d}x=N\int_{\mathbb{R}^N}|\nabla  \hat{h}|^N\mathrm{d}x.
\end{align}
Thus
$\int_{\mathbb{R}^N}|\nabla (h_j^{(k)}- \hat{h})|^N\mathrm{d}x\to0$. Then $h_j^{(k)}\to  \hat{h}$ strongly in $X$.

Next, we consider the other case that $n(j)\not\to+\infty$, as $j\to+\infty$. Then, by passing to a subsequence, we may suppose that $n(j)$ converges to $n_0\in\mathbb{N}$. Repeating the proof for the case $n(j)\to+\infty$, we find that a subsequence of $h_j^{(k)}$ converges in $X$ to a solution of \eqref{54eq} for $\var=\var_{n_0}$ and the exponent $\alpha=\hat{\alpha}$. This finishes our proof of Lemma \ref{th52}.
\end{proof}

\begin{lem}\label{th53}
For every $k\in\mathbb{N}$, the set $\limsup\limits_n(\hat{\mathcal{S}}^n_k)\backslash\{(\alpha(k),e^{U_{\alpha(k)}})\}$ is nonempty.
\end{lem}
\begin{proof}
By analogy with the proof of Theorem \ref{th38} and the regularity theory, for the global continuum $\hat{\mathcal{C}}(\alpha_k^n)$ which bifurcates from the point $(\alpha_k^n, f_{n,\alpha_k^n})$, we may conclude the following: either the continuum $\hat{\mathcal{C}}(\alpha_k^n)$ is unbounded in the space $\hat{\mathcal{H}}$, or there exists another bifurcation point $(\alpha_i^n, f_{n,\alpha_i^n})$ with $\alpha_i^n\not\in[\alpha(k) - \delta, \alpha(k) + \delta]$, or $\hat{\mathcal{C}}(\alpha_k^n)$ meets $\{0\} \times X$. As a result, on the closure of any component $ \hat{\mathcal{S}}^n_k$, there exists a point $(\hat{\alpha}_{k}^n, \hat{h}_{n,\hat{\alpha}_{k}^n}) \in \partial B_{\delta, \hat{\mathcal{H}}}(\alpha_k^n,   f_{n,\alpha_k^n})$, i.e.,
\begin{align}\label{56}
	|\hat{\alpha}_{k}^n - \alpha_k^n| + \|\hat{h}_{n,\hat{\alpha}_{k}^n} - f_{n,\alpha_k^{n}}\|_X = \delta,
\end{align}
where $\hat{h}_{n,\hat{\alpha}_{k}^n}$ is a solution of \eqref{54eq} in $B_{\frac{1}{\var_n}}$ with the exponent $\alpha = \hat{\alpha}_{k}^n$.

By the standard regularity theorems and the boundedness of $\hat{\alpha}_{k}^n$ and $\hat{h}_{n,\hat{\alpha}_{k}^n}$, by passing to the limit $n\to+\infty$, we find that
$(\hat{\alpha}_{k}^n,\hat{h}_{n,\hat{\alpha}_{k}^n})\to(\hat{\alpha}_k,\hat{h}_k)$, where $\hat{h}_k$ is a solution of \eqref{55eq} for the exponent $\alpha = \hat{\alpha}_k$, with $\hat{\alpha}_k\in[\alpha(k) - 2\delta,\alpha(k) + 2\delta]$, and $(\hat{\alpha}_k,\hat{h}_k)$ satisfies
\begin{align*}
	|\hat{\alpha}_k - \alpha(k)|+\|\hat{h}_k -e^{U_{\alpha(k)}}\|_X=\delta.
\end{align*}
Note that $(\hat{\alpha}_k,\hat{h}_k)\in\limsup\limits_n( \hat{\mathcal{S}}_k^n)$, but $(\hat{\alpha}_k,\hat{h}_k)\neq(\alpha(k),e^{U_{\alpha(k)}})$. This completes our proof of Lemma \ref{th53}.
\end{proof}

Define the curve
\begin{align}\label{57}
	\Gamma=\Big\{&(\al,e^{U_\al})\in \bigcup\limits_{\al\in(0,+\infty)}\bigcup\limits_{g\in X_\al}\{(\al,g)\} \ \text{such\  that}\ e^{U_\al}\ \text{is\ the} \\ &\text{isolated\ radial\ solution\ of}\ \eqref{55eq}\ such\ that\ e^{U_\al(0)}=C_{N,\al}\Big\}\nonumber.
\end{align}

\begin{thm}\label{th54}
	For any $k\geq2$, the points $(\alpha(k),e^{U_{\alpha(k)}})$ are nonradial bifurcation points for the curve
	$\Gamma$.
\end{thm}
\begin{proof}
For $k\geq2$, we study the bifurcation points $(\alpha_k^n,  f_{n,\alpha_k^n})$ of problem \eqref{54eq} in $B_{\frac{1}{\var_n}}$, along with the connected components $ \hat{\mathcal{S}}^n_k$ of the bifurcation continua in $B_{\delta, \hat{\mathcal{H}}}(\alpha_k^n,  f_{n,\alpha_k^n})$.

According to Lemma \ref{th52}, the sequence of sets $ \hat{\mathcal{S}}_k^n$ satisfies the hypotheses of Lemma \ref{th51} in the space $ \hat{\mathcal{H}}$. Consequently,
$$
 \hat{\mathcal{C}}_k=\limsup\limits_n( \hat{\mathcal{S}}^n_k)
$$
is nonempty, compact, and connected. Moreover, $ \hat{\mathcal{C}}_k$ contains the point $(\alpha(k),e^{U_{\alpha(k)}})$, and by Lemma \ref{th53}, it implies that $\hat{\mathcal{C}_k}\setminus\{(\al(k),e^{U_{\al(k)}})\}\neq\emptyset$.
	
If $(\hat{\alpha}_k,\hat{h}_k)\in \hat{\mathcal{C}}_k\setminus\{(\alpha(k),e^{U_{\alpha(k)}})\}$, then there exists a sequence of points $(\hat{\alpha}_k^n,\hat{h}_k^n)\in \hat{\mathcal{S}}^n_k$ such that $(\hat{\alpha}_k^n,\hat{h}_k^n)\to(\hat{\alpha}_k,\hat{h}_k)$ in $ \hat{\mathcal{H}}$, where $\hat{h}_k^n:=e^{v_{n,\hat{\alpha}_k^n}}-e^{\beta_{\hat{\alpha}_k^n}}$,  $v_{n,\hat{\alpha}_k^n}$ is a solution of \eqref{31} for the exponent $\alpha = \hat{\alpha}_k^n$ and $\hat{h}_k$ is a solution of \eqref{55eq} for the exponent $\alpha = \hat{\alpha}_k$. Our goal is to prove that $\hat{h}_k\neq e^{U_{\lambda,\hat{\alpha}_{k}}}$ for any $\lambda>0$.
	
Based on Proposition \ref{pp42}, we conclude that $\hat{h}_k\neq e^{U_{\lambda,\hat{\alpha}_k}}$ for any $\lambda\neq1$. In addition, if $\hat{\alpha}_k = \alpha(k)$, then $\hat{h}_k\ne e^{U_{\hat{\alpha}_k}}$, since $(\hat{\alpha}_k,\hat{h}_k)\in \hat{\mathcal{C}}_k\setminus(\alpha(k),e^{U_{\alpha(k)}})$.

It suffices to consider the case $\lambda = 1$ with $\hat{\alpha}_k \neq \alpha(k)$, and prove $\hat{h}_k \neq e^{U_{\hat{\alpha}_k}}$ by contradiction argument. Suppose on the contrary that $\hat{h}_k=e^{U_{\hat{\alpha}_k}}$, it follows that $e^{v_{n,\hat{\alpha}_k^n}}-e^{\beta_{\hat{\alpha}_k^n}}\rightarrow e^{U_{\hat{\alpha}_k}}$ in $X$, as $n\rightarrow+\infty$. Thus by Proposition \ref{pp44} and Remark \ref{bl}, we have
\begin{align}\label{58}
	\|\hat{h}_k^n-f_{n,\hat\al_k^n}\|_X=\left\|\lr e^{v_{n,\hat{\alpha}_k^n}}-e^{\beta_{\hat{\alpha}_k^n}}\rr-\lr e^{u_{n,\hat{\alpha}_k^n}}-e^{\beta_{\hat{\alpha}_k^n}}\rr\right\|_X>C>0
\end{align}
for all sufficiently large $n$, where $C>0$ is independent of $n$. From Proposition \ref{pp32}, we know that $f_{n,\hat\al_k^n}\to e^{U_{\hat\al_k}}$ in $X$, as $n\rightarrow+\infty$. Combining this with \eqref{58} and taking the limit $n\rightarrow+\infty$ imply that, for some constant $C>0$,
$$
\|\hat{h}_k-e^{ U_{\hat{\alpha}_k}}\|_X>C>0,
$$
which contradicts the assumption $\hat{h}_k=e^{U_{\hat{\alpha}_k}}$.

From the uniqueness result for radial solutions to \eqref{55eq}(similar to prove in Theorem \ref{thm12}), it follows that $\hat{h}_k$ is a nonradial solution to \eqref{55eq} associated with the exponent $\alpha = \hat{\alpha}_k$. This completes the proof of Theorem \ref{th54}.
\end{proof}

\begin{rem}\label{rem55}
We observe that the bifurcation at the points $(\alpha(k), e^{U_{\alpha(k)}})$ in Theorem \ref{th54} is actually global, analogous to Theorem \ref{th38}. In particular, we have shown the existence of a closed and connected set $ \hat{\mathcal{C}}_k$ that bifurcates from each point $(\alpha(k), e^{U_{\alpha(k)}})$.
\end{rem}
\noindent {\bf Proof of Theorem \ref{th14}.}
Theorem \ref{th54} proves the existence of a continuum $ \hat{\mathcal{C}}_k$ of nonradial solutions to equation \eqref{55eq}. These solutions are invariant under the action of $\mathcal{O}(N - 1)$ and bifurcate from the points $(\alpha(k), e^{U_{\alpha(k)}})$, where
$\alpha(k)=\frac{\sqrt{k(N-1)(k+N-2)}}{N-1}-1$
for all $k\geq2$. Transforming the solution back to equation \eqref{11}, we thus obtain (i) in Theorem \ref{th14}.

Moreover, when $k$ is an even number, we can reproduce the proof of Theorem \ref{th54}. From the space $ \hat{\mathcal{Q}}^l$, which is defined by
$$
 \hat{\mathcal{Q}}^l:=\{h\in X\ \text{such that $h$ is invariant by the action of}\  \mathcal{G}_l \}
$$
for $l = 1,2,\cdots,\left[\frac{N}{2}\right]$, where $ \mathcal{G}_l$ is defined by \eqref{347qq}. By also applying Remark \ref{rem312}, we are able to identify $\left[\frac{N}{2}\right]$ distinct continua that bifurcate from the point $(\alpha(k),e^{U_{\alpha(k)}})$. Each of these continua is invariant under the action of $ \mathcal{G}_l$. It can be derived by invoking the relation given in \eqref{347qq}. Then, transforming the solution back to equation \eqref{11}, we thus complete the proof of (ii) in Theorem \ref{th14}.

Lastly, the decay property of non-radial solutions derived by us in Theorem \ref{th14} can be deduced from Proposition \ref{pp41} and Proposition \ref{pro43}. From Proposition \ref{pp41}, we have $\int_{\R^N}|x|^{N\alpha}e^{h}\md x<+\infty$. Then, from \cite{E2} Theorem 1.4, we get $\int_{\R^N}|x|^{N\alpha}e^{h}\md x=N\left(\frac{N^2}{N-1}
		\right)^{N-1}(\alpha+1)^{N-1}\omega_N$. \qed

\section{The classification of solution for the singular case $\al\in(-1,0)$}
We first show that any solution $u$ of \eqref{11} is globally bounded from above.
\begin{thm}\label{th64}
Assume $\al\in(-1,0)$. Let $u \in W_{loc}^{1, N}(\R^N)$ be a weak solution of \eqref{11} with $\int_{\R^N}|x|^{N\alpha}e^{u}\md x<+\infty$. Then $\sup\limits_{x\in\R^N} u<+\infty$.
\end{thm}

To prove Theorem \ref{th64}, we need the following two lemmas.

The first lemma is the Brezis-Merle type exponential estimate for $N$-Laplacian.
\begin{lem}[\cite{E1} Proposition 2.1]\label{th62}
Let $\Omega$ be a bounded domain in $\R^N$ and $f \in L^{1}(\Omega)$. Let $u \in W^{1, N}(\Omega)$ be a weak solution of
\begin{align}\label{f62}
-\Delta_N u = f \quad \text{in } \Omega,
\end{align}
and let $h \in W^{1, N}(\Omega)$ be the weak solution of
$$
\begin{cases}
-\Delta_N h = 0 & \text{in } \Omega, \\
h = u & \text{on } \partial \Omega.
\end{cases}
$$
Then there exists a constant $\mu$, depending only on $N$, such that for every $0 < \lambda < \mu\|f\|^{-\frac1{N-1}}_{L^1(\Omega)}$, it holds
$$
  \int_{\Omega} e^{\lambda |u - h|} \md x \leq \frac{|\Omega|}{1 - \lambda \mu^{-1}\|f\|^{\frac1{N-1}}_{L^1(\Omega)}} .
  $$
\end{lem}

The next lemma is Serrin's local $L^\infty$ estimate from \cite{S1}.
\begin{lem}[\cite{S1} Theorem 2]\label{th63}
Let $u \in W_{loc}^{1, N}(\Omega)$ be a weak solution of \eqref{f62} with $f \in L^{\frac{N}{N-\varepsilon}}(\Omega)$ for some $0<\varepsilon \leq1$. Assume $B_{2 R}\subset\Omega$. Then
$$
\left\| u^{+}\right\| _{L^{\infty}\left(B_{R}\right)} \leq C R^{-1}\left\| u^{+}\right\| _{L^{N}\left(B_{2 R}\right)}+C R^{\frac{\varepsilon }{N-1}}\| f\| _{L^{\frac{N}{N-\varepsilon }}\left(B_{2 R}\right)}^{\frac{1}{N-1}},
$$
where $C=C(N,\varepsilon)$ is a constant.
\end{lem}

\noindent {\bf Proof of Theorem \ref{th64}.}
From Proposition 2.1 in \cite{E2}, we know that there exists $C>0$ such that $u\leq C$ in $B_1(0)\setminus B_\theta(0)$ for $\theta\in(0,1)$ small. Applying Theorem 1.1 in \cite{E2} to the spherical reflection transform $\hat u(x):=u\left(\frac{x}{|x|^{2}}\right)$, which solves $-\Delta_N \hat u=\frac1{|x|^{N(\alpha+2)}}e^{\hat u}$, we deduce that $u\leq C$ in $\R^N\setminus B_1(0)$. Thus we only need to prove there exists $C>0$ such that $u\leq C$ in $B_\theta(0)$ for $0<\theta<1$ small enough.

Consider weak solution $ h \in W^{1,N}(B_\theta(0))$ to the problem
$$
\begin{cases}
\Delta_N h = 0 & \text{in } B_\theta(0), \\
h = u & \text{on } \partial B_\theta(0).
\end{cases}
$$
By the comparison principle, we immediately deduce that $h\leq u$ in $B_\theta(0)$. This directly implies
\begin{align}\label{h63}
\int_{B_\theta(0)}\left(h^{+}\right)^{N} \md x \leq \int_{B_\theta(0)}\left(u^{+}\right)^{N} \md x \leq N! \int_{B_\theta(0)}|x|^{N \alpha} e^{u} \md x.
\end{align}
By applying Serrin's local $L^{\infty}$ regularity estimate (i.e., Lemma \ref{th63}) to $h$, we can derive
\begin{align}\label{h64}
\quad\left\| h^{+}\right\| _{L^{\infty}\left(B_{\theta/2}(0)\right)} \leq C(\theta) .
\end{align}
We now take the radius $\theta$ sufficiently small such that
\begin{align}\label{h65}
\||x|^{N\al}e^u\|_{L^{1}\left(B_{\theta}(0)\right)}^{\frac{1}{N-1}}\leq\frac{\al(1+\al)\mu}{2(\al-1)},
\end{align}
where $\mu$ is given by Lemma \ref{th62}. Using Lemma \ref{th62}, we get
\begin{align}\label{h66}
 \int_{B_{\theta}(0)}e^{\frac{\al-1}{\al(1+\al)}(u-h)}\md x\leq C(\theta).
\end{align}
Subsequently, from \eqref{h64} and \eqref{h66}, we obtain
\begin{align}\label{h67}
 &\quad\int_{B_{\theta/2}(0)}\lr|x|^{N\al}e^{u} \rr^{\frac{N}{N-\var}}\md x=\int_{B_{\theta/2}(0)}|x|^{\frac{N}{N-\var}N\al}e^{\frac{N}{N-\var}(u-h)}e^{\frac{N}{N-\var}h}\md x \\
&\leq  C(\theta)\lr\int_{B_{\theta/2}(0)}|x|^{\frac{N}{N-\var}\frac{2}{1-\al}N\al}\md x \rr^{\frac{1-\al}{2}}\lr\int_{B_{\theta/2}(0)}e^{\frac{\al-1}{\al(1+\al)}(u-h)}\md x\rr^{\frac{1+\al}{2}}\leq C(\theta), \nonumber
\end{align}
where $\var:=\min\{\frac{N(\al+1)}{2(1-\al)},1\}$.
By \eqref{h63}, \eqref{h67} and Lemma \ref{th63}, we have
\begin{align}\label{h68}
 \|u^+\|_{L^{\infty}\left(B_{\theta/4}(0)\right)}\leq C(\theta).
\end{align}
This completes the proof of Theorem \ref{th64}.
\qed
\smallskip

\noindent {\bf Proof of Theorem \ref{thm11}.}
Let $\hat u(x)=u\left(\frac{x}{|x|^2}\right)$, then $\hat u(x)$
satisfies
\begin{equation}\label{24t}
	-\Delta_N \hat u=\frac1{|x|^{N(\alpha+2)}}e^{\hat u}, \quad x\in\R^N\setminus\{0\}.
\end{equation}
By Proposition 2.2 in \cite{E2}, we get $\hat u\in W^{1,q}_{loc}(\R^N)$ for any $1\leq q<N$. So, we have
\begin{align}\label{25t}
	\int_{\R^N\backslash B_1(0)}\frac{|\nabla u|^q}{|x|^{2(N-q)}}\md x<+\infty
\end{align}
for any $1\leq q<N$. Let $\displaystyle t_0:=\sup_{x\in\R^N}u$, from Theorem \ref{th64}, we have $t_0<+\infty$. Through the standard regularity theorem, we can obtain $u\in C^{1,\eta}(\R^N\backslash\{0\})$. Therefore, the set $\Omega_t=\{u>t\}$ is a smooth set, for a.e. $t<t_0$.

The following necessary condition on existence of solutions with finite total mass and quantitative identity of the total mass have been obtained in \cite{E2}.
\begin{lem}[\cite{E2}]\label{th13}
Let $u \in W_{loc}^{1, N}(\R^N)$ be a weak solution of \eqref{11} with $\int_{\R^N}|x|^{N\alpha}e^{u}\md x<+\infty$. Then $\al>-1$ and
	\begin{align}\label{23t}
		\int_{\R^N}|x|^{N\alpha}e^{u}\md x=N(\frac{N^2}{N-1}
		)^{N-1}(\alpha+1)^{N-1}\omega_N.
	\end{align}
\end{lem}

By Lemma \ref{th13}, we can remove the extra condition $\int_{\R^N}|x|^{N\alpha}e^{u}\md x\leq N\left(\frac{N^2}{N-1}\right)^{N-1}(\alpha+1)^{N-1}\omega_N$ in \cite{N}. Since we have proved $\sup\limits_{x\in\R^N} u<+\infty$ in Theorem \ref{th64},  the classification results in Theorem \ref{thm11} follows immediately from Theorem 6.1 in \cite{N}. This finishes the proof of Theorem \ref{thm11}.   \qed

\end{document}